\documentclass[11pt]{article}

\usepackage{amsthm}
\usepackage{amsmath}
\usepackage{amsfonts, epsfig, amsmath, amssymb, color, amscd}
\usepackage[cmtip,arrow,all]{xy}
\usepackage{pb-diagram, pb-xy}
\usepackage{amssymb,epsfig}
\usepackage{color}
\usepackage[cmtip,arrow]{xy}
\usepackage{pb-diagram, pb-xy}
\usepackage{amssymb,epsfig,amsfonts}

\usepackage[linktocpage=true,plainpages=false,pdfpagelabels=false]{hyperref}

\usepackage{tikz}
\usepackage{eufrak}

\newfont{\sdbl}{msbm9}
\newfont{\dbl}{msbm10 at 12pt}
\theoremstyle{definition}

\newcommand{\dz}{{\mbox{\dbl Z}}}

\newcommand{\dn}{{\mbox{\dbl N}}}

\newcommand{\dc}{{\mbox{\dbl C}}}
\newcommand{\dq}{{\mbox{\dbl Q}}}

\newcommand{\ord}{\mathop{\rm ord}\nolimits}

\newcommand{\End}{\mathop {\rm End}}

\newcommand{\Aut}{\mathop {\rm Aut}}

\newcommand{\Sup}{\mathop {\rm Supp}}

\newcommand{\rk}{\mathop {\rm rk}}

\newcommand{\idm}{\mathfrak{m}}
\newcommand{\Ord}{\mathop {\rm \bf ord}}

\makeatother
\newtheorem{Def}{Definition}[section]
\newtheorem{rem}{Remark}[section]

\newtheorem{ex}{Example}[section]

\theoremstyle{plain}
\newtheorem{Prop}{Proposition}[section]
\newtheorem{theorem}{Theorem}[section]
\newtheorem{lemma}{Lemma}[section]
\newtheorem{cor}{Corollary}[section]

\numberwithin{equation}{section}

\newcommand{\co}{{{\cal O}}}
\newcommand{\crr}{{{\cal R}}}
\newcommand{\cf}{{{\cal F}}}

\newcommand{\cm}{{{\cal M}}}

\begin{document} 
\title{The Conjecture of Dixmier for the first Weyl algebra is true}
\author{Alexander Zheglov}
\date{}
\maketitle


\begin{abstract}
Let $K$ be a field of characteristic zero, let $A_1=K[x][\partial ]$ be the first Weyl algebra. In this paper we prove that the Dixmier conjecture for the first Weyl algebra is true, i.e. each
algebra endomorphism of the algebra $A_1$ is an automorphism. 
\end{abstract}

\tableofcontents

\section{Introduction}
\label{S:introduction}

In 1968, Jacques Dixmier \cite{Dixmier} posed six problems for the Weyl algebra $A_{1}=K [x][\partial ]$, where $K$ is a field of characterictic zero\footnote{In Dixmier's paper this algebra was defined as  $A_{1}=K[p,q]$ with the relation $[p,q]=1$. We prefer to present the Weyl algebras as rings of differential operators with polynomial coefficients, as it is essential in our paper. When comparing, we should take $\partial$ instead of $p$, and $x$ instead of $q$. Though Dixmier used another form of elements (he put $p$ on the left of $q$), the difference is not essential for applications of his results in our paper.}. 

The first problem from this list is the most famous of them,  and is known as the Dixmier conjecture for the first Weyl algebra. The Dixmier conjecture claims that any non-zero endomorphism of $A_{1}=K [x][\partial ]$ is  an automorphism. It's easy to see that it's sufficient to prove this only for $K=\dc$. 

Analogous conjectures exist for Weyl algebras $A_n=K[x_1,\ldots ,x_n][\partial_1, \ldots ,\partial_n]$ of any dimension\footnote{These conjectures were not formulated in the original paper by Dixmier \cite{Dixmier}, and are known as a folklore for a long time. Perhaps the first mention of the connection between these conjectures and the Jacobian conjectures can be found in a paper \cite{BCR}, where the authors refer to communications with L. Vaserstein and V. Kac, see p.297 of loc. cit.} 
and are stably equivalent to the famous Jacobian conjectures, namely, the Dixmier conjecture for $A_n$ ($DC_n$ for short) implies the Jacobian conjecture for $K[x_1,\ldots ,x_n]$ ($JC_n$ for short), and the Jacobian conjecture for $K[x_1,\ldots ,x_{2n}]$ implies the Dixmier conjecture for $A_n$ (all these conjectures are still open in any dimension). These equivalences were obtained by  Yoshifumi Tsuchimoto \cite{Ts1}, \cite{Ts2}, and independently by Alexei Belov-Kanel and Maxim Kontsevich  \cite{BK}, see also the paper by Vladimir Bavula \cite{Bavula2}. As one can see, the original Dixmier conjecture for the first Weyl algebra  occupies a special place among other conjectures; in particular, it is unknown whether any of the Jacobian conjectures can be derived from it.  


There are really many papers around the Dixmier problems, and it is difficult to mention them all. Let me try to mention the most interesting and the most recent  papers relevant to the Dixmier conjecture. There are many interesting papers by Vladimir Bavula devoted to the problems of Dixmier. In particular, in \cite{Bavula} he proved that the Dixmier conjecture is true for the algebra of polynomial integro-differential operators; in \cite{Bavula3} he proved the fifth problem.  
The history of other Dixmier problems and further references can be found also in \cite{Bavula}. There is  one important paper by Joseph \cite{Joseph}, where the third and sixth problems were solved; besides this paper contains important results about Newton polygons used later in many papers relevant to the Dixmier or Jacobian conjectures. At last, the most recent papers devoted to the Dixmier conjecture, are \cite{Han}, \cite{GGV3}. Further relevant references can be found in these papers.

The goal of this paper is to prove that the Dixmier conjecture for the first Weyl algebra is true: 

\begin{theorem}
\label{T:main}
The Dixmier conjecture for the first Weyl algebra is true, i.e. $\End_K (A_1)=\Aut_K (A_1)$.
\end{theorem}

The proof is based on the generalised Schur theory offered in the paper \cite{A.Z} and its extension for ordinary differential operators -- the theory of normal forms, developed in a recent paper \cite{GZ1} (we recall all necessary results of this theory in section \ref{S:prelim}), as well as on some known results about the shape of possible counterexamples to the Dixmier conjecture from the work \cite{GGV} (one approach, used in \cite{GGV}, to try to solve the conjecture for $A_1$ is the  "minimal" counterexample strategy - assume that the Dixmier conjecture is false and find properties that a  counterexample must satisfy). 
Surprisingly, the proof we offer in this paper was  obtained as a byproduct of the theory of normal forms for commuting ordinary differential operators, which was originally developed for other aims, the well known classification theory of commutative subrings of ordinary differential operators,
and of some observations (see lemma \ref{L:9}) obtained in \cite{GZ}  in connection with the Berest conjecture (however, most of the results around the Berest conjecture, discussed first in papers \cite{MZh}, \cite{SMZh} and which remains interesting in itself, are not used in the proof). It seems to us that it can serve as evidence of the existence of a deep connection between the Dixmier and Jacobian  conjectures and the theory of commuting differential operators. We hope that further development of the normal form technique together with effective methods of positive characteristic from papers \cite{BK}, \cite{BK1}, \cite{Ts1}, \cite{Ts2}, used in proving the equivalence of the Dixmier conjectures and the Jacobian conjectures, will allow people to come closer to solving at least the $JC_2$ conjecture.

\smallskip

The scheme of the proof is as follows. Each endomorphism of the first Weyl algebra is determined by a pair of operators $Q,P\in A_1$ of orders $q,p$ with $[Q,P]=1$. The known results about the shape of possible counterexamples claim that if there is a counterexample $Q,P$ to the Dixmier conjecture (which we'll call, in spirit of \cite{GGV}, as a DC-pair), then there are (many) counterexamples of special (subrectangular) type (see section \ref{S:DC-pairs}).

A natural question is to study  counterexamples $Q,P$ of subrectangular type with the help of the Schur-Sato theory from \cite{A.Z} and the theory of normal forms from \cite{GZ1}. Let us briefly recall the main points of these theories. 

First, it appears to be very useful to consider the first Weyl algebra $A_1$ as a subring of a more general ring of ordinary differential operators $D_1=K[[x]][\partial ]$, and the  ring 
$D_1$ -- as a subring of a more general ring $\hat{D}_1^{sym}$ (appeared first in \cite{BurbanZheglov2017} and then extensively studied in \cite{A.Z} and \cite{GZ1}). 
According to a generalised Schur theorem \cite[Th.7.1]{A.Z} there is an invertible operator 
$S\in \hat{D}_1^{sym}$ of order zero (a Schur operator) such that $SPS^{-1}=\partial^{\Ord (P)}$ (without loss of generality the operator $P$ can be assumed to have invertible highest coefficient). The operator $S$ is determined up to multiplication with an operator from the centralizer $C(\partial^{\Ord (P)})\subset \hat{D}_1^{sym}$, and, according to a centralizer theorem \cite[Prop.7.1]{A.Z} the centraliser $C(\partial^{\Ord (P)})$ consists of polynomials with constant coefficients in the differentiation, integration and shift operators, which have a restricted finite order in each variable. 

It is well known that any ordinary differential operator with invertible highest coefficient can be normalised by means of some change of variables and conjugation by invertible function (see section \ref{S:prelim}). In \cite{GZ1} we studied in details the shape of the operator $S$ for a {\it normalised} operator $P$ and the structure of the centralizer $C(\partial^k) \subset \hat{D}_1^{sym}$. We proved that $S$ satisfies some specific properties (condition $A_p(0)$ -- this is a condition on homogeneous components of $S$), and can be normalised by assuming its first homogeneous components to be $S_0=1$, $S_{-1}=0$.  We also proved that the centraliser $C(\partial^k)$ is isomorphic to the matrix ring $M_k(K[D^k])$ (here $D$ is a formal variable), and all operators from $C(\partial^k)$ have a presentation in a so called vector form (see the list notations), which is very convenient for calculations. 

For a pair of operators $P,Q$ we introduced in \cite{GZ1} the notion of a {\it normal form} of $Q$ with respect to $P$: this is the operator $\tilde{Q}=SQS^{-1}$, where $S$ is a Schur operator for $P$ (sometimes we call the pair $(\partial^{\Ord (P)}, \tilde{Q})$ as a normal form for the pair $P,Q$). 

The first natural question about any DC-pair $P,Q$ (which can be made monic after applying  some simple  automorphism) is the following: what is the shape of a normal form of $Q$ with respect to $P$? We proved (step 1 of the proof of theorem \ref{T:DC-pairs} and lemma \ref{L:Q_p}) that this normal form has a very simple structure -- it is a sum of an operator from the centraliser $C(\partial^{\Ord (P)})\subset \hat{D}_1^{sym}$ of order $q=\Ord (Q)$ and a homogeneous operator of order $-p:=-\Ord (P)$, which is uniquely determined and whose shape can be explicitly calculated (let's call it the "tail"):
$$
\tilde{Q}= \partial^q + \tilde{Q}_{q-1}+\ldots + \tilde{Q}_{-p}. 
$$
The second natural question is about a sequence of operators constructed recursively starting with the pair $(\partial^{\Ord (P)}, \tilde{Q})$ as a result of a procedure which we'll call as a "head-chopping process". To explain this process, it is convenient to introduce the following "zoological" terminology. Let's call the sum of homogeneous components $\partial^q + \tilde{Q}_{q-1}+\ldots + \tilde{Q}_{-p+1}$ of the operator $\tilde{Q}$ belonging to the centralizer as a "body", and the highest order homogeneous component ($\partial^q$) as the "head". Let's consider the operator
$$
\tilde{Q}_1:= \tilde{Q}^p-\partial^{pq}. 
$$
The power of $\tilde{Q}$ can be again decomposed as a sum of a "body" from the centraliser $C(\partial^p)$ and a "tail", which is, by definition, a sum of products of the "body" $(\tilde{Q}-\tilde{Q}_{-p})$ of $\tilde{Q}$ and the "tail" $\tilde{Q}_{-p}$ (appearing in the expansion of $\tilde{Q}^p$). The highest order homogeneous component of $\tilde{Q}^p$, $\partial^{pq}$, is the "head". After  subtraction of the operator $\partial^{pq}$ (cutting the "head") we get an operator with a smaller "body", the same "tail" and a new "head". If the order of the new "head" is greater than the order of the beginning of the "tail" (i.e. the highest order of homogeneous components of the "tail"), it can be shown that this new head has the shape $c\partial^k$, $c\in K$, and we can continue the process by considering the operator 
$$
\tilde{Q}_2:= \tilde{Q}_1^{p}-c^{n_1}\partial^{pm_1} 
$$
with appropriate $n_1, m_1$ such that the head of $\tilde{Q}_1^{n_1}$ is cutted after subtraction of the operator $\partial^{pm_1}$. Here the operator $\tilde{Q}_1^{n_1}$ has analogous "body-tail" decomposition: its "body" is a power of the body of $\tilde{Q}_1$, and the "tail" is a sum of products of the "body" of $\tilde{Q}_1$ and its "tail" (appearing in the expansion of $\tilde{Q}_1^{n_1}$).

So, for given operators $\partial^{p}, \tilde{Q}\in \hat{D}_1^{sym}$ we get the following recurrent sequence of polynomials:  $F_{i+1}(X,Y)= F_{i}(X,Y)^{n_i}-\varepsilon_i^{p}X^{m_i}$, where $F_0(X,Y):=Y$, $m_0=q$, and numbers  $\varepsilon_i$ and $m_i$ are defined in such a way that 
$$
\Ord (F_{i+1}(\partial^{p}, \tilde{Q}))<\Ord (F_{i}(\partial^{p}, \tilde{Q})^{n_i}).
$$
And we also get a sequence of operators: $\tilde{Q}_i=F_i(\partial^{p}, \tilde{Q})$. 

Let's say the operators $\partial^{p}, \tilde{Q}$ {\it satisfy a termination condition for the head-chopping process}, if the sequence of operators $\tilde{Q}_i$ satisfies the following condition:
 the sequence of numbers 
$$
SeqGCD_i(\partial^p, \tilde{Q})=\{GCD\{p, \Ord (\tilde{Q}_0), \ldots , \Ord (\tilde{Q}_i)\} \}
$$
ends by 1. Given such a sequence we can define its {\it distance} -- the sum of differences (distances) of orders of the head of $\tilde{Q}_i^{n_i}$ and of the head of $\tilde{Q}_{i+1}$. By definition, it can not be grater than $p+q-1$. 

Going back to differential operators from the Weyl algebra, we can formulate similar condition for the original pair of operators $P,Q$ (see formulation of theorem \ref{T:DC-pairs}). 
The natural question is: does a DC-pair of subrectangular type satisfy the termination condition for the head-chopping process, and if yes, what is the distance?

The first preparatory theorem \ref{T:DC-pairs} gives the affirmative answer on this question and claims that distance is maximal possible, i.e. equal to $p+q-1$. The second preparatory theorem \ref{T:subrectangular_sequence} together with remark \ref{R:DC-pais} clarify the shape of the operators $F_i(P,Q)$ from the sequence of the head-chopping process -- they are all of subrectangular type -- and also they clarify the information about the edges of rectangulars for the first and the last operators in the sequence, i.e. about the polynomials $f_{0,1}(P), f_{0,1}(Q), f_{1,0}(P_I), f_{1,0}(Q_I)$. This information is important for calculations in the last step of the proof of the Dixmier conjecture. 

The subrectangular form of the Newton polygons of such DC pairs appears to be the key property for the proof of theorems \ref{T:DC-pairs}, \ref{T:subrectangular_sequence}, and is used in the key lemma \ref{L:polynomials}. For DC-pairs of non-subrectangular type of general form, the proof of the preparatory theorem \ref{R:DC-pais} generally fails, although in some cases the result remains true. For example, theorem \ref{R:DC-pais} remains true in the case of a pair of operators defining an automorphism of the Weyl algebra (cf. remark \ref{R:spectral_data_for_string}), but the proof in this case must use Dixmier's theorem on the structure of the automorphism group $\Aut_K (A_1)$. However, in this case the distance is not maximal (and it is far from to be maximal). 

A rigorous proof of the preparatory theorems requires verifying a number of technical conditions at each step of the inductive argument. These arguments are  often similar or even repetitive, and could be omitted in some cases. Nevertheless, we have decided to describe them all in detail to avoid errors and facilitate verification.

\smallskip

In the next step of the proof of the Dixmier conjecture we investigate more general solutions to the equation $[Q,P]=1$ (which is called the string equation), when $Q,P$ not necessarily  belong to the Weyl algebra, but may be any ordinary differential operators. These solutions are of independent interest (from the point of view of integrable systems or string theory). This equation was studied earlier (rather, from the point of view of physical applications) by many authors, see e.g. works \cite{Brez}, \cite{Dougl1}, \cite{Dougl2}, \cite{Gross}, \cite{GN1}, \cite{GN2} and references  therein. The approach to non-critical string theory based on the considerations of matrix models in \cite{Brez}, \cite{Dougl1}, \cite{Gross} led to the problem of description of all pairs $P,Q$ satisfying $[Q,P]=1$ (see \cite{Dougl2}), what motivated the study of such operators in Albert Schwarz's work \cite{Schwarz1}. In this paper he established a one-to-one correspondence between the set of solutions $P,Q$ to the string equation of {\it coprime orders} and the set of pairs of commuting differential operators of the same orders. 

This description appears to be very similar to our construction: we show that solutions $P,Q$ to the string equation (of arbitrary order) satisfying the termination condition for the head-chopping process are in one-to-one correspondence with pairs $P',Q'$ of commuting differential operators of the same orders satisfying analogous condition and such that $\rk K[P',Q']=1$\footnote{It is well known that any commutative ring of ordinary differential operators of rank one contains a pair of operators of coprime orders, and these operators generate a subring of rank one "almost" coinciding with the original ring. However, a pair of operators of not coprime orders may generate a ring of of rank one as well.} (see section \ref{S:string}). By definition, the termination condition for the head-chopping process is automatically satisfied for the pair $P,Q$ of coprime orders (and in this case our construction, as it turns out, coincides with the A. Schwarz construction). However, for the proof of the Dixmier conjecture we need to investigate solutions to the string equation of {\it non-coprime} orders, because DC-pairs of subrectangular type are known to be of  non-coprime orders.\footnote{There is an interesting relevant question about a connection of different definitions of {\it quantum curves} appearing in the literature. A. Schwarz in \cite{Schwarz2} defined a quantum curve as a solution to the string equation. On the other hand, various authors defined a quantum curve as a D-module (or a family of D-modules) like e.g. in \cite{DHS}, \cite{MulaseDumitrescu}, see also references therein. A. Schwarz in \cite{Schwarz2} established a correspondence between the $D$-module version of definition and  definition (quantum curve = solution to the string equation), but apparently this correspondence is established only for some specific $D$-modules (i.e. corresponding to solutions with coprime orders). It would be interesting  to extend it  also for solutions of arbitrary orders, cf. also remark \ref{R:spectral_data_for_string}.} 

Our investigations were motivated by the following natural question: how to classify, if possible, all  solutions to the string equation, and how  they are controlled by their normal forms. In more details, the one-to-one correspondence mentioned above can be described by the following diagram:
\begin{equation}
\label{E:diagram}
\begin{array}{c}
\xymatrix{
(P,Q) \ar[rr]_-{\mbox{$Ad(S)$}} & &(\partial^p, \tilde{Q})\ar@/_10pt/[ll]_-{\mbox{$Ad(S^{-1})$}} \ar[rr]_-{\mbox{cutting the tail}} & & (\partial^p, \tilde{Q}')\ar@/_10pt/[ll]_-{\mbox{glueing the tail}} \ar[rr]_-{\mbox{$Ad((S')^{-1})$}} & & (P', Q')\ar@/_10pt/[ll]_-{\mbox{$Ad(S')$}} \\
}
\end{array}
\end{equation}
where $(P,Q)\in D_1$ is a {\it normalized } solution to the string equation, satisfying the  termination condition for the head-chopping process,  $(P',Q')\in D_1$ are commuting {\it normalized} operators of the same orders, satisfying analogous termination condition for the head-chopping process and such that $K[P', Q']\subset D_1$ is a  subring of rank one\footnote{It is sufficient (and more convenient) to establish a one-to-one correspondence between {\it normalized} operators}, and $S,S'$ are Schur operators for $P,P'$ correspondingly. In one way this correspondence is described by theorem \ref{T:string_equation} (from $P,Q$ to $P',Q'$ via cutting the tail), and in another way by theorem \ref{T:inverse_string_equation} (from $P',Q'$ to $P,Q$ via glueing the tail). The operators $S,S'$ are uniquely defined if we fix some normalisations of the normal string pair $(\partial^p, \tilde{Q})$ and the commutative normal pair $(\partial^p, \tilde{Q}')$ correspondingly (theorems \ref{T:one-to-one-correspondence}, \ref{T:one-to-one-correspondence1}), however the the composition of maps in the diagram above does not depend on the choice of these normalisations. 

The the one-to-one correspondences in the diagram above are based not only on the generalized Schur-Sato theory and the theory of normal forms, but also on the well known classification theory of commuting ordinary differential operators. This theory can be illustrated in short by the following commutative diagram (see also theorems \ref{T:classif2}, \ref{T:classif_Schur_pairs} in section  \ref{S:prelim} and explanations there):
\begin{equation}\label{E:classif_diagram}
\begin{array}{c}
\xymatrix{
[B\subset D]  \ar[rr]_-{\mbox{direct construction}}  \ar@/^10pt/[rd]_-{\mbox{$S_{cl}BS_{cl}^{-1}$}} & & \mbox{[Geometric data]} \ar@/_10pt/[ll]_-{\mbox{inverse construction}} \ar@/_10pt/[ld]^-{\mbox{Krichever map}}\\
& \mbox{[Schur pairs $(A,W)$]} \ar@/_20pt/[ru]_-{\mbox{direct construction}} \ar@/^20pt/[lu]^-{\mbox{$S_{cl}^{-1}AS_{cl}$}} & \\
}
\end{array}
\end{equation}
For our proof the most important part of this classification diagram is the left part, i.e. the classification of commutative rings of operators $B\subset D_1$ in terms of Schur pairs. This part of classification is described by the classical Schur-Sato theory, where the classical Schur-Sato  operators $S_{cl}$ play a crucial role. It is well known that for any monic normalised ordinary differential operator $P\in D_1$ there is a monic zeroth order invertible pseudo-differential operator $S_{cl}\in E$ such that $S_{cl}PS_{cl}^{-1}=\partial^{\ord (P)}$. This operator is defined up to multiplication with a pseudo-differential operator with constant coefficients, so the {\it normalised} operator, i.e. such that $S_{cl}(0)=1$, is uniquely defined. Analogously, by the Sato theorem, for any Schur space $W$ there is a uniquely defined monic zeroth order invertible pseudo-differential Sato operator $S_{cl}\in E$ such that $W=F\cdot S_{cl}$, where $F=K[\partial ]$ (for details see the list of notations or section \ref{S:prelim}). In each equivalence class of Schur pairs there is a {\it normalised} space  $W$ (i.e. such that $1\in W$), and the corresponding Sato operator is also normalised. So, essentially, Schur and Sato operators are the same, and we'll call them as the classical Schur-Sato operators. 

Although the rest of the classification diagram does not play so important role for our paper, for the sake of completeness we establish a one-to-one correspondence not only between objects from diagram \ref{E:diagram}, but also between them and corresponding spectral data (corollary \ref{C:one-to-one2}). 

The generalized Schur theory explains  relations between different versions of Schur operators: in \cite{GZ1} we showed that the subring $D_1$,  Schur operators for differential operator of fixed order $p$ and centralizers $C(\partial^p)$ are contained in a smaller subring $ \widehat{Hcpc}_{B}(p)\subset \hat{D}_1^{sym}$, which can be embedded into the subring $ E_p$ of skew pseudo-differential operators:
$$
 E_p=K^{\oplus p}[\Gamma_1]((\tilde{D}^{-1})) 
$$
with the commutation relations 
$$
\tilde{D}^{-1}a=\sigma (a)\tilde{D}^{-1}, \quad a\in K^{\oplus p}[\Gamma_1] \quad \mbox{where \quad }
$$
$$
\sigma (a_0, \ldots ,a_{k-1})=(a_{k-1}, a_0, \ldots , a_{k-2}), \quad \sigma (\Gamma_1)=\Gamma_1-1 
$$
(see section \ref{S:prelim}). The classical ring of pseudo-differential operators $E=K[[x]]((\partial^{-1}))$ can also be embedded into $E_p$. 

There is the following relation between operators $S_{cl}$ and $S$ embedded into $E_p$ (see Step 2 in the proof of theorem \ref{T:one-to-one-correspondence})
\begin{equation}
\label{E:Schur_oper_relation}
 \tilde{S}:= S S_{cl}\in  K^{\oplus p}((\tilde{D}^{-1})) \subset E_p.
\end{equation}
The auxiliary operator $\tilde{S}$ (see lemma \ref{L:invariant_conjugation} and Step 1 in the proof of theorem \ref{T:one-to-one-correspondence}) sends the normal form $\tilde{Q}$ of the operator $Q$ with respect to $P$ (and the normal form $\tilde{B}$ of a commutative ring $B$ generated by $P,Q$) to operators with constant coefficients:
$$
\tilde{S}^{-1}\tilde{B}\tilde{S}\in K((\tilde{D}^{-1})),
$$
$S$ is a Schur operator for $P$ and $S_{cl}$ is the classical Sato operator for the space $W:=F'\cdot\tilde{S}$. 
Moreover,  the operator $\tilde{S}$ is normalised (in the sense of Step 1 in the proof of theorem \ref{T:one-to-one-correspondence}) iff the operator $S_{cl}$ is normalised, and in this case the operators in relation \eqref{E:Schur_oper_relation} are uniquely determined. Relation \eqref{E:Schur_oper_relation} plays a major role in establishing one-to-one correspondences between normalised solutions to the string equation and  normalised normal string pairs (correspondingly, normalised commuting operators and normalised commutative normal pairs) in diagram \eqref{E:diagram} (these are the left and right parts of \eqref{E:diagram}  correspondingly).

\smallskip

The last step of the proof of the Dixmier conjecture is one explicit calculation which leads to a contradiction. This last step is described in section \ref{S:proof}. This section begins with several preliminary observations. First we note that if  there is a DC pair $P,Q$ of subrectangular type, then we can  multiply DC pairs by constructing several new families of such examples (remark \ref{R:another_subrectangular}, lemma \ref{L:Phi_N(P)}).

Namely, assuming that there is a DC pair $P,Q$ of subrectangular type we construct in subsection \ref{S:prelim_lemmas} at the first step another DC-pair 
$P_n, Q_n$ of subrectangular type with additional  properties: the right vertical sides of their rectangulars are monomials, and the upper horizontal sides have specific coefficients preceding the highest coefficients (see formulae \eqref{E:P_n,Q_n}). 

After that we construct a family of  DC-pairs $\Phi_{N,1}(P_n), \Phi_{N,1}(Q_n)$ by applying additional automorphism $\Phi_{N,1}$ (with appropriate $N>1$) to this pair. We show (lemma \ref{L:Phi_N(P)}) that  there are two cases, depending on the original DC-pair (and not on $N$), such that in the first case the new DC pair has a sequence (of polynomials $F_i$, operators $Q_{N,i}$ etc.) of length one more than the length of the sequence  for the original DC pair, and in the second case of the same length. Moreover, in both cases the sequences of polynomials $F_i$ are the same for $i\le I$. In both cases we calculate sequential GCDs, distances and the orders of the last operators in the sequences. It appears that in both cases the distances are almost maximal possible and in both cases this DC-pair satisfy the termination condition for the head-chopping process. 

Lemma \ref{L:Phi_N(P)} relies heavily on the preparatory theorems \ref{T:DC-pairs}, \ref{T:subrectangular_sequence}. One of the most important conclusions of this lemma is that the sequences for the original operators $P,Q$ and the new ones $\Phi_{N,1}(P_n), \Phi_{N,1}(Q_n)$ are similar (for all $N$) in the following sense: if $(\ord (Q_i),\varepsilon_i)$ is the sequence for $P,Q$, then  $(\ord (Q_{N,i}),\varepsilon_i)$ for $i\le I$ is a part of the sequence for $\Phi_{N,1}(P_n)$, $\Phi_{N,1}(Q_n)$ (it is the whole sequence in the second case, and the sequence without the last element in the first case), where we moreover have the following equalities:
$$
\ord (Q_{N,i})=(d_2N+l_2)\ord (Q_i)/d_2 \quad \mbox{for $i<I$};
$$
besides the distance between the "head" of the $I$th operator $\tilde{Q}_{N,I}$ of the sequence and the beginning of its "tail" is relatively small (it is equal to $N+1$). This conclusion is used in the final step of the proof (see below). 

After that in section \ref{S:Schur_adj_DC_pairs} we prove theorem \ref{T:main} with the help of one explicit calculation. Namely, according to diagram \eqref{E:diagram} (see corollary \ref{C:one-to-one2})  the pair of commuting operators $P_{N,n}', Q_{N,n}'$ corresponding to the pair $\Phi_{N,1}(P_n), \Phi_{N,1}(Q_n)$ satisfy an algebraic relation given by some polynomial $f(X,Y)$ of BC-type (see definition \ref{D:BC-distance} and remark \ref{R:classif_ODO}). Taking the normal form $\partial^{\ord P_{N,n}'}, \tilde{Q}_{N,n}'$ of the pair $P_{N,n}', Q_{N,n}'$ and conjugating the operator $\hat{\Phi}(\tilde{Q}_{N,n}')$ by the auxiliary operator $\tilde{S}$, we get a Puiseux series $X_{N,n}$ which gives a solution of the equation $f(\tilde{D}^{\ord P_{N,n}'}, X)=0$. 

This series has a {\it distinguished coefficient} at $\tilde{D}^{-\ord P_{N,n}'}$ which is equal to zero (see  remark \ref{R:distinguished_coefficient}; this is an important property which holds for  series obtained from commuting operators generating the ring of {\it rank one} due to the Newton-Puiseux theorem combined with the Wilson theorem \ref{T:Wilson}). Because of this property it is possible to find distinguished coefficients of the sequence $X_{N,i}=F_i(\tilde{D}^{\ord P_{N,n}'}, X_{N,n})$ -- these are coefficients of homogeneous terms corresponding to the beginning of the "tail" of operators $\tilde{Q}_{N,i}$. 

We get a contradiction by showing that, on the one hand side, these distinguished coefficients don't depend on $N$ (lemma \ref{L:distinguished coefficients}), and on the other hand side by calculating the distinguished coefficient of the series $X_{N,I}$ in Steps 4-7 of section \ref{S:Schur_adj_DC_pairs}  (this calculation shows that it depends on $N$). 

To show that the distinguished coefficients don't depend on $N$ we use the similarity of sequences for the original operators $P,Q$ and the new operators $\Phi_{N,1}(P_n), \Phi_{N,1}(Q_n)$ from lemma \ref{L:Phi_N(P)}. Translating this property to the Puiseux series $X_{N,n}$ we get that the beginnings of these series, almost up to the  distinguished homogeneous component, are similar (for all $N$) in the following sense: $X_{N_1,n}=X_{N_2,n}|_{\tilde{D}\mapsto \tilde{D}^{(d_2N_1+l_2)/(d_2N_2+l_2)}}$, and by that reason the expressions for the  distinguished coefficients of the sequences $X_{N,i}$ must be the same. 

To provide the calculation of the distinguished coefficient of the series $X_{N,I}$, we calculate first the corresponding distinguished homogeneous component  of the operator $\tilde{Q}_{N,I}$. To calculate this component it is enough to calculate just one homogeneous component of the Schur operator $S$ for the operator $\Phi_{N,1}(P_n)$ and then calculate the result of conjugation of $Q_{N,I}$ by $S$. The main reason why it is sufficient to calculate only one homogeneous component of $S$ is the small distance between the "head" of the series $X_{N,I}$ and the beginning of its tail, which follows from the conclusion of  lemma \ref{L:Phi_N(P)} mentioned above. This calculation is done in Step 4. 

After that in Steps 5 and 6 we calculate the body-tail decomposition for this distinguished homogeneous component (this component is the beginning of the tail, so again, as in Step 4, the calculation is not difficult), and finally calculate the result of conjugation by the auxiliary operator $\tilde{S}$ in Step 7. In fact, we not even need to calculate $\tilde{S}$, because the result of conjugation for this component is just the averaging of components of its vector form.

\smallskip

One corollary of the one-to-one correspondences from diagram \eqref{E:diagram} (not necessary for the proof of the conjecture, but of independent interest) is a relation between commutative normalised pairs of operators corresponding to {\it adjoint} solutions of the string equation. The relation between commutative pairs of operators and their adjoint pairs (i.e. pairs obtained after applying the anti-involution $*:D_1\rightarrow D_1$, $*(x):=x$, $*(\partial ):=-\partial$) is well known in the theory of commuting operators (see remark \ref{R:adjoint}). In particular, if $S_{cl}$ is a classical Schur-Sato operator for an operator $P'$, then the operator  
\begin{equation}
\label{E:Schur_adjoint}
S_{cl}^*:=*S_{cl}^{-1}
\end{equation}
will be a corresponding classical Schur-Sato operator $S_{cl}^*$ for the operator $*(P')$. So, it is not difficult to find a relation between {\it normalised} Schur-Sato operators. In view of diagram \eqref{E:diagram} the following natural question appears: if we take a normalised solution $(P,Q)$ to the string equation and the adjoint solution $(*(P), -*(Q))$, 
what is the relation between the corresponding commutative normalised pairs of operators? We show in proposition \ref{P:adjoint} that the relation is given by the following commutative diagram:
\begin{equation}
\label{E:*diagram}
\begin{array}{c}
\xymatrix{
(P,Q) \ar[rr] \ar[d]_-{\mbox{$*$}} & & (P', Q')\ar@/_10pt/[ll] \ar[d]_-{\mbox{$*$}} \\
(*P,-*Q) \ar[rr] \ar[u]_-{\mbox{$*$}} & & (*P', -*Q')\ar@/_10pt/[ll] \ar[u]_-{\mbox{$*$}}
}
\end{array}
\end{equation}

\smallskip

For convenience of the reader, trying to make the paper  selfcontained, we recall all necessary important results  and definitions from relevant papers in section \ref{S:prelim}. In section \ref{S:DC-pairs} we prove several preliminary technical statements about possible counterexamples to the Dixmier conjecture (which we call, inspired by \cite{GGV}, as the DC-pairs). For these statements we use one lemma from \cite{GZ}. 
In section \ref{S:preparatory} we prove the main preparatory theorems -- this is one of the most difficult parts; the proof extensively uses the generalised Schur theory and the theory of normal forms from \cite{GZ1}. In section \ref{S:string} we establish a one to one correspondence between solutions of the string equation satisfying certain conditions and commuting ordinary differential operators of rank one. This section also extensively uses the generalised Schur theory and the theory of normal forms. 
In the last section we prove theorem \ref{T:main}. 

{\bf Acknowledgements.}  We are grateful to the Sino-Russian Mathematics Center at Peking University for hospitality and excellent working conditions while preparing this paper.

We are grateful to Leonid Makar-Limanov and to Christian Valqui for pointing out useful references and fruitful discussions. 
We are  grateful to Junhu Guo for many stimulating discussions. 

We are also grateful to the anonymous referees  for pointing out the papers by Albert Schwarz, which allowed us to improve the exposition of our results.

\subsection{List of notations}
\label{S:list}

 Since this work uses quite different techniques, for convenience of the reader we introduce now the most important notations used in this paper.
 
1.  $\dz_+$ is the set of all non-negative integers, $\dn$ is the set of natural numbers (all positive integers). $K$ is a field of characteristic zero. Recall some notation from \cite{A.Z}:
	$\hat{R}:=K [[x_1,\ldots ,x_n]]$, the $K$-vector space 
$$
\cm_n := \hat{R} [[\partial_1, \dots, \partial_n]] = \left\{
\sum\limits_{\underline{k} \ge \underline{0}} a_{\underline{k}} \underline{\partial}^{\underline{k}} \; \left|\;  a_{\underline{k}} \in \hat{R} \right. \;\mbox{for all}\;  \underline{k} \in \dn_0^n
\right\},
$$
$\upsilon:\hat{R}\rightarrow \dn_0\cup \infty$ --  the discrete valuation defined by the unique maximal ideal $\idm = (x_1, \dots, x_n)$ of $\hat{R}$,  \\
for any element
$
0\neq P := \sum\limits_{\underline{k} \ge \underline{0}} a_{\underline{k}} \underline{\partial}^{\underline{k}} \in \cm_n
$
$$
\Ord (P) := \sup\bigl\{|\underline{k}| - \upsilon(a_{\underline{k}}) \; \big|\; \underline{k} \in \dn_0^n \bigr\} \in \dz \cup \{\infty \},
$$
$$
\hat{D}_n^{sym}:=\bigl\{Q \in \cm_n \,\big|\, \Ord (Q) < \infty \bigr\};
$$
$
P_m:= \sum\limits_{ |\underline{i}| - |\underline{k}| = m} \alpha_{\underline{k}, \underline{i}} \,  \underline{x}^{\underline{i}} \underline{\partial}^{\underline{k}}
$ -- the $m$-th \emph{homogeneous component} of $P$,\\
$\sigma (P):=P_{\Ord (P)}$ -- the highest symbol.
	
2. In this paper we use: 	
	$\hat{R}:=K[[x]]$, $D_1:=\hat{R}[\partial]$, 
$$\hat{D}_1^{sym}:=\{Q=\sum_{k\ge 0}a_k\partial^k|\Ord(Q)<\infty\}.$$ 

Operators:  $\delta:=\exp((-x)\ast \partial)$\footnote{ Here and further $\ast$ in all exponentials means that we consider normalized Taylor power series, i.e. the powers of $x$ always stand on the left of powers of $\partial$, for example $\delta:=\exp((-x)\ast \partial)=\sum_{k=0}^{+\infty}\frac{(-1)^k}{k!}x^k\partial^k$.},   $\int:=(1-exp((-x)\ast \partial))\partial^{-1} $,  	$A_{k;i}:=\exp((\xi^{i}-1)x\ast \partial)\in \hat{D}^{sym}_{1}\hat{\otimes}_{K}\tilde{K}$ (in the case when $k$ is fixed, simply written as $A_i$), where $\tilde{K}=K[\xi]$, $\xi$ is a primitive $k$th root of unity, 
$\Gamma_i=(x\partial)^i$. $B_n=\frac{1}{(n-1)!}x^{n-1}\delta\partial^{n-1}$. 

$\hat{D}^{sym}_{1}\hat{\otimes}_{K}\tilde{K}$ means the same ring $\hat{D}^{sym}_{1}$, but defined over the base field $\tilde{K}$.

The operator $P\in D_1$ is called {\it normalized} if $P=\partial^p+a_{p-2}\partial^{p-2}+\ldots $. The operator $P\in D_1$ is {\it monic} if its highest coefficient is 1. Analogously, $P\in \hat{D}_1^{sym}$ is monic if $\sigma (P)=\partial^p$. 

	We denote $D^i=\partial^i$ if $i\ge 0$ and $\int^{-i}$ if $i<0$. Operators written in the (Standard) form as 
	$$
	H=[\sum_{0\leq i<k}f_{i;r}(x, A_{k;i}, \partial )+\sum_{0<j\leq N}g_{j;r}B_{j}]D^{r}
	$$
	are called HCP and  form a sub-ring $Hcpc(k)$.   Here $f_{i;r}(x, A_{k;i}, \partial)$ is a polynomial of $x, A_{k;i},\partial$,  $\Ord(f_{i;r})=0$,  of the form
		$$
		f_{i;r}(x, A_{k;i}, \partial )=\sum_{0\leq l\leq d_{i}}f_{l,i;r}x^l A_{k;i}\partial^l
		$$
		for some $d_i\in \dz_+$, where $f_{l,i;r}\in \tilde{K}$. The number  $d_i$ is called the {\it $x$-degree of $f_{i;r}$}:  $deg_{x}(f_{i;r}):= d_i$; $g_{j;r}\in \tilde{K}$, $g_{j;r}=0$ for $j\le -r$ if $r<0$.

	They can be written also in G-form: 
	$$
	H=(\sum_{0\leq i<k}\sum_{0\leq l\leq d_i} f'_{l,i;r}\Gamma_lA_i+\sum_{0<j\leq N}g_{j;r}B_{j})D^{r}
	$$
	The $A$ and $B$ Stable degrees of HCP are defined as 
	$$
	Sdeg_A(H)=\max \{d_i|\quad 0\leq i<k \} \quad \mbox{or $-\infty$, if all $f_{l,i;r}=0$ } 
	$$
	and 
	$$Sdeg_B(H)=\max\{j|\quad g_{j;r}\neq0\} \quad \mbox{or $-\infty$, if all $g_{j;r}=0$}
	$$
	
	In the case when $Sdeg_B(HD^p)=-\infty,\forall p\in \mathbf{Z}$ $H$ is called {\it totally free of} $B_j$.
	
	An operator $P\in  \hat{D}_1^{sym}$ satisfies {\it condition $A_q(k)$}, $q,k\in \dz_+$, $q>1$ if 
	\begin{enumerate}
		\item
		$P_{t}$ is a HCP  from $Hcpc (q)$  for all $t$;
		\item
		$P_{t}$ is totally free of $B_j$ for all $t$;
		\item
		$Sdeg_A(P_{\Ord (P)-i})< i+k$ for all $i>0$;
		\item
		$\sigma (P)$ does not contain $A_{q;i}$, $Sdeg_A(\sigma (P))=k$.
	\end{enumerate}

3. In section 3, $\mathfrak{B}=\crr_S$ is the right quotient ring of $\crr = \tilde{K}^{\oplus k} [D,\sigma ]$ by $S=\{D^k|k\ge 0\}$. And the ring of skew pseudo-differential operators is defined as
	$$
	E_k:=\tilde{K}[\Gamma_1, A_1]((\tilde{D}^{-1}))=\{\sum_{l=M}^{\infty}P_l\tilde{D}^{-l} | \quad P_l\in \tilde{K}[\Gamma_1, A_1]\} \simeq \tilde{K}^{\oplus k}[\Gamma_1]((\tilde{D}^{-1}))
	$$
with the commutation relations 
$$
\tilde{D}^{-1}a=\sigma (a)\tilde{D}^{-1}, \quad a\in \tilde{K}[\Gamma_1, A_1] \quad \mbox{where \quad }
\sigma (A_1)= \xi^{-1} A_1, \quad \sigma (\Gamma_1)=\Gamma_1-1.
$$	
	
	$\widehat{Hcpc}_B(k)$ is the $\tilde{K}$-subalgebra in $\hat{D}_1^{sym}\hat{\otimes}\tilde{K}$ consisting of operators whose homogeneous components are HCPs totally free of $B_j$. 
	
	$$\Phi : \tilde{K}[A_1,\ldots ,A_{k-1}]\rightarrow \tilde{K}^{\oplus k} , \quad   P\mapsto (\sum_i p_i\xi^{i}, \ldots , \sum_i p_i\xi^{i(k-1)})
$$
is an isomorphism of $\tilde{K}$-algebras. 	

The map 
$$
\hat{\Phi}: \widehat{Hcpc}_B(k) \hookrightarrow E_k
$$
defined on monomial HCPs from $\widehat{Hcpc}_B(k)$ as $\hat{\Phi} (a A_j\Gamma_iD^l):=a \Phi (A_j)\Gamma_i\tilde{D}^l$ and extended by linearity on the whole $\tilde{K}$-algebra $\widehat{Hcpc}_B(k)$, is an embedding of $\tilde{K}$-algebras. 
	
	Suppose $B$ is a commutative sub-algebra of $D_1$, then $(C,p,\cf)$ stands for the spectral data of $B$ (the spectral curve, point at infinity and the spectral sheaf with vanishing cohomologies).
	
	The classical ring of pseudo-differential operators  is defined as
	$$
E=K[[x]]((\partial^{-1})).
	$$

There is an isomorphism of $\tilde{K}$-algebras $\psi:\mathfrak{B}\rightarrow M_k(C(\mathfrak{B}))$, where $C(\mathfrak{B})\simeq \tilde{K}[\tilde{D}^k, \tilde{D}^{-k}]$, ($\tilde{K}$ is diagonally embedded into $\tilde{K}^{\oplus k}$):
	$$
	\psi\begin{pmatrix}
		h_0 \\
		h_1 \\
		\cdots \\
		h_{k-1}
	\end{pmatrix}=\begin{pmatrix}
		h_0 &  &  &  \\
		& h_1 &  &  \\
		&  & \cdots &  \\
		&  &  & h_{k-1}
	\end{pmatrix}\quad \psi(D)=T:=\begin{pmatrix}
		& 1 &  & \cdots &  \\
		&  & 1 & \cdots &  \\
		\cdots & \cdots & \cdots & \cdots & \cdots \\
		&  &  & \cdots & 1 \\
		D^k &  &  & \cdots & 
	\end{pmatrix}
	$$
	with $\psi(D^l)=T^l$, and extended by linearity. The map $\psi$ can be obviously extended to 
	$$
\psi : \tilde{K}^{\oplus k}((\tilde{D}^{-1})) \hookrightarrow M_k(\tilde{K}((\tilde{D}^{-k}))).
	$$

Elements of the centralizer $C(\partial^k)$ embedded to $\mathfrak{B}\subset E_k$ via $\hat{\Phi}$ we'll call as a {\it vector form} presentation, and the same elements embedded to $M_k(\tilde{K}[D^k])$ via $\psi\circ \hat{\Phi}$ -- as a {\it matrix form} presentation. The embedding determines an isomorphism $C(\partial^k)\simeq M_k(\tilde{K}[D^k])$. 

Translating the description of the centralizer $C(\partial^k)$  into the vector form, we get that $\hat{\Phi}(C(\partial^k))$ consists of Laurent polynomials in 
$\tilde{D}$ with coefficients from $K^{\oplus k}$ and with additional conditions: the coefficient $s_i$ at $\tilde{D}^{-i}$, $i>0$, has a shape $s_{i,j}=0$ for $j=0, \ldots , i-1$.

4. The Newton Polygon of a polynomial $f(x,y)=\sum c_{ij}x^{i}y^{j}\in K[x,y]$ is the convex hull of all the points $(i,j)$ with respect to those $c_{ij}\neq 0$.

Following \cite{Dixmier}, we define $E(f)=\{(i,j)|c_{ij}\neq 0 \}$ as the point set of the Newton Polygon. Suppose $\sigma, \rho$ are two real numbers, we define the weight degree $v_{\sigma,\rho}(f)$: 
$$v_{\sigma,\rho}(f)=sup\{\sigma i+\rho j|(i,j)\in E(f)\}, \quad E(f,\sigma,\rho)=\{(i,j)|v_{\sigma,\rho}(f)=\sigma i +\rho j \}.$$

For $P\in A_1$ we set $\ord_x(P)=v_{1,0}(P)$, $\ord (P)=v_{0,1}(P)$. 

For any polynomial $f$, we call  $f_{\sigma , \rho}=f_{\sigma , \rho} (f):= \sum_{(i,j)\in E(f,\sigma,\rho)} c_{ij}x^{i}y^{j}$ the {\it $(\sigma, \rho)$-homogeneous polynomial associated to} $f$. 

The line $l$ through the points of the Newton polygon of $f_{\sigma , \rho}$ is called the {\it $(\sigma, \rho)$-top line} of the Newton polygon.

Following \cite{GGV} for an operator $B\in A_1$ with $f_{1,0}(B)=hx^{v_{1,0}(B)}$, $h\in \dc [y]$, $f_{0,1}(B)=gy^{v_{0,1}(B)}$, $g\in \dc [x]$, we define  
$$en_{1,0}(B):=(v_{1,0}(B),\deg_y(h(y))), \quad st_{1,0}(B):=(v_{1,0}(B),-\Ord (h(x))),$$ 
$$en_{0,1}(B):=(-\Ord (g(x)),v_{0,1}(B)),\quad st_{0,1}(B):=(\deg_x(g(x)),v_{0,1}(B)).$$

A {\it counterexample} to the Dixmier conjecture is called as {\it DC-pair} (in spirit of \cite{GGV}). 

For an endomorphism $\varphi$ of the first Weyl algebra of subrectangular type with $P=\varphi (x)$, $Q=\varphi (\partial )$,  $d= GCD(\ord (P),\ord (Q))$, $l=GCD(\ord_x(P),\ord_x(Q))$, $d<l$ we define {\it essential GCD of $P$, $Q$} (usually denoted for short as $d_2$) as 
$$
EssGCD(P,Q):= \frac{d}{GCD(l,d)}>1.
$$

5. Let $P\in  \hat{D}^{sym}_{1}$, $\Ord(P)=k>0$ be a regular operator. An  invertible operator $S\in  \hat{D}^{sym}_{1}$ with $\Ord (S)=0$ such that 
	$$P=S^{-1}\partial^{k}S $$
is called a {\it Schur operator for $P$}. A Schur operator for {\it normalized differential operator} $P\in D_1$ has additional condition $S_0=1$, $S_{-1}=0$. Besides, it satisfies additional conditions ($A_p(0)$). 

For a given pair of operators $Q,P\in D_1$ with $\Ord (Q)=\ord (Q)=q> 0$, $\Ord (P)=\ord (P)=p\ge 0$ we define a {\it normal form} of $P$ with respect to $Q$ as the operator $P':=S^{-1}PS$, where $S$ is a Schur operator for $Q$. 

A solution to the string equation $P,Q\in D_1$ such that $P$ is monic and normalized is called a {\it normalized solution to the string equation}.

Analogously, a pair of commuting operators $P',Q'\in D_1$ such that $P'$ is monic and normalized is called a {\it normalized pair of commuting operators}.

A sequence of numbers, called {\it sequential GCD}, defined for the head-chopping process  (see definition \ref{D:distance}) as   
$$
SeqGCD_i(P,Q):=GCD (\ord (P), \ord (Q_0), \ldots ,\ord (Q_i)), \quad i=0, \ldots ,I
$$
and a {\it distance} of $P,Q$ is defined by the formula 
$$
Dist(P,Q):= \sum_{i=0}^{I-1} (n_i\ord F_i(P,Q)-\ord F_{i+1}(P,Q)).
$$ 
The condition on orders imply that $\sigma (Q_i)=c_i\partial^{\Ord (Q_i)}$, $c_i\in K$ (see remark \ref{R:ord=deg}). We'll call the pairs of constants $(\Ord (Q_0), c_0), \ldots , (\Ord (Q_I), c_I)$ as the {\it sequence} of $P,Q$. 

A pair of elements $(\partial^p, \tilde{Q}')\in C(\partial^p)\subset \hat{D}_1^{sym}$ is called a {\it normal pair} of order $(p,q)$, where $q=\Ord (\tilde{Q}')$, if it satisfies conditions of lemma \ref{L:invariant_conjugation} (the termination condition for the head-chopping process), $SeqGCD_I(\partial^p, \tilde{Q}')=1$, $SeqGCD_{I-1}(\partial^p, \tilde{Q}')>1$ and $Dist(\partial^p, \tilde{Q}')<p+q$.

This pair is called a {\it normalized normal pair} of order $(p,q)$ if the ring $B':=K[\partial^p, \tilde{Q}']$ is a normalized normal form of rank one with respect to some set of generators (see definition \ref{D:normalised_normal_form}).

A pair of elements $(\partial^p, \tilde{Q})\in  \hat{D}_1^{sym}$ is called a {\it normal string pair}  of order $(p,q)$, where $q=\Ord (\tilde{Q})$, if  $[\tilde{Q},\partial^p]=1$ and the pair 
$(\partial^p, \tilde{Q}-\tilde{Q}_{-p})$ is a normal pair of order $(p,q)$.
This pair is called a {\it normalized normal string pair} of order $(p,q)$ if the pair $(\partial^p, \tilde{Q}-\tilde{Q}_{-p})$ is a normalized normal pair of order $(p,q)$.

\section{Preliminaries from the theory of normal forms}
\label{S:prelim}

Suppose $K$ is a field of characteristic zero,  $K\subset \dc$. Although in the main part of the paper we will work with the field $\dc$, we have decided to present a number of results in arbitrary generality for the convenience of further references.

Following the notations in \cite{A.Z}, denote $\hat{R}:=K[[x]]$, $D_1:=\hat{R}[\partial ]$, define the $K$-vector space
$$
\mathcal{M}_{1}:=\hat{R}[[\partial]]=\{\sum_{k\ge 0}a_{k}\partial^{k}|a_{k}\in \hat{R}\quad  \forall k \in \mathbb{N}_{0}\} ,
$$
where $v:\hat{R}\rightarrow \mathbb{N}_0 \cup \infty$ is the discrete valuation defined by the unique maximal ideal $\idm$ of $\hat{R}$, for any element $0\neq P:=\sum_{k \ge 0}a_{k}\partial^{k}\in \mathcal{M}_{1}$ define the order function 
$$
\Ord(P):=\sup\{k-v(a_{k})|\quad k\in \mathbb{N}_{0}\}\in \mathbb{Z}\cup \{\infty\}.
$$
Define the space
$$
\hat{D}^{sym}_{1}:=\{Q\in\mathcal{M}_1|\quad \Ord(Q)<\infty \} .
$$
By definition, any element $P\in \hat{D}^{sym}_{1}$ is written in the {\it canonical form} 
$$
P:=\sum_{k-i\le \Ord (P)}\alpha_{k,i}x^{i}\partial^{k}.
$$ 
We call $P_{m}:=\sum_{k-i=m}\alpha_{k,i}x^{i}\partial^{k}$ as the $m$-th {\it homogeneous component} of $P$, we call $\sigma(P):=P_{\Ord(P)}$ as the {\it highest symbol} of $P$. Then we have the (uniquely defined) {\it homogeneous decomposition} for any $P\in \hat{D}^{sym}_{1}$:
$$
P=\sum_{m=-\infty }^{\Ord }P_{m}. 
$$

\begin{theorem}{(\cite{A.Z},Theorem 2.1)}
	\label{T:A.Z 2.1}
	The following statement are properties of $\hat{D}^{sym}_{1}$
	\begin{enumerate}
		\item $\hat{D}^{sym}_{1}$ is a ring (with natural operations $\cdot$ ,+ coming from $D_{1}$); $\hat{D}^{sym}_{1} \supset D_{1}$.
		\item $\hat{R}$ has a natural structure of a left $\hat{D}^{sym}_{1}$-module, which extends its natural structure of a left $D_{1}$-module.
		\item We have a natural isomorphism of $K$-vector spaces 
		$$
		F:=\hat{D}^{sym}_{1}/ \mathfrak{m}\hat{D}^{sym}_{1}\rightarrow K[\partial]
		$$
		where $\mathfrak{m}=(x)$ is maximal ideal of $\hat{R}$.
		\item Operators from $\hat{D}^{sym}_{1}$ can realise arbitrary endomorphisms of the $K$-algebra $\hat{R}$ which are continuous in the $\mathfrak{m}$-adic topology\footnote{ We consider here the representation of the ring  $\hat{D}_1^{sym}$ given in item 2, in terms of this representation any continuous endomorphism can be represented by an operator from $\hat{D}_1^{sym}$ (the details can be found in \cite{A.Z}), so $\hat{D}_1^{sym}\supset \End_{K-alg}^c(\hat{R})$, however, say, the operator $\partial$ is not a $K$-algebra endomorphism}.
		\item There are Dirac delta functions, operators of integration, difference operators.
	\end{enumerate}
\end{theorem}

To give an example and set up notations, let $\alpha \in \End_{K-alg}^c\hat{R}$ be a continuous algebra endomorphism. Then it is defined by the image $\alpha (x)\in \idm $. Put $u=\alpha (x)-x$, and define $\hat{D}_1^{sym}\ni P_{\alpha}=\sum_{{i}\ge 0} ({i}!)^{-1}{u}^{{i}}{\partial}^{{i}}$. Then for any $f\in \hat{R}$ we have 
$$
P_{\alpha}\circ f(x)=f(x+u);
$$
in particular, $P_{\alpha}$ realize the endomorphism $\alpha$. Note that this operator can be produced by the following receipt. Consider a commutative power series ring \\
$K[[ x, \partial ]]$ with a product denoted by $*$. If we take the formal exponent $\exp (u*\partial )$, write all coefficients at the left hand side and replace $*$ by $\cdot$, we'll obtain the operator  $P_{\alpha}$.  

With the help of this new notation the Dirac delta function is given by the series $\delta:=exp((-x)\ast \partial):= 1- x\partial +\frac{1}{2!}x^2\partial^2-\ldots$,  which satisfies $\delta \circ f(x)=f(0)$, and the operator of integration is given by the series
$$\int:=(1-exp((-x)\ast \partial))\cdot \partial^{-1}=\sum_{k=0}^{\infty}\frac{x^{k+1}}{(k+1)!}(-\partial)^{k}, $$
note that it satisfies  
$$
\int \circ x^{m}=\frac{x^{m+1}}{m+1}.
$$
Note that  $\int$ is only the right inverse of $\partial$, because $\partial \int =1$ and  $\int \cdot \partial =((1-exp((-x)\ast \partial))\cdot \partial^{-1}\cdot \partial)=1-\delta$. 

\begin{rem}
\label{R:ord_properties}
There are the following obvious properties of the order function:
\begin{enumerate}
\item
$\Ord (P\cdot Q)\le \Ord (P)+\Ord (Q)$, and the equality holds if $\sigma (P)\cdot \sigma (Q)\neq 0$,
\item
$\sigma (P\cdot Q)=\sigma (P)\cdot \sigma (Q)$, provided $\sigma (P)\cdot \sigma (Q)\neq 0$,
\item
$\Ord (P+Q)\le \max\{\Ord (P),\Ord (Q)\}$.
\end{enumerate}
In particular, the function $-\Ord$ determines  a discrete pseudo-valuation on the ring $\hat{D}_1^{sym}$. 
\end{rem}

Denote by $A_1:=K[x][\partial ]$ the first Weyl algebra. Clearly, $A_1\subset D_1 \subset \hat{D}_1^{sym}$. 

Recall that the Newton Polygon of a polynomial $f(x,y)=\sum c_{ij}x^{i}y^{j}\in K[x,y]$ is the convex hull of all the points $(i,j)$ with respect to those $c_{ij}\neq 0$.

Following \cite{Dixmier}, we define $E(f)=\{(i,j)|c_{ij}\neq 0 \}$ as the point set of the Newton Polygon. Suppose $\sigma, \rho$ are two real numbers, we define the weight degree $v_{\sigma,\rho}(f)$ and the top points set $E(f,\sigma,\rho)$ as follows:
\begin{Def}
\label{D:weight}
	$v_{\sigma,\rho}(f)=sup\{\sigma i+\rho j|(i,j)\in E(f)\}, \quad E(f,\sigma,\rho)=\{(i,j)|v_{\sigma,\rho}(f)=\sigma i +\rho j \}$. 
\end{Def}

If $E(f)=E(f,\sigma,\rho)$, then we call such $f$  {\it $(\sigma,\rho)$-homogeneous}. 
\begin{Def}
For any polynomial $f$, we call  $f_{\sigma , \rho}=f_{\sigma , \rho} (f):= \sum_{(i,j)\in E(f,\sigma,\rho)} c_{ij}x^{i}y^{j}$ the {\it $(\sigma, \rho)$-homogeneous polynomial associated to} $f$. 

The line $l$ through the points of the Newton polygon of $f_{\sigma , \rho}$ is called the {\it $(\sigma, \rho)$-top line} of the Newton polygon.
\end{Def}

We define the Newton Polygon of an operator $P\in A_1$, $P= \sum a_{ij}x^{i}\partial^{j}$ as the Newton polygon of the polynomial $f= \sum a_{ij}x^{i}y^{j}$, and we'll use the same notations $v_{\sigma,\rho}(P)$, $f_{\sigma ,\rho}(P)$, $E(P)$ for operators.

Set $\ord_x(P):=v_{1,0}(P)$, $\ord (P):=v_{0,1}(P)$\footnote{In \cite{GZ1} a notation $\deg$ instead of $\ord$ was used.}. For any differential operator (not necessarily from $A_1$) define $HT(P):= a_n$, $\ord (P):=n$ if $P=\sum_{i=0}^na_i\partial^i$ with $a_n\neq 0$. We'll say $P$ is {\it monic} if $HT(P)=1$, and we'll say $P$ is {\it formally elliptic} if $HT(P)\in K$. 

For any two polynomials $f,g$ we define the standard Poisson bracket $\{f,g\}=\frac{\partial f}{\partial x}\frac{\partial g}{\partial y}-\frac{\partial f}{\partial y}\frac{\partial g}{\partial x}$.

The {\it tame automorphisms} of $A_{1}$ are defined as follows: for any integer $n$, and $\lambda\in K$, define 
\begin{equation}
\label{E:tame}
\Phi_{n,\lambda}:
\begin{cases}
\partial  \rightarrow & \partial
\\x   \rightarrow & x+\lambda \partial^{n}
\end{cases},  \quad 
\Phi'_{n,\lambda}:
\begin{cases}
\partial  \rightarrow & \partial+\lambda x^{n}
\\x   \rightarrow & x
\end{cases},  
\quad 
\Phi_{a,b,c,d}:
\begin{cases}
\partial  \rightarrow & a\partial+bx
\\x   \rightarrow & c\partial +dx
\end{cases}  
\end{equation}
with $ad-bc=1$. According to \cite[Th. 8.10]{Dixmier} the group $Aut (A_1)$ is generated by all $\Phi_{n,\lambda}, \Phi'_{n,\lambda}$ and $\Phi_{a,b,c,d}$ (cf. also \cite{ML1}).

For the convenience of the reader we recall one important result from Dixmier's paper:

\begin{theorem}{(\cite[Lemma 2.4, 2.7]{Dixmier})}
\label{T:Dixmier2.7}
	Suppose $P,Q$ are two differential operators from $A_{1} $, $\sigma, \rho$ are real numbers, with $\sigma +\rho >0, v=v_{\sigma.\rho}(P) $ and $ w=w_{\sigma.\rho}(Q) $. If now $f_{1}$, $g_{1}$ are the $(\sigma,\rho)$-homogeneous polynomials  associated to $P, Q$, then 
	\begin{enumerate}
	\item
	$\exists T,U\in A_{1}$ such that
	\begin{enumerate}
	\item
	$[P,Q]=T+U$
	\item
	$v_{\sigma.\rho}(U)<v+w-\sigma-\rho$
	\item
	$E(T)=E(T,\sigma,\rho)$ and $v_{\sigma.\rho}(T)=v+w-\sigma-\rho$ if $T\neq 0$.
	\end{enumerate}
	\item
	The following are equivalent
	\begin{enumerate}
	\item
	T=0
	\item
	$\{f_{1},g_{1}   \}=0$;
	
	If $v,w$ are integers, then these two conditions are equivalent to
	\item
      $g_{1}^{v}$ is proportional to $f_{1}^{w}$ 
      \end{enumerate}
  \item Suppose $T\neq 0$, then the polynomial $(\sigma,\rho)$ associated to $[P,Q]$ is $\{f_{1},g_{1}\}$.
  \item We have $f_{\sigma ,\rho} (PQ)=f_1 g_1$, $v_{\sigma ,\rho}(PQ)=v+w$. 
	\end{enumerate}
\end{theorem}

\begin{rem}
\label{R:ord=deg}
	Note that the order function $\Ord$ coincide with the weight function $v_{1,-1}$ on the ring $A_1$.   If $P\in D_1$ has an invertible highest coefficient $HT(P)$, then it is easy to see that  $\ord (P)=\Ord (P)$. Vice versa, any $P\in D_1$  with $\ord (P)=\Ord (P)$ has an invertible highest coefficient $HT(P)$. Moreover, in this case $\sigma (P)=c\partial^{\Ord (P)}$, $c\in K$. 
\end{rem}

\begin{Def}
\label{D:Aki}
	Let $\xi$ be a primitive $k$-th root of unity, $\tilde{K}=K[\xi]$. For any $i\in \mathbb{Z}$, we define operators 
	$$
	A_{k;i}:=\exp((\xi^{i}-1)x\ast \partial)\in \hat{D}^{sym}_{1}\hat{\otimes}_{K}\tilde{K},
	$$
	where $\hat{D}^{sym}_{1}\hat{\otimes}_{K}\tilde{K}$ means the same ring $\hat{D}^{sym}_{1}$, but defined over the base field $\tilde{K}$. 
	
	Further, if it will be clear from the context, we'll omit index $k$ and use notation $A_i:=A_{k;i}$.
\end{Def}

\begin{lemma}{(cf. \cite{A.Z},Lemma 7.2)}
	\label{T:A.Z L 7.2}
	The sum $$
	A:=c_0+c_1 A_{k,1}+\dots+c_{k-1}A_{k;k-1},\quad c_{i}\in \tilde{K}
	$$
	is equal to zero iff $c_{i}=0, i=0,\dots,k-1$. If it is not equal to zero, then it is of order zero. 
	
	Moreover, $A$ is a polynomial in $\partial$ iff $c_1=\ldots = c_{k-1}=0$.
\end{lemma}

The following proposition is a union of a partial case of \cite[Prop. 7.1]{A.Z} and \cite[L. 2.4]{GZ1}:

\begin{Prop}
	\label{T:A.Z 7.1}
	In the notation of definition \ref{D:Aki} we have 
	\begin{enumerate}
		\item For any $i,j,p$, we have
	$$A_{i} A_{j}=A_{i+j} \quad \partial^{p}A_{i}=\xi^{pi}A_{i}\partial^{p} \quad A_{i}x^{p}=\xi^{pi}x^{p}A_{i}\quad \int^{p}A_{i}=\xi^{-pi}A_{i}\int^{p}.$$
		\item For a given $Q \in  \hat{D}^{sym}_{1} $ assume that $[\partial^{k},Q]=0 $. 
		Then 
		$$Q=\sum_{m=0}^{\Ord (Q)}c_{m}\partial^{m}+c_{-1}\int+\dots+c_{-k+1}\int^{k-1}, $$
		where $c_{i}= c_{i,0}+c_{i,1}A_{1}+\dots+c_{i,k-1}A_{k-1}\in \hat{D}^{sym}_{1}\hat{\otimes}_{K}\tilde{K}$, the coefficients $c_{i,j}\in \tilde{K}$ are uniquely defined, and the coefficients $c_{i,j}$ with $i<0$ satisfy the following conditions:
	$$
	\sum_{j=0}^{k-1}c_j=0, \quad \sum_{j=1}^{k-1}c_j (\xi^j-1)^q=0, \quad 1\le q\le -l-1,
	$$
	or, equivalently, 
	$$
\sum_{j=0}^{k-1}c_j\xi^{j(q-1)}=0 \quad \mbox{for\quad } q=1, \ldots , -l.
	$$
	Vice versa, any such operator $Q$ commutes with $\partial^k$.
	\end{enumerate}
	
\end{Prop}

Consider the space $F=K[\partial ]$. It has a natural structure of a right $\hat{D}_1^{sym}$-module via the isomorphism of vector spaces $F\simeq \hat{D}_1^{sym}/\idm \hat{D}_1^{sym}$.

\begin{Def}{(cf. \cite{A.Z}, Def. 6.4)}
\label{D:regular}
An element $P \in \hat{D}_1^{sym}$ is called \emph{regular} if the $K$--linear map $F \xrightarrow{ (-\circ \sigma(P))} F$ is injective, where $\circ$ means the action on the module $F$. In particular, $P$ is regular if and only if its symbol $\sigma(P)$ is regular.
\end{Def}

It's easy to see that an operator $P\in D_1$ with an invertible highest coefficient (cf.remark \ref{R:ord=deg}) is an example of a regular operator. 

Recall that any such operator can be normalised, i.e. reduced to the form $P=\partial^k + c_{k-2}\partial^{k-2}+\ldots + c_0$, with the help of some change of variables and conjugation by invertible function, see e.g. \cite[Prop. 1.3 and Rem. 1.6]{BZ}.

The following proposition is a union of \cite[Prop. 7.2]{A.Z} and \cite[Prop.2.3]{GZ1}:

\begin{Prop}
	\label{T:A.Z 7.2}
	\begin{enumerate}
	\item	
	Let $P\in  \hat{D}^{sym}_{1}$, $\Ord(P)=k>0$ be a regular operator. Then there exists an invertible operator $S\in  \hat{D}^{sym}_{1}$ with $\Ord (S)=0$ such that 
	$$P=S^{-1}\partial^{k}S. $$
 \item
Let $P\in D_1$ be a normalized operator of positive degree, i.e. $P=\partial^k + c_{k-2}\partial^{k-2}+\ldots + c_0$, $k>0$. Then there exists an operator $S$ from item 1 such that $S_0=1$, $S_{-1}=0$. 
	\end{enumerate}
	
We'll call the operator $S$ from item 1 as a Schur operator for $P\in  \hat{D}^{sym}_{1}$, and the operator $S$ from item 2 as a Schur operator for {\it normalized differential operator} $P\in D_1$. 	
\end{Prop}

We define a series of operators $B_i$:  $B_{1}:=\delta, B_{2}:=x\delta \partial ,\ldots, B_{n}:=\frac{1}{(n-1)!}x^{n-1}\delta \partial^{n-1}$, and  $B_j:=0$ for any $j\leq 0$. Define $\Gamma_i=(x\partial)^i$ for $i\ge 0$, and for $i<0, \Gamma_i=0$. For convenience, we introduce also a new notation: for any integer $n$ we set  $D^n=\partial^n$ if $n\ge 0$ and $D^n=\int^{-n}$  otherwise. Obviously, we have $\partial \delta=\delta x=0$ and $\Ord(B_{j})=0$ for any $j \in \mathbb{N}$. From this observation and definition of $\delta$ we also get $\delta\delta=\delta$. 

\begin{lemma}{(\cite[Lemma 2.5]{GZ1})}
	\label{L:1}
	For a fixed $k\in \dn$ let $A_i=A_{k;i}$, $\xi$ be the $k$-primitive root, $B_j$ are defined as above. Then we have
	\begin{enumerate}
		\item $A_{i} B_{j}=B_{j} A_{i}=\xi^{i(j-1)}B_{j}$ for any $i,j \in \mathbb{N}$;
		\item 
	$$
\int x^m = \frac{m!}{(m+1)!}x^{m+1} + \sum_{i=1}^{\infty}(-1)^i \frac{m!}{(m+i+1)!}x^{m+i+1}\partial^i;
$$
	In particular, $\int x^m\delta = \frac{1}{m+1}x^{m+1}\delta$;
		\item $\int^{m} \partial^{m}=1-\sum_{k=1}^{m}B_{k}$ for any $m \in \mathbb{N}$;
		\item 
		$$\int^u f(x)=f(x)\int^u +
		\sum_{l=1}^{\infty}\binom{-u}{l}f(x)^{(l)}\int^{u+l}$$ for any  $f(x)\in \hat{R}$, $u\in\dn$;
		
		\item $B_i B_j=\delta_i^jB_j$, where $\delta_i^j$ is the Kronecker delta;
		
		\item $A_i\Gamma_j=\Gamma_jA_i$;
		\item $$
		D^i\Gamma_j=\sum_{l=0}^{j}\binom{j}{l}i^{j-l}\Gamma_lD^i, \quad \Gamma_jx^i=x^i(\sum_{l=0}^j\binom{j}{l}i^{j-l}\Gamma_l)
		$$
		\item $\Gamma_iB_j=B_j\Gamma_i=(j-1)^iB_j$;
		\item $D^u B_j=B_{j-u}D^u$. In particular, $D^u B_j=B_{j-u}D^u=0$ if $u>0$ and $j-u\le 0$ or $u<0$ and $j\le 0$.
	\end{enumerate}	
	where we assume in all these formulae that $0^0:=1$.
\end{lemma}

\begin{Def}{(\cite[Def. 2.3]{GZ1})}
\label{D:HCP}
Let's fix some $k\in \dn$. Assume $\xi$ is the $k$-primitive root, $A_i=A_{k,i}$. Let $\tilde{K}=K[\xi]$. An element $H\in \hat{D}_{1}^{sym}\hat{\otimes}_K\tilde{K}$ is  called {\it homogeneous canonical polynomial  (in short of HCP)}  if $H$ can be written in the form
	\begin{equation}
	\label{E:HCP}
	H=[\sum_{0\leq i<k}f_{i;r}(x, A_{k;i}, \partial )+\sum_{0<j\leq N}g_{j;r}B_{j}]D^{r}
	\end{equation}
	for some $N\in\dn$, $r\in \dz$, where 
	
		1. $f_{i;r}(x, A_{k;i}, \partial)$ is a polynomial of $x, A_{k;i},\partial$,  $\Ord(f_{i;r})=0$,  of the form
		$$
		f_{i;r}(x, A_{k;i}, \partial )=\sum_{0\leq l\leq d_{i}}f_{l,i;r}x^l A_{k;i}\partial^l
		$$
		for some $d_i\in \dz_+$, where $f_{l,i;r}\in \tilde{K}$. The number  $d_i$ is called the {\it $x$-degree of $f_{i;r}$}:  $deg_{x}(f_{i;r}):= d_i$.
		
		2. $g_{j;r}\in \tilde{K}$.
		
		3. $g_{j;r}=0$ for $j\le -r$ if $r<0$.

In particular, $H$ is homogeneous and $\Ord (H)=r$. 
\end{Def}

The form \eqref{E:HCP} of HCP from definition is uniquely defined:

\begin{lemma}{(\cite[Lem. 2.6]{GZ1})}
\label{L:correct_HPC}
Let $H$ be a HCP. Then $H=0$ iff $g_{j;r}=0$ for all $j$ and $f_{l,i;r}=0$ for all $i,l$. 

In particular, any HCP can be uniquely written in the form \eqref{E:HCP}.
\end{lemma}

\begin{Def}{(\cite[Def. 2.4]{GZ1})}	
\label{D:Sdeg_A}
Let $H$ be a HCP. We define 
$$Sdeg_A(H)=\max \{d_i|\quad 0\leq i<k \} \quad \mbox{or $-\infty$, if all $f_{l,i;r}=0$ } 
$$
and 
$$Sdeg_B(H)=\max\{j|\quad g_{j;r}\neq0\} \quad \mbox{or $-\infty$, if all $g_{j;r}=0$}
$$

We define a  {\it homogeneous canonical polynomial combination} (in short HCPC) as a finite sum of HCP (of different orders).
We extend the functions $Sdeg_A$, $Sdeg_B$ in an obvious way to all HCPCs. 

We'll say that a HCPC $H$ {\it doesn't contain $A_i$} if $f_{l,i;r}=0$ for all $i>0$ and all $r$. We'll say that a HCPC $H$ {\it doesn't contain $B_j$} if $Sdeg_B(H)=-\infty$.
\end{Def}

The following lemmas describe basic properties of $HCPC, Sdeg_A, Sdeg_B$.

\begin{lemma}{(\cite[Lem. 2.7]{GZ1})}
\label{L:obvious}
	Suppose $H$ and $M$ are two HCPCs, $k_1,k_2\in \tilde{K}$ are two arbitrary constant. Then $T=k_1H+k_2M$ is also a HCPC, with 
	$$
	Sdeg_A(T)\leq \max\{Sdeg_A(H), Sdeg_A(M) \}
	$$$$
	Sdeg_B(T)\leq \max\{Sdeg_B(H), Sdeg_B(M) \}
	$$
\end{lemma}

\begin{lemma}{(\cite[Cor. 2.1]{GZ1})}
\label{C:HCPC(k)}
	For a fixed $k$ all $HCPCs$  form a subring in the ring $\hat{D}_1^{sym}\hat{\otimes}_K\tilde{K}$.
\end{lemma}

Denote by
$$
Hcpc(k):=\{\text{all HCPCs assoicated to }k\}
$$
the subring from corollary. 
It's easy to observe that if $a|b$, $b=ra$ and $\xi,\eta$ are respectively  the $a,b$-primitive roots, then $\eta^r=\xi$, i.e. $A_{b,ir}=A_{a,i}$. This means
$$
Hcpc(a)\subseteq Hcpc(b)
$$
Thus for $p,q$, if we assume $r=gcd(p,q)$, then we have 
$$
Hcpc(p)\bigcap Hcpc(q)\supseteq Hcpc(r).
$$

\begin{lemma}{(\cite[Lem. 2.10]{GZ1})}
		\label{L:HCPC}
	Suppose $H,M$ are two  HCPCs. Then $T:=HM$ is also a HCPC, what's more we have 
	\begin{enumerate}
		\item $Sdeg_A(T)\leq Sdeg_A(H)+Sdeg_A(M)$ (here we assume $-\infty +n=-\infty$). 
		\item 
		if $H,M$ are two  HCPs then
		\begin{enumerate}
			\item If $\Ord(H)\geq 0$, then $Sdeg_B(T)\leq \max \{Sdeg_B(H),Sdeg_B(M)\}$
			\item If $\Ord(H)< 0$, then $Sdeg_B(T)\leq \max \{Sdeg_B(H),Sdeg_B(M)-\Ord(H),-\Ord(H)\}$
		\end{enumerate}
	\end{enumerate}
\end{lemma}

\begin{Def}{(\cite[Def. 2.6]{GZ1})}
\label{D:total_B}
Let $H$ be an element from $Hcpc(k)$. We'll say $H$ is {\it totally free of $B_j$} if $Sdeg_B(HD^p)=-\infty$ for all $p\in\dz$\footnote{Note that definition of $Sdeg_B$ is sensitive with respect to the multiplication by $D^p$. For example, $H=\int$ is a HCP with $Sdeg_B(H)=-\infty$, but $HD= 1-\delta$ is a HCP with $Sdeg_B(HD)=1$. In this example $H$ doesn't contain $B_j$, but is not totally free of $B_j$.}. 
\end{Def}

\begin{lemma}{(\cite[Lem. 2.12]{GZ1})}
\label{L:total_B}
The subset of totally free of $B_j$ elements from $Hcpc(k)$ for a fixed $k$ form a subring in $\hat{D}_1^{sym}\hat{\otimes}_K\tilde{K}$. 
\end{lemma}

\begin{Prop}{(\cite[Prop. 2.4]{GZ1})}
\label{P:lower estimate}
Suppose 
$$
0\neq H=\sum_{i=0}^{k-1}\sum_{m=0}^da_{m,i}A_i\Gamma_m\int^u.
$$ 	
is a HCP  from $Hcpc(k)$,  with $\Ord(H)=-u<0$. Then $H$ is totally free of $B_j$ iff the linear system of equations on $a_{m,i}$ holds:
\begin{equation}
\label{E:system_on_B_j}
\sum_{i=0}^{k-1}\sum_{m=0}^d\xi^{i(j-1)}(j-1)^ma_{m,i}=0
\end{equation}
for any $1\leq j\leq u$.

\end{Prop}

\begin{Prop}{(\cite[Prop. 2.5]{GZ1})}
	\label{P:Si}
	Let $Q\in D_1$ be a normalized operator, assume $S\in \hat{D}_{1}^{sym}$ is a Schur operator for $Q$ from proposition \ref{T:A.Z 7.2}, i.e.  $S^{-1}QS=\partial^q$, such that   $S_{0}=1,S_{-1}=0$. Then we have 
	\begin{enumerate}
		\item $S_{-t}$ is a HCP from $Hcpc(q)$ for any $t\ge 0$ (i.e. it can be written as a HCP in $\hat{D}_1^{sym}\hat{\otimes}_K\tilde{K}$).
		\item If $S_{-t}\neq 0$ then $\frac{t}{q}-1< Sdeg_A(S_{-t})<t$ for any $t> 0$.
		\item $S_{-t}$ is totally free of $B_j$ for any $t\ge 0$.
		\end{enumerate}
\end{Prop} 

\begin{cor}{(\cite[Cor. 2.3]{GZ1})}
	\label{C:Si^(-1) no B}
In the notation of proposition \ref{P:Si} set $\tilde{S}=S^{-1}$. Then we have 
	\begin{enumerate}
	\item $\tilde{S}_0=1$, $\tilde{S}_{-1}=0$.
		\item $\tilde{S}_{-t}$ is a HCP from $Hcpc(q)$ for any $t\ge 0$.
		\item If $S_{-t}\neq 0$ then $\frac{t}{u}-1< Sdeg_A(\tilde{S}_{-t})<t$ for any $t> 0$.
		\item $\tilde{S}_{-t}$ is totally free of $B_j$ for any $t\ge 0$. 
	\end{enumerate}
\end{cor}

\begin{Def}
\label{D:normal_form}
For a given pair of operators $Q,P\in D_1$ with $\Ord (Q)=\ord (Q)=q> 0$, $\Ord (P)=\ord (P)=p\ge 0$ we define a {\it normal form} of $P$ with respect to $Q$ as the operator $P':=S^{-1}PS$, where $S$ is a Schur operator for $Q$ from proposition \ref{T:A.Z 7.2}, i.e. $S^{-1}QS=\partial^q$, $\Ord (S)=0$. 
\end{Def} 

\begin{rem}
\label{R:normal_forms}
The Schur operator is not uniquely defined, but up to multiplication by the elements of the centralizer  $C(\partial^q)\subset \hat{D}_1^{sym}$ from proposition \ref{T:A.Z 7.1}. By this reason the normal form of the operator $P$  is  not uniquely defined, but up to conjugation by elements from this centralizer. 

If $\sigma (Q)=\partial^{\Ord (Q)}$, $c\in K$ (this is the case we'll usually deal with), then there exists a monic  Schur operator (i.e. $S_0=1$). Indeed, in this case  $\sigma (S)\partial^q=\partial^q\sigma (S)$, so that $\sigma (S)\in C(\partial^q)$, $\sigma (S)$ is invertible and we can divide on $\sigma (S)$ if necessary. In this case we also can assume $S_{-1}=0$ (see proposition \ref{T:A.Z 7.2}, item 2).
\end{rem}

\begin{Def}{(\cite[Def. 2.8]{GZ1})}
\label{D:conditionA}
We'll say that an operator $P\in  \hat{D}_1^{sym}\hat{\otimes}_K\tilde{K}$ satisfies {\it condition $A_q(k)$}, $q,k\in \dz_+$, $q>1$ if 
\begin{enumerate}
\item
$P_{t}$ is a HCP  from $Hcpc (q)$  for all $t$;
\item
$P_{t}$ is totally free of $B_j$ for all $t$;
\item
$Sdeg_A(P_{\Ord (P)-i})< i+k$ for all $i>0$;
\item
$\sigma (P)$ does not contain $A_{q;i}$, $Sdeg_A(\sigma (P))=k$.
\end{enumerate}
\end{Def}

\begin{ex}
\label{Ex:Schur_op}
From proposition \ref{P:Si} an corollary \ref{C:Si^(-1) no B} we know that $S, S^{-1}$ satisfy condition $A_q(0)$, where $S$ is any monic Schur operator for normalised $Q\in D_1$. Moreover, if $S$ is a monic operator from $\hat{D}_1^{sym}$ of order zero satisfying condition $A_q(0)$ (but not necessary with $S_{-1}=0$), then $S^{-1}$ also satisfies the same properties. The proof is just the same as in \cite[Cor. 2.3]{GZ1}. 
\end{ex}

\begin{lemma}{(\cite[Lem. 2.15]{GZ1})}
\label{L:conditionA}
Suppose $P,Q\in \hat{D}_1^{sym}\hat{\otimes}_K\tilde{K}$ satisfy conditions $A_q(k_1)$, $A_q(k_2)$ correspondingly. Then 
$PQ$ satisfies condition $A_q(k_1+k_2)$.
\end{lemma}

\begin{Prop}{(\cite[Cor. 2.4]{GZ1})}
	\label{C:P'}
	Suppose $Q\in D_1$ is a normalized operator with  $\Ord (Q)=\ord (Q)=q> 0$. 
	Suppose $P\in \hat{D}_1^{sym}\hat{\otimes}_K\tilde{K}$  satisfies condition $A_q(0)$, $\Ord (P)=p$. 
	 Put $P'=S^{-1}PS$, where $S$ is a Schur operator for $Q$ from proposition \ref{P:Si}. 
	 
	 Then $P'$ satisfies condition $A_q(0)$.
\end{Prop}

\begin{theorem}{(\cite[Th. 2.2]{GZ1})}
\label{T: P in hat(D) is a dif_op}
	The following statements are equivalent:
	\begin{enumerate}
		\item $P\in \hat{D}_1^{sym}\hat{\otimes}_K\tilde{K}$ is a {\it differential operator} (i.e. $P\in D_1\otimes_K\tilde{K}$)  with constant highest symbol.
		\item $\forall p>1$, $P\in \hat{D}_1^{sym}\hat{\otimes}_K\tilde{K}$ satisfies condition $A_p(0)$ with an extra property: all homogeneous components $P_j$ don't contain  $A_i$.
		\item $\exists p>1$, $P\in \hat{D}_1^{sym}\hat{\otimes}_K\tilde{K}$ satisfies condition $A_p(0)$ with an extra property: all homogeneous components $P_j$ don't contain  $A_i$.
	\end{enumerate}
\end{theorem}

By definition, the normal form of an operator $P$ with respect to $Q$ belongs to the centralizer $C(\partial^q)$ if $[P,Q]=0$. The centralizer has another convenient description alternative to one given in proposition \ref{T:A.Z 7.1}, which will be used later. Let's recall this description from \cite[\S 3]{GZ1}. 

For a fixed $k$ we have an isomorphism of $\tilde{K}$-algebras 
$$
\Phi : \tilde{K}[A_1,\ldots ,A_{k-1}]\rightarrow \tilde{K}^{\oplus k} , \quad P\mapsto (1\circ P, \ldots , (\partial^{k-1}\circ P)\partial^{-k+1})
$$
(here $\partial^l\circ P$ is the notation from definition \ref{D:regular}, i.e. if $P=\sum_{i=0}^{k-1}p_iA_i$, then $\partial^l\circ P=(\sum p_i\xi^{il})\partial^l$, and $\tilde{K}^{\oplus k}$ denotes the semisimple algebra - the direct sum of algebras $\tilde{K}$).

\begin{rem}
\label{R:Phi_over_K}
Note that, by definition of $\Phi$, if $P\in \tilde{K}[A_1]$ belong to $\hat{D}_1^{sym}$ (i.e. all coefficients of the series representing $P$ in  $\hat{D}_1^{sym}$ belong to $K$), then $\Phi (P)\in K^{\oplus k}$, i.e. $\Phi$ is compatible with the extension of scalars of $\hat{D}_1^{sym}$.
\end{rem}

Consider the skew polynomial ring $\crr = \tilde{K}^{\oplus k} [D,\sigma ]$, where 
$$\sigma (a_0,\ldots ,a_{k-1})= (a_{k-1}, a_0, \ldots , a_{k-2}).$$ 
It is a right and left noetherian ring. The multiplicatively closed subset $S=\{ D^k, k\ge 0\}$ obviously satisfies the right Ore condition, consists of regular elements, and ${\bf ass} (S)=0$. So, the right quotient ring $\mathfrak{B}=\crr_S$ exists. Clearly, 
$$
\mathfrak{B}\simeq\tilde{K}^{\oplus k}[D,D^{-1}]=\{\sum_{l=M}^NP_lD^l | \quad P_l\in \tilde{K}^{\oplus k}\}\simeq \tilde{K}[A_1][D,D^{-1}],
$$
i.e. any element from $\mathfrak{B}$ can be written as a Laurent polynomial, and the commutativity relations of polynomials are given above: $D^{-1}a=\sigma (a) D^{-1}$, $a\in \tilde{K}^{\oplus k}$.

Let $C(\mathfrak{B})$ be the center of $\mathfrak{B}$. Then 
$$
C(\mathfrak{B})\cong \tilde{K}[D^k,D^{-k}],
$$
where $\tilde{K}$ is diagonally embedded into $\tilde{K}^{\oplus k}$, and there is an isomorphism of $\tilde{K}$-algebras $\psi: \mathfrak{B}\cong M_k(C(\mathfrak{B}))$ (see \cite[Lem 3.1]{GZ1} for details). 

Now consider the ring of skew pseudo-differential operators 
$$
E_k:=\tilde{K}[\Gamma_1, A_1]((\tilde{D}^{-1}))=\{\sum_{l=M}^{\infty}P_l\tilde{D}^{-l} | \quad P_l\in \tilde{K}[\Gamma_1, A_1]\} \simeq \tilde{K}^{\oplus k}[\Gamma_1]((\tilde{D}^{-1}))
$$
with the commutation relation as above (here $\tilde{K}[\Gamma_1, A_1]$ is a commutative subring in $\hat{D}_1^{sym}$):
$$
\tilde{D}^{-1}a=\sigma (a)\tilde{D}^{-1}, \quad a\in \tilde{K}[\Gamma_1, A_1] \quad \mbox{where \quad }
\sigma (A_1)= \xi^{-1} A_1, \quad \sigma (\Gamma_1)=\Gamma_1-1.
$$
The ring $E_k$ is endowed with a natural discrete pseudo-valuation, which we will denote as $-\ord_{\tilde{D}}$ (i.e. $\ord_{\tilde{D}}(\sum_{l=M}^{\infty}P_l\tilde{D}^{-l})= M$).  We extend the usual terminology used in this paper also for operators from $E_k$ (such as the notion of the highest coefficient, monic operators, etc.)

Denote $\widehat{Hcpc}_B(k)$ as the $\tilde{K}$-subalgebra in $\hat{D}_1^{sym}\hat{\otimes}\tilde{K}$ consisting of operators whose homogeneous components are HCPs totally free of $B_j$ (cf. lemma \ref{L:total_B}). 

\begin{lemma}{(\cite[Lem. 3.2]{GZ1})}
\label{L:embedding_Phi}
The map 
$$
\hat{\Phi}: \widehat{Hcpc}_B(k) \longrightarrow E_k,
$$
defined on monomial HCPs from $\widehat{Hcpc}_B(k)$ as $\hat{\Phi} (a A_j\Gamma_iD^l):=a \Phi (A_j)\Gamma_i\tilde{D}^l$ and extended by linearity on the whole $\tilde{K}$-algebra $\widehat{Hcpc}_B(k)$, is an embedding of $\tilde{K}$-algebras.
\end{lemma}

\begin{rem}
\label{R:hatPhi_over_K}
Again as in remark \ref{R:Phi_over_K}, if $P\in \hat{D}_1^{sym}\cap \widehat{Hcpc}_B(k)$, then $\hat{\Phi} (P)$ will be an operator with coefficients from $K$. 
\end{rem}

Note that the ring $\mathfrak{B}$ is naturally embedded into the ring $E_k$. Any operator from the centralizer $C(\partial^k)$, can be embedded in $E_k$ via $\hat{\Phi}$, because clearly $C(\partial^k)\subset \widehat{Hcpc}_B(k)$. From proposition \ref{T:A.Z 7.1} we immediately get that $\hat{\Phi}(C(\partial^k))\subset \mathfrak{B}$. 

Restricting $\psi$ on $\hat{\Phi}(C(\partial^k))$, we get the equality 
\begin{equation}
\label{E:matrix_isom}
\psi (\hat{\Phi}(C(\partial^k)))= M_k(\tilde{K}[D^k])
\end{equation}
(see \cite[Rem 3.3]{GZ1}). Elements of the centralizer $C(\partial^k)$ embedded to $\mathfrak{B}$ we'll call as a {\it vector form} presentation, and elements embedded to $M_k(\tilde{K}[D^k])$ -- as a {\it matrix form} presentation. 

\begin{rem}
\label{R:vector_form}
Translating the description of the centralizer $C(\partial^k)$ given in proposition \ref{T:A.Z 7.1}, item 2, into the vector form, we get that $\hat{\Phi}(C(\partial^k))$ consists of Laurent polynomials in 
$\tilde{D}$ with coefficients from $K^{\oplus k}$ and with additional conditions: the coefficient $s_i$ at $\tilde{D}^{-i}$, $i>0$, has a shape $s_{i,j}=0$ for $j=0, \ldots , i-1$. 
\end{rem}

\begin{rem}
\label{R:S_over_K}
If $\tilde{P}\in K^{\oplus k}((\tilde{D}^{-1}))$, then the operator $S$ from lemma \ref{L:conjugation in E_q} also belongs to $K^{\oplus k}((\tilde{D}^{-1}))$, cf. remark \ref{R:hatPhi_over_K}.
\end{rem}

Embedding normal forms of differential operators to the ring $E_k$, we can transform them to a simpler form by conjugation:
\begin{lemma}{(\cite[Lem. 3.3]{GZ1})}
\label{L:conjugation in E_q}
Assume $\tilde{P}\in \tilde{K}^{\oplus k}((\tilde{D}^{-1}))\subset E_k$ is a monic operator with $\ord_{\tilde{D}}(\tilde{P})=p$. Let $p=dn$, $q=dm$, where $d=GCD (p,q)$. 

Then there exists a monic invertible operator $S\in \tilde{K}^{\oplus k}((\tilde{D}^{-1}))\subset E_k$ with $\ord_{\tilde{D}}(S)=0$ such that all coefficients (i.e. elements from $\tilde{K}^{\oplus k}$) of the operator $S^{-1}\tilde{P}S$ commute with $\tilde{D}^d$ (we'll call such elements {\it $\tilde{D}^d$-invariant}).
\end{lemma}

\begin{cor}
\label{C:Schur_subring}
Let $B'\subset \tilde{K}^{\oplus k}((\tilde{D}^{-1}))\subset E_q$ be a commutative subring containing $\tilde{D}^q$ and a monic operator $P'$ of order $p$. Suppose $GCD(p,q)=1$. 

Then $S^{-1}B'S\subset \tilde{K}((\tilde{D}^{-1}))$ for an operator $S$ from lemma \ref{L:conjugation in E_q}.
\end{cor} 

Finally, for our proof we'll need a classification of rank one commutative subrings of ordinary differential operators in terms of spectral data and Schur pairs, going back to the works by Burchnall-Chaundy \cite{BC1}-\cite{BC3}, Krichever \cite{Kr1}, Drinfeld \cite{Dr}, Mumford \cite{Mumford_article}, Verdier \cite{Verdier}, Segal-Wilson \cite{SW} and Mulase \cite{Mu}. We'll recall briefly some major results we'll need further following the exposition in \cite{GZ1}. 

Let $B\subset D_1$ be an elliptic commutative subring (i.e. it contains a formally elliptic differential operator; in this case by one Verdier's lemma, all operators from $B$ are formally elliptic). Then by Schur theory from \cite{Schur},  there exists an invertible operator $S=s_0+s_1\partial^{-1} +\ldots $ in the usual (Schur's) ring of pseudo-differential operators $E=K[[x]]((\partial^{-1}))$ such that $A:=S^{-1}BS \subset K((\partial^{-1}))$. Consider the homomorphism (of vector spaces)
\begin{equation}
\label{E:Sato_hom}
E\rightarrow E/xE \simeq K((\partial^{-1}))
\end{equation}
(sometimes it is called {\it the Sato homomorphism}). It defines a structure of a right $E$-module on the space  $K((\partial^{-1}))$: for any $P\in K((\partial^{-1}))$, $Q\in E$ we put  $P\cdot Q= PQ$ (mod $xE$). 

Analogously, the homomorphism 
$$
1\circ : E_q\rightarrow  \tilde{K}((\tilde{D}^{-1})), \quad \sum_l (\sum_{i=0}^{n_l} p_{l_i}\Gamma_1^i)\tilde{D}^l \mapsto  \sum_l p_{l_0,0}\tilde{D}^l
$$
(cf. lemma \ref{L:embedding_Phi}) defines a structure of a right $E_q$-module on the space  $\tilde{K}((\tilde{D}^{-1}))$: for any $P\in \tilde{K}((\tilde{D}^{-1}))$ and $Q\in E_q$ we put $P\cdot Q= 1\circ (PQ)$.

Now define the space $W:= F\cdot S \subset K((\partial^{-1}))$ (here $F$ is the same as in theorem \ref{T:A.Z 2.1}). Note that $W$ is an $A$-module, where the module structure is defined via the multiplication in the {\it field} $K((\partial^{-1}))$ and this module structure is induced by the $E$-module structure on $K((\partial^{-1}))$, because $K((\partial^{-1}))\subset E$ and $W\cdot A=(F\cdot S)\cdot (S^{-1}BS)=F\cdot (BS)=(F\cdot B)\cdot S= F\cdot S= W$. 
Note also that the modules $W$ and $F$ are isomorphic ($W$ is an $A$-module, $F$ as a $B$-module, and clearly $A\simeq B$). For convenience of notation, we will replace $\partial^{-1}$ by $z$ in the field $K((\partial^{-1}))$, i.e. $A,W\subset K((z))\simeq K((\partial^{-1}))$. 

Analogously, we can define the space $W':=F'\cdot S\subset \tilde{K}((\tilde{D}^{-1}))$, where $F'=\tilde{K}[\tilde{D}]$ and $S\in \tilde{K}[A_1]((\tilde{D}^{-1}))$ is an operator from lemma 
\ref{L:conjugation in E_q} or $W'':=F''\cdot S\subset K((\tilde{D}^{-1}))$, where $F''=K[\tilde{D}]$ if $S\in K^{\oplus q}((\tilde{D}^{-1}))$ (for algebraically closed $K$ $F'=F''$). 

If $B'=\hat{\Phi}(S^{-1}BS)$, where $S$ is Schur operator from proposition \ref{T:A.Z 7.2} (constructed for some monic differential operator $Q$ from $B$), and $A':=S^{-1}B'S$, where $S$ is an operator from lemma \ref{L:conjugation in E_q} (constructed for some  operator $P'\in B'$), and, by remark \ref{R:S_over_K} $S\in K^{\oplus q}((\tilde{D}^{-1}))$, so that $A'\subset K((\tilde{D}^{-1}))$, 
then, clearly, the modules $F''$ and $F$ are isomorphic ($F''$ as a $B'$-module, $F$ as a $B$-module), and  
$W''$ and $F$ are isomorphic ($W''$ is a $A'$-module, $F$ as a $B$-module), so also $W''\simeq W$. Moreover, all these modules are isomorphic as filtered modules (with respect to the order filtration).

\begin{rem}
\label{R:left-module_structure}
If we consider another presentation of differential operators from the ring $D_1$, such that the coefficients stand on the right from derivations (i.e. $P=\sum \partial^ia_i$), and analogous presentation of pseudo-differential operators from $E$, then there is a symmetric description of all modules described above as {\it left} modules (i.e. $K((\partial^{-1}))$ will be a left $E$-module, $F$ - a left $B$-module, $W$ - a left $A$-module, etc.). Such a description is used e.g. in \cite{Mu}. Clearly, both descriptions are equivalent. This description, however, could be useful to describe the corresponding Schur pairs and spectral data of {\it adjoint} rings of commuting differential operators in section \ref{S:string}. 
\end{rem}

The {\it rank} of the ring $B$ is, by definition, $\rk B:=GCD\{\ord P, P\in B \}$. For subrings in $K((z))$ we can introduce the same notion of rank as for subrings in $D_1$\footnote{here $z$ will play a role of $\partial^{-1}$ or $\tilde{D}^{-1}$}:

\begin{Def}
\label{D:Q5.1} 
Let $A$ be a $K$-subalgebra of $K((z))$, and $r\in \dn$. $A$ is said to be an algebra of rank $r$ if 
$r= \gcd (\ord(a)|\quad  a\in A)$, where the order is defined in the same way as the usual order  in $D_1$ (cf. remark \ref{R:ord=deg}).
\end{Def}

\begin{Def}
\label{D:support} 
Let $W$ be a $K$-subspace in $K((z))$. The {\it support} of an element $w\in W$ is its highest symbol, i.e. $\sup (w):=HT(w)z^{-\ord (w)}$. The {\it support} of the space is  $\Sup W:=\langle \sup (w)|\quad w\in W\rangle$. 
\end{Def}

\begin{Def}
\label{D:Schur_pair}
An embedded Schur pair of rank $r$  is a pair $(A,W)$ consisting of 
\begin{itemize}
\item
$A\subset K((z))$ a $K$-subalgebra of rank $r$ satisfying $A\cap K[[z]]=K$;
\item 
$W\subset K((z))$ a $K$-subspace with $\Sup W= K[z^{-1}]$
\end{itemize}
such that $W\cdot A\subseteq W$.
\end{Def} 

So, to any elliptic subring $B\subset D_1$ we can associate an embedded Schur pair. 

\begin{Def}
\label{D:equiv_Schue_pair}
Two embedded Schur pairs   $(A_i, W_i)$, $i=1,2$ of rank $r$ are {\it equivalent} if there exists an admissible operator $T$ such that $A_1=T^{-1}A_2T$, $W_1=W_2\cdot T$. An operator $T=t_0+t_1\partial^{-1}+\ldots $ is called {\it admissible} if $T^{-1}\partial T\in K((\partial^{-1}))$. 
\end{Def}

\begin{theorem}
\label{T:classif2}
There is a one-to-one correspondence 
$$
[B\subset {D_1}\mbox{\quad of rank $1$}]/\thicksim  \longleftrightarrow  [(C,p,\cf  )\mbox{\quad of rank $1$}]/\simeq 
$$
where 
\begin{itemize}
\item
$[B]$ means a class of equivalent commutative elliptic subrings (i.e. $B$ containing a monic differential operator),  where   $B\cong B'$ iff $B=f^{-1}B'f$, $f\in D_1^*=K[[x]]^*$. 
\item
$\thicksim$ means "up to a scale automorphism $x\mapsto c^{-1}x$, $\partial\mapsto c\partial$".
\item $C$ is a (irreducible and reduced) projective curve over $K$, $p$ is a regular $K$-point and $\cf$ is a torsion free sheaf of rank one with $H^0(C, \cf)=H^1(C,\cf )=0$ (a spectral sheaf). 
\item  $[]/\simeq$ on the right hand side means an equivalence class of triples up to a natural isomorphism.
\end{itemize}
\end{theorem}

\begin{theorem}
\label{T:classif_Schur_pairs}
There are one-to-one correspondences  $[B]\longleftrightarrow [(A,W)]$ and $[B]/\thicksim\longleftrightarrow [(A,W)]/\thicksim$, where $\thicksim$ means the equivalence from theorem \ref{T:classif2} for rank one data.
\end{theorem}

\begin{rem}
\label{R:classif_ODO}
Each equivalence class $[B]$ contains a {\it normalised representative}, which determines the whole class uniquely. For example, we can choose a monic differential operator of minimal positive order in $B$, and normalise it (conjugating by an appropriate function $f\in \hat{R}$) -- then the ring containing such a normalized operator will be a normalized representative. 

Analogously, in a class $[(A,W)]$ we can consider pairs with normalised space $W$. We'll call a subspace $W$ form a Schur pairs  {\it normalised} if $1\in W$. In this case clearly $A\subset W$. 

Let $C$ be a projective curve over $K$ and $p\in C$ be a regular $K$-point. Then $C_0:=C\backslash p$ is an affine curve. According to theorem \ref{T:classif2} for any torsion free rank one sheaf $\cf$ on $C$ there exists a normalised commutative elliptic ring of differential operators $B\subset D_1$ defined uniquely up to a scale transform, such that $B$ is isomorphic to the ring of regular functions on $C_0$, $B\simeq \co_{C_0}(C_0)$. Vice versa, any such ring $B$ is isomorphic to the ring of regular functions on some affine curve $C_0$, which can be compactified with the help of one regular $K$-point. 

In the case when $B$ is generated by two commuting operators, $B=K[P,Q]$, the affine spectral curve is a plane curve determined by a Burchnall-Chaundy polynomial. Recall the famous Burchnall-Chaundy lemma (\cite{BC1}) says  that any two commuting differential operators $P,Q\in D_1:=K[[x]][\partial ]$ are algebraically dependent. More precisely, if the orders $n,m$ of operators $P,Q$ are coprime (i.e. the rank of the ring $K[P,Q]$ is 1), then there exists an irreducible polynomial $f(X,Y)$ of weighted degree $v_{n,m}(f)=mn$ of {\it special form} (here the weighted degree is defined as in Dixmier's paper \cite{Dixmier}): $f(X,Y) = \alpha X^m\pm Y^n+\ldots $ (here $\ldots$ mean terms of lower weighted degree, $0\neq\alpha\in K$; in particular, for coprime $n$ and $m$ the polynomial $f$ is automatically irreducible), such that $f(P,Q)=0$. A similar result for commuting operators of rank $r$ was established in \cite[Appendix]{Wilson}, in this case $m=\ord (Q)/r$, $n=\ord (P)/r$, and again $GCD (m,n)=1$. More precisely, by theorem in {\it loc. cit.} the polynomial given by the differential resultant $f(X,Y)=\partial Res(P-X,Q-Y)$ gives an algebraic relation between $P$ and $Q$, and $f=h^r$, where $h$ is irreducible and $r=\rk B$. In particular, if $r=1$, then $f$ is irreducible; besides, as it follows from definition of differential resultant, $f(X,Y) = (\alpha X^{\ord (Q)/d}\pm Y^{\ord (P)/d})^d+\ldots $, where $d=GCD \{\ord (Q), \ord (P)\}$, even if the orders of $P,Q$ are not coprime. 
\end{rem}

For the convenience of the reader we formulate the Wilson theorem here, following the exposition of \cite[Th. 2.11]{Previato2019}:

\begin{theorem}
\label{T:Wilson}
 Let $K$ be the field of fractions of the ring $\dc [[x]]$ or $\dc \{x\}$. Given $L, M$ commuting differential operators in $K[\partial]$. Let $r$ be the rank of the algebra $\dc [L, M]$, $f=\partial Res(L-X,M-Y)$ the $BC$-polynomial of $L$ and $M$. 
 
 Then 
 $f=h^r$, where $h$ is the unique (up to a constant multiple) irreducible polynomial satisfied by $L$ and $M$. 
\end{theorem}

\begin{Def}
\label{D:affine_spec_curve}
We'll call an affine curve $C_0$ over $K$ as {\it affine spectral curve} if it can be compactified with the help of one regular $K$-point, i.e. if there exists a projective curve $C$ over $K$ and a regular $K$-point $p$ such that $C_0\simeq C\backslash p$ (note that such $C$ is uniquely defined up to an isomorphism).

We'll call an irreducible polynomial $f(X,Y)\in K[X,Y]$ as a polynomial of Burchnall-Chaundy type if $f$ has the shape 
$$
f(X,Y) = (\alpha X^{n}+ Y^{m})^d+\ldots
$$
as above with coprime $m,n$ and the equation $f(X,Y)=0$ determines an affine spectral curve.\footnote{This is equivalent to a condition that the pull-back of the singular point "at infinity" to the normalisation consists of one point. By \cite[Th. 8.3.1]{BrK} this holds if $f$ is irreducible in the ring $K\{X,Y\}$.} 
\end{Def}

The correspondences in theorem \ref{T:classif_Schur_pairs} are constructed as follows: in one way (from $B$ to $(A,W)$) with the help of the classical Schur operator mentioned above, and in another way via Sato theorems below. 
The following statements are due to M. Sato \cite{SN}, cf. \cite[Appendix]{Mu}. 

\begin{Prop}
\label{P:Sato}
If $G\subset E$ is a subring that stabilises $F=K[\partial ]$, i.e. for each operator $P\in G$ holds $F\cdot P\subseteq F$ (cf. \eqref{E:Sato_hom}), then $G\subseteq D_1$.
\end{Prop}

\begin{theorem}
\label{T:Sato}
Let $W$ be a subspace in the space $K((z))\simeq K((\partial^{-1}))$ with $\Sup W=K[z^{-1} ]$. Then there exists a unique Sato operator, i.e. a zero-th order invertible operator $S=1+s_1\partial^{-1}+\ldots \in E$, such that $W=F\cdot S$. 
\end{theorem}

Normal forms of differential operators of coprime orders can be normalized (in their conjugacy class), e.g. as in the following lemma (such normalisations are of technical nature, they will be used in section \ref{S:string}):

\begin{lemma}{(\cite[Lem. 3.4]{GZ1})}
\label{L:normalisation}
Let $Q\in D_1$ be a normalized operator, let $P\in D_1$ be  a monic differential operator of positive order $p$ and $[P,Q]=0$. Fix $k=q=\Ord (Q)$, suppose $\tilde{K}$ is an algebraic closure of $K$.  Assume $GCD(p,q)=1$. 
For $i=1, \ldots , q-1$ define the sets $N_i:=\{n\in \dn |\quad \mbox{($np$ mod $q$)$\ge i$}\}$. 
  
Then there is a normalised normal form $P'$ of $P$ with respect to $Q$ uniquely defined in the conjugacy class  by elements from the centraliser $C(\partial^q)$. Namely, we can explicitly describe its image under the embedding $\hat{\Phi}$ as follows:
$$
\hat{\Phi}(P')=\tilde{D}^p +\sum_{l=-q+1}^{p-1} p_l \tilde{D}^l, \quad p_l=(p_{l,0},\ldots ,p_{l,q-1})\in \tilde{K}^{\oplus q},
$$
where $p_{l,j}=0$ for indices $j$ satisfying the following rules (all indices are taken mod $q$):
$$
\begin{cases}
\mbox{If $p\ge q$ then} 
\begin{cases}
p_{l,(j-1)p}=\mbox{$0$ for $j\in N_{p-l}$}\quad l>p-q \\
p_{l,j}=\mbox{$0$ for $j=0,\ldots ,-l-1$} \quad l<0
\end{cases}
\\
\mbox{If $p< q$ then}
\begin{cases}
p_{l,(j-1)p}=\mbox{$0$ for $j\in N_{p-l}$}\quad l\ge 0 \\
p_{l,j}=\mbox{$0$ for $j=0,\ldots ,-l-1$ or if $j=(k-1)p$ and $k\in N_{p-l}$}\quad p-q<l<0 \\
p_{l,j}=\mbox{$0$ for $j=0,\ldots ,-l-1$} \quad l\le p-q
\end{cases}
\end{cases}
$$
Let's call coefficients $p_{l,j}$ supplementary to the list above as {\it coordinates} of  $P'$. If $P'$ and $P''$ are two operators with different values of coordinates, then there are no invertible operators $S\in C(\partial^q)$ such that $P'=S^{-1}P''S$, and there are no invertible $1\neq S\in C(\partial^p)$ such that $P'=S^{-1}P'S$.  

\end{lemma}

This normalisation can be used to normalize subrings of rank one in the centralizer:

\begin{Def}
\label{D:normalised_normal_form}
We'll call a commutative ring $B'\in C(\partial^q)\subset \hat{D}_1^{sym}$ generated over $K$ by monic operators $P_1'=\partial^q, P_2', \ldots, P_m'$, where $P_2'$ is normalised and $q$ is coprime with the order of $P_2'$,  as a {\it normalised normal form} of rank one with respect to a (ordered) set of generators $P_1',\ldots , P_m'$. 
\end{Def}

\section{DC-pairs}
\label{S:DC-pairs}

In this section we will work with the first Weyl algebra $A_1$ over the field $\dc$. First recall one definition, cf. \cite[P. 620]{GGV}. 

\begin{Def}
\label{D:subrectangular}
	The Newton Polygon $N$ is called of subrectangular type, if $N$ is contained in the rectangle constructed by $\{(0,0),(0,k),(l,0),(l,k)\}$, and $(l,k)$ is the vertex of $N$ with $l,k \ge 1$. 
	
	Let $\varphi$ be an endomorphism of $A_1$. If the Newton Polygons of $\varphi(\partial),\varphi(x)$ are of subrectangular type, then we call such $\varphi$ of subrectangular type. An operator $P\in A_1$ whose Newton Polygon $N$ is of subrectangular type we'll call  of subrectangular type too.
	
	For an operator $P\in A_1$  of subrectangular type as above, we'll call the monomial $Hm(P):=x^ly^k$ as {\it the highest monomial} of $P$. 
\end{Def}

 It's easy to see that  for any positive pair $(\sigma,\rho)$, $f_{\sigma,\rho}(\varphi(\partial ,x))$ are  monomials, if $\varphi$ is of subrectangular type.
 
\begin{rem}
\label{R:subrectangular}
Note that for two operators $P_1, P_2$ of subrectangular type the operator $P_1P_2$ is also of subrectangular type. Indeed, by theorem \ref{T:Dixmier2.7} we have $f_{1,0}(P_1P_2)= f_{1,0}(P_1)f_{1,0}(P_2)$ and 
$f_{0,1}(P_1P_2)= f_{0,1}(P_1)f_{0,1}(P_2)$, so the Newton polygon contains a monomial $Hm(P_1)Hm(P_2)$, and clearly this monomial is the highest monomial with respect to the weights $v_{1,0}$ and $v_{0,1}$, i.e. $v_{1,0}(P_1P_2)=v_{1,0}(Hm(P_1)Hm(P_2))$ and $v_{0,1}(P_1P_2)=v_{0,1}(Hm(P_1)Hm(P_2))$. So, 
$P_1P_2$ is of subrectangular type with the highest monomial $Hm(P_1)Hm(P_2)$.

\end{rem}
 
 Recall one definition from \cite{GZ}:
 
 \begin{Def}
 \label{D:rate}
Suppose $\varphi$ is an endomorphism of subrectangular type. Let $(\sigma,\rho)$ be any  two positive integer numbers.

Assume $f_{\sigma,\rho}(\varphi(\partial))=c_1 x^{l}y^{k}$, $f_{\sigma,\rho}(\varphi(x))=c_2 x^{l'}y^{k'}$, where $l,k,l',k'$ are non-zero integers, $c_1, c_2\in \dc$. By \cite[L. 2.7]{Dixmier}, $f_{\sigma,\rho}(\varphi(\partial))$ is proportional to $f_{\sigma,\rho}(\varphi(x))$. Define $\epsilon(\varphi)\in \mathbf{Q}$ as
$$
(l',k')=\epsilon(\varphi)(l,k)
$$

We'll call this  $\epsilon(\varphi)$ as the rate of $\varphi(\partial)$ with $\varphi(x)$. Note that $\epsilon(\varphi)$ does not depend on $(\sigma,\rho)$. 
\end{Def}

A pair of operators $P,Q\in A_1$ with the property $[Q,P]=1$ giving a {\it counterexample} to the Dixmier conjecture will be called as a {\it DC-pair} (in spirit of \cite{GGV}). That is, $P=\varphi (x)$, $Q=\varphi (\partial )$ for some $\varphi\in \End (A_1)\backslash \Aut (A_1)$. 

Assume $(P,Q)$ is a DC-pair. Then according to \cite[Prop. 6.9]{GGV} we can assume without loss of generality that $P,Q$ and the corresponding endomorphism $\varphi$ are of subrectangular type and that $P,Q$ is a minimal such pair. Then $Hm(P)=x^ay^p$, for some $a,p\in \dn$. By \cite[Th. 4.1 (3)]{GGV} $a\neq p$. Without loss of generality (applying the Fourier transform if necessary) we can assume $a<p$. Besides, applying if necessary a generic linear change of variables $x\mapsto x+\varepsilon$, $\partial \mapsto \partial$ we can  assume without loss of generality that $\ord (P)=\Ord (P)=p$ and $\ord (Q)= \Ord (Q)=q$. In particular, $HT(P)\in \dc [x]$ is a polynomial with non-zero free term: $HT(P)(0)\neq 0$.

Put  $p=dn$, $q=dm$, where $(n,m)=1$\footnote{Here and further we assume that $(n,m)=1$ means $GCD(m,n)=1$ and $m,n>1$.}, $d=GCD(p,q)$.  By \cite[L. 2.7]{Dixmier} the highest terms of $P$ and $Q$ are proportional: $f_{0,1}(P)^m=\tilde{c} f_{0,1}(Q)^n$ for some $\tilde{c}\in \dc^*$. In particular, $n| \deg_x(f_{0,1}(P))$. Assume  $\deg_x(f_{0,1}(P))=nl$. Denote by $h\in \dc [x]$ the coefficient at $y^p$ in $f_{0,1}(P)$: $f_{0,1}(P)=hy^p$. Note that $h=g^n$ for some $g\in \dc [x]$ with $\deg_x g=l$. 
Since $nl=a<p=nd$, we must have $d>1$ and $l<d$. Analogously, $Hm(Q)=x^{ml}y^{md}$ and the rate $\epsilon (\varphi )=n/m$. Since $P,Q$ is minimal, $\epsilon \notin \dn$. 

\begin{Def}
\label{D:essGCD}
Let $\varphi$ be an endomorphism of subrectangular type, $P=\varphi (x)$, $Q=\varphi (\partial )$, with highest monomials $Hm(P)=x^{ln}y^{dn}$, $Hm(Q)=x^{lm}y^{dm}$, where $d>1$, $l<d$, $(n,m)=1$ (so, $d= GCD(\ord (P),\ord (Q))$, $l=GCD(\ord_x(P),\ord_x(Q))$). 

Define {\it essential GCD of $P$, $Q$} as 
$$
EssGCD(P,Q):= \frac{d}{GCD(l,d)}>1.
$$
\end{Def}

Recall one lemma from \cite{GZ}:

\begin{lemma}{(\cite[Lem. 3.2]{GZ})}
\label{L:9}
	Suppose $\varphi$ is an endomorphism of subrectangular type, $\epsilon = \epsilon (\varphi )$. 
	
	Assume $P\in A_{1}$, $P\notin \dc$ is  such that $f_{\epsilon,1}(P)$ is a monomial. Then  for any $\sigma,\rho >0$ the polynomial $f_{\sigma,\rho}(\varphi(P))$ is a monomial with 
\begin{equation} \label{E:vi}
	v_{\sigma,\rho}(\varphi(P))=(\rho k+\sigma l)v_{\epsilon,1}(P), 
\end{equation}
where $f_{\sigma,\rho}(\varphi(\partial))=c_1x^{l}y^{k}$, $f_{\sigma,\rho}(\varphi(x))=c_2x^{l\epsilon}y^{k\epsilon}$.
 \end{lemma}

\begin{cor}
\label{C:9}
Under assumptions of lemma let $f_{\epsilon,1}(P)=cx^uy^v$, $c\in \dc^*$. Then the Newton polygon of $\varphi(P)$ is of subrectangular type and $Hm(\varphi(P))=x^{l(u\epsilon +v)}y^{k(u\epsilon +v)}$.
\end{cor}

\begin{proof}
Just note that for any monomial $x^{u'}\partial^{v'}$ we have $\varphi(x^{u'}\partial^{v'})=\varphi(x)^{u'}\varphi(\partial)^{v'}$ and by  \cite[L. 2.7]{Dixmier} we have therefore 
 $$
 f_{\sigma ,\rho}(\varphi(x^{u'}\partial^{v'}))=\tilde{c}x^{l(u'\epsilon+v')}y^{k(u'\epsilon +v')}
 $$
for any positive $\sigma ,\rho$ and some $\tilde{c}\in \dc^*$. 
For all non-zero monomials $c'x^{u'}\partial^{v'}$ from $P$ we have, by assumption, 
$$
v_{\epsilon,1}(c'x^{u'}\partial^{v'})=u'\epsilon+v'\le v_{\epsilon,1}(cx^{u}\partial^{v})=u\epsilon+v,
$$
and the equality holds only for the monomial $cx^{u}\partial^{v}$, 
i.e. $v_{\epsilon,1}(P)= u\epsilon+v$ and moreover, the monomial $x^{l(u\epsilon+v)}y^{k(u\epsilon +v)}$ belongs to the Newton polygon of $\varphi(P)$. 

From \eqref{E:vi} we have $\ord_x(\varphi(P))\le lv_{\epsilon,1}(P)$ and $\ord (\varphi(P))\le kv_{\epsilon,1}(P)$, i.e. the Newton Polygon of $\varphi(P)$ is contained in the rectangle constructed by $\{(0,0),(0,k(u\epsilon +v)),(l(u\epsilon+v),0),(l(u\epsilon+v),k(u\epsilon +v))\}$. Since the vertex $(l(u\epsilon+v),k(u\epsilon +v))$ belongs to the Newton Polygon of $\varphi(P)$, we are done.
\end{proof}

\begin{lemma}
\label{L:order_of_polynomial}
Let $\varphi$ be an endomorphism of subrectangular type, $\epsilon =\epsilon (\varphi)$, $P=\varphi (x)$, $Q=\varphi (\partial )$, with highest monomials $Hm(P)=x^{ln}y^{dn}$, $Hm(Q)=x^{lm}y^{dm}$, where $d>1$, $l<d$, $(n,m)=1$.

Let $\{p_1,\ldots , p_I\}$, $\{q_1,\ldots , q_J\}$ be the set of non-zero coefficients at the monomials of $P,Q$ respectively. Let $F(X,Y)=\sum_{i,j}\alpha_{ij}X^iY^j\in \dq (p_1,\ldots ,p_I,q_1,\ldots ,q_J)[X,Y]$ be a non-zero polynomial with coefficients in the field $K:=\dq (p_1,\ldots ,p_I,q_1,\ldots ,q_J)\subset \dc$. 

Let $\Phi_{N,\lambda}$ be a tame automorphism of $A_1$ such that $\lambda\in \dc$ is not algebraic over the field $K$ and $N> \epsilon$. 

Then the operator $\varphi\circ \Phi_{N,\lambda }(F(P,Q))\in A_1$ is of subrectangular type\footnote{Here $F(P,Q):=\sum_{i,j}\alpha_{ij}X^iY^j|_{X\mapsto P, Y\mapsto Q}$.} and 
$$
\ord (\varphi\circ \Phi_{N,\lambda }(F(P,Q)))=dm v_{\epsilon ,1}(\Phi_{N,\lambda }(F(P,Q)))=dm \ord (\Phi_{N,\lambda }(F(P,Q))).
$$
\end{lemma}

\begin{proof}
First note that for any monomial $x^l\partial^k$ the operator $\Phi_{N,\lambda }(x^l\partial^k)$ has the $(N,1)$-top line trough the point $(l,k)$, and this line contains a non-zero monomial on the vertical $y$-axis, whose coefficient is a polynomial in $\lambda$ (maybe of zero degree) with integer coefficients. 

Let $l_{N,1}(F(P,Q))$ be the $(N,1)$-top line of the operator $F(P,Q)$. Since all coefficients of monomials belonging to $l_{N,1}(F(P,Q))$ are from $K$, the $(N,1)$-top line $l_{N,1}(\Phi_{N,\lambda }(F(P,Q)))$ of the operator $\Phi_{N,\lambda }(F(P,Q))$ contains a non-zero monomial on the vertical $y$-axis, say $c_1y^M$ (because all monomial on $l_{N,1}(F(P,Q))$ have different degrees in the variable $x$ and $\lambda$ is not algebraic over $K$).

Then we have $v_{N,1}(c_1\partial^M)=M=v_{\epsilon ,1}(c_1\partial^M)$, but since $N>\epsilon$, for all other monomials of the line $l_{N,1}(\Phi_{N,\lambda }(F(P,Q)))$ we have 
$$
M=v_{N,1}(x^l\partial^k)=Nl+k> \epsilon l+k=v_{\epsilon ,1}(x^l\partial^k),
$$
and for all other monomials of $\Phi_{N,\lambda }(F(P,Q))$ away from the line $l_{N,1}(\Phi_{N,\lambda }(F(P,Q)))$ we have 
$$
M>v_{N,1}(x^{l'}\partial^{k'})=Nl'+k'\ge \epsilon l'+k'=v_{\epsilon,1}(x^{l'}\partial^{k'}).
$$
So, the conditions of lemma \ref{L:9} are satisfied ($f_{\epsilon ,1}(\Phi_{N,\lambda }(F(P,Q)))=c_1y^M$ is a monomial) and therefore by corollary \ref{C:9} the operator $\varphi\circ \Phi_{N,\lambda }(F(P,Q))$ is of subrectangular type and 
$$
\ord (\varphi\circ \Phi_{N,\lambda }(F(P,Q)))=dm v_{\epsilon ,1}(\Phi_{N,\lambda }(F(P,Q)))=dm \ord (\Phi_{N,\lambda }(F(P,Q))).
$$
\end{proof}

\begin{rem}
\label{R:order_of_polynomial}
Note that the assertion of lemma \ref{L:order_of_polynomial} is true also for $F(X,Y)\in \dc [X,Y]$ and any $\lambda \neq 0$ if it is known that $F(P,Q)$ is of subrectangular type. Indeed, in this case the $(N,1)$-top line of $F(P,Q)$ has a unique vertex, and therefore the $(N,1)$-top line of the operator $\Phi_{N,\lambda }(F(P,Q))$ contains a non-zero monomial on the vertical $y$-axis. The rest of the proof is the same. 

In fact, we'll use lemma \ref{L:order_of_polynomial} exactly for this situation. However, we decided to formulate this lemma in a more general form for the convenience of further references in future works.
\end{rem}

\begin{cor}
\label{C:order_of_polynomial}
In the assumptions of remark \ref{R:order_of_polynomial} (i.e. $\varphi$ satisfies assumptions of lemma \ref{L:order_of_polynomial}, $N>\epsilon$ and $\lambda\in \dc^*$) put $\hat{P}:=\varphi\circ \Phi_{N,\lambda}(P)$, $\hat{Q}:=\varphi\circ \Phi_{N,\lambda}(Q)$. Then 
$$
\ord (\hat{P})=dm(nd+Nln), \quad \ord (\hat{Q})=dm(md+Nlm)
$$
and the operators $\hat{P}, \hat{Q}$ are of subrectangular type with $EssGCD(\hat{P},\hat{Q})=EssGCD(P,Q)$.
\end{cor}

\begin{proof}
By remark \ref{R:order_of_polynomial} and lemma we know that $\hat{P}, \hat{Q}$ are of subrectangular type with $$Hm(\hat{P})=x^{lm(nd+Nln)}y^{dm(nd+Nln)}, \quad Hm(\hat{Q})=x^{lm(md+Nlm)}y^{dm(md+Nlm)}.$$
Hence 
$$
GCD(\ord (\hat{P}),\ord (\hat{Q}))=dm (d+Nl), \quad GCD(\ord_x (\hat{P}),\ord_x (\hat{Q}))=lm (d+Nl)
$$
and therefore 
$$
EssGCD(\hat{P},\hat{Q})=\frac{dm(d+Nl)}{m(d+Nl)GCD(l,d)}=\frac{d}{GCD(l,d)}=EssGCD(P,Q).
$$
\end{proof}

\section{The main preparatory theorems}
\label{S:preparatory}

The aim of this section is to prove the following theorem and a corollary -- theorem \ref{T:subrectangular_sequence}. In this section we will work with the first Weyl algebra $A_1$ over the field $\dc$ (it may be assumed to be any other algebraically closed field of characteristic zero).

\begin{theorem}
\label{T:DC-pairs}
Assume $(P,Q)$ is a subrectangular DC-pair (i.e. $P=\varphi (x)$, $Q=\varphi (\partial )$ for some endomorphism $\varphi$ of $A_1$ of subrectangular type) with $\ord (P)=p$ and $\ord (Q)= q$. Assume $p=dn$, $q=dm$, $d>1$, $(n,m)=1$, 
$Hm(P)=x^{ln}y^{dn}$, $l<d$, and $\ord (P)=\Ord (P)$, $\ord (Q)=\Ord (Q)$. Then $Hm(Q)=x^{ml}y^{md}$ and the rate $\epsilon (\varphi )=n/m$ (see section \ref{S:DC-pairs}). 

Let $\{p_1,\ldots , p_I\}$, $\{q_1,\ldots , q_J\}$ be the set of non-zero coefficients at the monomials of $P,Q$ respectively. Put $K:=\dq (p_1,\ldots ,p_I,q_1,\ldots ,q_J)\subset \dc$. 

Then there exist sequences of natural numbers $n_0,\ldots ,n_I$, $m_0,\ldots ,m_I$, a sequence of constants $\varepsilon_0, \ldots ,\varepsilon_I\in K$\footnote{This statement about the field of definition of constants is not essential for the proof of theorem \ref{T:main}, however, we decided to formulate this theorem in a more general form for the convenience of further references in future works.

The numbers $n_i,m_i,\epsilon_i$ are not uniquely defined.}, and a sequence of polynomials $F_i(X,Y)\in K[X,Y]$, $i=0, \ldots ,I$ defined by recursion as $F_{i+1}(X,Y)= F_{i}(X,Y)^{n_i}-\varepsilon_i^{n_i}X^{m_i}$, where $F_0(X,Y):=Y$\footnote{These polynomials determine also a sequence of operators in $A_1$ after substitution $X\mapsto P$, $Y\mapsto Q$. We'll denote this sequence also as $F_i(P,Q)$.}, such that 
$$
d_2:=EssGCD(P,Q) \quad | \quad GCD\{\ord (P), \ord (F_0(P,Q)), \ldots , \ord (F_i(P,Q))\} \quad | \quad d, \quad \mbox{for $i<I$},
$$
$$
\ord (F_I(P,Q))=1 \mbox{\quad mod\quad } GCD\{\ord (P), \ord (F_0(P,Q)), \ldots , \ord (F_{I-1}(P,Q))\}.
$$
Moreover, we have the equality 
\begin{equation}
\label{E:distance}
\sum_{i=0}^{I-1} (n_i\ord F_i(P,Q)-\ord F_{i+1}(P,Q))=p+q-1,
\end{equation}
and for all $i=0, \ldots ,I$ we have $\ord F_i(P,Q)=\Ord F_i(P,Q)$. 
\end{theorem} 

\begin{rem}
\label{R:DC-pais}
In fact, the proof of this theorem will provide also an alternative proof of the following result, which is already known:\footnote{I learned from Leonid Makar-Limanov that it can be proved with the help of the same technique that was used in particular in papers \cite{Joseph}, \cite{GGV2}  to show that the Newton polygon $N (f)$ of a JC-pair $(f,g)$ (giving a counterexample to the JC conjecture for $K[x,y]$) can be assumed to be included in a trapezium with
one vertex $v$, two edges parallel to the $y$-axis and to the bisectrix of the first quadrant adjacent to $v$, and two edges belonging to the coordinate axes, cf. \cite{ML2} and references therein. I'd like to thank Christian Valqui who kindly informed me  how this result can be deduced from  their work \cite{GGV2}.} let $f_{0,1}(P)=g^ny^p$, where $g\in K[x]$ is a polynomial with $g(0)\neq 0$. 

Then $g$ has only one root, i.e. $g=(ax+b)^l$ for some $a,b\in K$\footnote{Besides, the proof of our theorem works also in the case $l=d$, thus we get an alternative proof of \cite[Th. 4.1 (3)]{GGV}.}. Moreover, from the proof we'll see that for any $i=0, \ldots ,I$ $f_{0,1}(F_i(P,Q))=(ax+b)^{l_i}y^{d_i}$ for some $l_i,d_i\in \dn$. 

Besides, the theorem holds also in the case $l>d$, if $d$ does not divide $l$. The last assumption is necessary for the condition $d_2>1$; it is also essential in lemma \ref{L:polynomials}. 
\end{rem}

\begin{proof}
Because of our assumptions on $P,Q$ (namely, that $g(0)\neq 0$) there exists an automorphism $\Phi_u$ of the ring $D_1=K[[x]][\partial ]$  of the form 
\begin{equation}
\label{E:Phi_u}
\Phi_u (x)=u\in K[[x]], \quad \Phi_u (\partial )=\frac{1}{u'}\partial +v,
\end{equation} 
where $u(0)\neq 0$, $u'(0)\neq 0$ and $v\in K[[x]]$, such that $P':=\Phi_u (P)\in D_1$\footnote{Starting from this point onwards we use a different (from section \ref{S:prelim}) notation for a normal form of $Q$ with respect to $P$: instead of $Q'$ we'll use $\tilde{Q}$.} is normalized, i.e. $P'=\partial^p+b_{p-2}\partial^{p-2}+\ldots +b_0$, see e.g. \cite[Prop. 1.3]{BZ}. Then, obviously, $Q':=\Phi_u (Q)= c\partial^q +l.o.t.$ for some $c\in K$. Since $[\varepsilon Q',\varepsilon^{-1} P']=1$ for any $\varepsilon\in \dc^*$, composing, if necessary, $\Phi_u$ with an appropriate scale transform $\partial \mapsto \varepsilon \partial$, $x \mapsto \varepsilon^{-1}x$, such that $c\varepsilon^{q+p}=1$,  taking the DC-pair $\varepsilon^{-p}P, \varepsilon^{p}Q$ instead of $(P,Q)$ (and extending $K$ by $c^{-1/(p+q)}$), we can assume without loss of generality that $P',Q'$ are both monic differential operators. Observe that 
\begin{equation}
\label{E:Phi_u^{-1}}
\Phi^{-1}_u (\partial)= g^{1/d}\partial +v'
\end{equation} 
for an appropriate $v'\in K[[x]]$, and $\Phi^{-1}_u (x)=f\in K[[x]]$, where $f':=\partial (f)=g^{-1/d}$, as $[\Phi^{-1}(\partial ),f]=1$, and $f(0)=0$. Note that $\Phi_u$ preserves the order filtration on $D_1$, i.e. $\ord (\Phi_u (Y))=\ord (Y)$ for any $Y\in D_1$. 

{\it Idea of proof:} we use the theory of normal forms from section \ref{S:prelim} and calculate first a  normal form (any normal form) of the operator $Q'$ with respect to $P'$. This normal form has a very simple structure -- it is a sum of an operator from the centraliser $C(\partial^p)\subset \hat{D}_1^{sym}$ of $\Ord$-order $q$ and a homogeneous operator of $\Ord$-order $-p$, which is uniquely determined and whose shape can be explicitly calculated (let's call it the "tail"). Constructing by induction the operators $F_i(P,Q)$ and using specific properties of the Schur operator for $P'$ (the condition $A_p(0)$ and conditions on the first two homogeneous components $S_{0}, S_{-1}$) as well as the Dixmier lemma \cite[L. 2.7]{Dixmier} (theorem \ref{T:Dixmier2.7}), we show that the orders of operators $F_i(P,Q)$ satisfy the property claimed in the theorem. For the recursive construction of operators it is important that original  operators, $P,Q$, are of subrectangular type: first, this is necessary for application of the Dixmier lemma, second, this condition is used to show that the inductive process stops. The crucial step here is a specific differential equation on polynomials (see lemma \ref{L:polynomials}). 

The proof is divided into 3 steps: in step 1 we calculate the "tail" of the normal form of $Q'$ with respect to $P'$, in step 2 we construct $F_1$ and prove its properties needed for the induction procedure, and in step 3 we continue the inductive procedure and construct other operators $F_i$. 

More precisely, to construct $F_i$, we first consider a normal form $\tilde{Q}$ of $Q'$ with respect to $P'$ (i.e. we choose some Schur operator $S$ for $P'$) and then  construct in steps 2 and 3, by the recursive procedure of "head chopping", the operators $\tilde{U}_i$, which are the normal forms  of the operators $F_i(P',Q')$ with respect to $P'$. Using the specific properties of the Schur operator for $P'$ (which hold because $P'$ is normalized), we can check, by induction, the corresponding properties of orders of the operators $\tilde{U}_i$. Again due to the specific properties of the Schur operator, by conjugating back the operators $\tilde{U}_i$ and applying the automorphism $\Phi_u^{-1}$, we obtain the needed properties of orders of the operators $F_i(P,Q)$ (as such a conjugation and the automorphism preserve orders) at the end of step 3. 

\smallskip

{\it Step 1.} By proposition \ref{T:A.Z 7.2} there exists a Schur operator $S\in \hat{D}_1^{sym}$ with $\Ord (S)=0$ such that $P'=S^{-1}\partial^pS$ and $S_0=1$, $S_{-1}=0$. Then by corollary \ref{C:Si^(-1) no B} we have also $S^{-1}_0=1$, $S_{-1}^{-1}=0$. Put  $\tilde{P}:=SP'S^{-1}=\partial^p$, $\tilde{Q}:= SQ'S^{-1}$. Then by proposition \ref{T:A.Z 7.1}, item 2  the homogeneous decomposition of $\tilde{Q}$ can be written as 
$$
\tilde{Q}= \partial^q + \tilde{Q}_{q-1}+\ldots + \tilde{Q}_{-p},
$$
where $[\tilde{Q}_j, \partial^p]=0$ for $-p+1\le j\le q-1$ and therefore $\tilde{Q}_j=(\sum_{i=0}^{p-1}c_{ij}A_{p;i})D^j$, where $c_{ij}\in \tilde{K}=K[\xi]$, and $[\tilde{Q}_{-p}, \partial^p]=1$. 

\begin{lemma}
\label{L:Q_p}
In the notation above we have 
$$
\tilde{Q}_{-p}=(-\frac{1}{p} x\partial +\sum_{i=0}^{p-1} e_iA_{p;i})\int^p,
$$
where $e_k=\frac{1}{p(\xi^{-k}-1)}$ for $k=1, \ldots , p-1$, $e_0=\frac{p-1}{2 p}$. Here $\xi$ is a primitive $p$-th root of unity.

The vector form $\hat{\Phi}(\tilde{Q}_{-p})$ has the shape
$$
\hat{\Phi}(\tilde{Q}_{-p})=(-\frac{1}{p} \Gamma_1 +(0,\frac{1}{p}, \frac{2}{p}, \ldots , \frac{p-1}{p}))\tilde{D}^{-p}.
$$
\end{lemma}

\begin{proof}
First note that by proposition \ref{T:A.Z 7.1}  the operator $\tilde{Q}_{-p}$ is uniquely defined by the property $[\tilde{Q}_{-p}, \partial^p]=1$ if it exists. Let's show that the operator from formulation satisfies this property. We have by lemma \ref{L:1} 
\begin{multline*}
[\tilde{Q}_{-p}, \partial^p]=1+\frac{1}{p} x\partial (B_1+\ldots + B_p)-(\sum_{i=0}^{p-1} e_iA_{p;i}) (B_1+\ldots + B_p)=\\
1+\frac{1}{p}(B_2+2B_3+\ldots + (p-1)B_p) - (\sum_{i=0}^{p-1} e_i) B_1- (\sum_{i=0}^{p-1} e_i\xi^i) B_2-\ldots - (\sum_{i=0}^{p-1} e_i\xi^{(p-1)i}) B_p=\\
1- (\sum_{i=0}^{p-1} e_i) B_1 -(\sum_{i=0}^{p-1} e_i\xi^i-\frac{1}{p}) B_2-\ldots - (\sum_{i=0}^{p-1} e_i\xi^{(p-1)i} -\frac{p-1}{p}) B_p
\end{multline*}
Hence we need to show that $(e_0,\ldots , e_{p-1})$  is a solution of the linear system 
$$
\left (
\begin{array}{cccc}
1&1&\ldots &1\\
1&\xi& \ldots & \xi^{p-1}\\
\vdots & \vdots &\ddots & \vdots\\
1&\xi^{p-1}& \ldots & \xi^{(p-1)^2}
\end{array}
\right ) \left (
\begin{array}{c}
e_0\\
e_1\\
\vdots\\
e_{p-1}
\end{array}
\right ) = 
\left (
\begin{array}{c}
0\\
\frac{1}{p}\\
\vdots\\
\frac{p-1}{p}
\end{array}
\right )
$$
with the Vandermonde matrix $W$. It is well known that the inverse of the Vandermonde matrix $W^{-1}=(b_{ij})$, where $b_{ij}=\frac{1}{p}\xi^{(1-i)(j-1)}$, $i,j =1,\ldots ,p$. Therefore for $k=1,\ldots , p$ we have
$$
e_{k-1}=\sum_{j=1}^{p}b_{kj}\frac{j-1}{p}= \frac{1}{p^2}\sum_{j=1}^{p}\xi^{(1-k)(j-1)}(j-1)=\frac{1}{p^2} (1+x+\ldots + x^p)'|_{x=\xi^{1-k}}=\frac{1}{p(\xi^{-k}-1)},
$$
and $e_0=\frac{p-1}{2 p}$ as stated.

The last claim is straightforward.
\end{proof}

{\it Step 2.} (The first induction step) Recall the notation: $\hat{R}=K[[x]]$.

\begin{lemma}
\label{L:induction_step}
In the notation above put $\tilde{U}_1:=\tilde{Q}^n-\tilde{P}^m$. Then we have
\begin{enumerate}
\item
$U_1':=S^{-1}\tilde{U}_{1}S\in D_1$ and $\Phi_u^{-1}(U_1')\in A_1$;
\item
the operator $[\tilde{P},\tilde{U}_1]$ satisfies condition $A_p(0)$;
\item
$\Ord (\tilde{U}_1)> q(n-1)-p$;
\item
$\Ord (\tilde{U}_1)> q(n-1)-p+1$ if the polynomial $g$ has more than one root. 

Moreover, if $g$ has only one root and $\Ord (\tilde{U}_1)= q(n-1)-p+1$, then $\sigma (\tilde{U}_1)=\varepsilon_1 \partial^{\Ord (\tilde{U}_1)}$ with 
$$
\varepsilon_1 = \frac{\alpha^{-1}}{d}(1-\frac{l}{d}),
$$  
where $g=(1+\alpha x)^l$, and
$$
U_1'=a_{q(n-1)-p+1}\partial^{q(n-1)-p+1} +a_{q(n-1)-p}\partial^{q(n-1)-p}+ \ldots +a_0
$$
with $a_{q(n-1)-p+1}=\varepsilon_1 +\frac{1}{d} x$ and $a_i\in \hat{R}$ for $i<q(n-1)-p+1$.
\item
If $\Ord (\tilde{U}_1)> q(n-1)-p+1$, then the operator $\tilde{U}_1$ satisfies condition $A_p(0)$;
\item
if $\Ord (\tilde{U}_1)> q(n-1)-p+1$, then the number $\Ord (\tilde{U}_1)$ is divisible by $d_2:=\frac{d}{GCD(l,d)} >1$; in particular, $d_2$ divides $GCD(\Ord (\tilde{P}), \Ord (\tilde{Q}), \Ord (\tilde{U}_1))$ (and this GCD divides $d$).
\item
$Sdeg_A((\tilde{U}_1)_{j})\le 0$ for $\Ord (\tilde{U}_1)\ge j>q(n-1)-p$ and $Sdeg_A((\tilde{U}_1)_{q(n-1)-p})=1$. Besides, 
$$
(\tilde{U}_1)_{q(n-1)-p}=c_1 x\partial^{q(n-1)-p+1}+c_{0} \partial^{q(n-1)-p},
$$
where $c_1\in K$ and $c_{0}$ is a HCP from $Hcpc (p)$ with $Sdeg_A(c_{0})\le 0$ and $Sdeg_B(c_0)=-\infty$.
\item
$M_1:=\Ord (\tilde{U}_1)-(q(n-1)-p)< q+p =: M_0$.
\end{enumerate}

\end{lemma}

\begin{proof} {\it Item} 1. is obvious ($U_1'=(Q')^n-(P')^m$). 

{\it Item} 2. We have 
$$
[\tilde{P}, \tilde{U}_1]=[\tilde{P}, \tilde{Q}]\tilde{Q}^{n-1}+\tilde{Q}[\tilde{P}, \tilde{Q}]\tilde{Q}^{n-2}+\ldots + \tilde{Q}^{n-1}[\tilde{P}, \tilde{Q}]=-n \tilde{Q}^{n-1}.
$$
Therefore, $\Ord ([\tilde{P}, \tilde{U}_1])= q(n-1)$. Since $\tilde{Q}_q=\partial^q$, we have $(\tilde{Q}^{n-1})_{q(n-1)}=\partial^{q(n-1)}$ and therefore $Sdeg_A([\tilde{P}, \tilde{U}_1]_{q(n-1)})=0$. By lemma \ref{L:total_B} all homogeneous components of $[\tilde{P}, \tilde{U}_1]$ are totally free of $B_j$, because $\tilde{Q}$ satisfies condition $A_p(0)$ by proposition \ref{C:P'}. By lemma \ref{L:conditionA}  $\tilde{Q}^{n-1}$ satisfies condition $A_p(0)$, i.e. the operator $[\tilde{P},\tilde{U}_1]$ satisfies condition $A_p(0)$.

{\it Items} 3., 4.  From item 2. we immediately get $\Ord (\tilde{U}_1)\ge q(n-1)-p$. Assume that $\Ord (\tilde{U}_1)\le q(n-1)-p+1$, i.e. $\Ord (\tilde{U}_1)= q(n-1)-p+1$ or $\Ord (\tilde{U}_1)= q(n-1)-p$. 

1). By item 2. we know that $Sdeg_A([\tilde{P}, \tilde{U}_1]_{q(n-1)-t})<t$ for all $t>0$. Therefore,
$$
Sdeg_A(\tilde{U}_1)_{q(n-1)-p-t}< t+1
$$
for all $t>0$.

2). Note that $\Ord (U_1')=\Ord (\tilde{U}_1)$ (by remark \ref{R:ord_properties}). In the first case (i.e. $\Ord (\tilde{U}_1)= q(n-1)-p+1$) let's prove that 
$$(U_1')_{q(n-1)-p+1}=(\tilde{U}_1)_{q(n-1)-p+1}=\varepsilon_1 \partial^{q(n-1)-p+1},$$ 
where $0\neq \varepsilon_1\in K$. Indeed, for all $j>q(n-1)-p$ we have
\begin{equation}
\label{E:@@}
(\tilde{Q}^n)_j= \sum_{i_1+\ldots +i_n=j}\tilde{Q}_{i_1}\cdots \tilde{Q}_{i_n},
\end{equation}
where $i_k\le q$. Since $j>q(n-1)-p$, we have $i_k>-p$ for all $k$. Therefore, $Sdeg_A(\tilde{Q}_{i_k})\le 0$ for all $k$ and by lemma \ref{L:HCPC} we get $Sdeg_A((\tilde{Q}^n)_j)\le 0$ for all $j>q(n-1)-p$.
Therefore $Sdeg_A((\tilde{U}_1)_j)\le 0$ for all $j>q(n-1)-p$ too. Since $\Ord (\tilde{U}_1)= q(n-1)-p+1$, we must have $Sdeg_A((\tilde{U}_1)_{q(n-1)-p+1})=0$. 

Now note that 
$(\tilde{U}_1)_{q(n-1)-p+1}=(\tilde{Q}^n)_{q(n-1)-p+1}=(U_1')_{q(n-1)-p+1}\in D_1$  because $S^{-1}_{0}=S_{0}=1$, i.e. it is a differential operator, and by lemmas \ref{T:A.Z L 7.2}, \ref{L:correct_HPC} this HCP does not depend on $A_{p;i}$, i.e. it looks like $\varepsilon_1 \partial^{q(n-1)-p+1}$ for $\varepsilon_1\in K$. 

3). By lemma \ref{L:Q_p} we have
\begin{multline}
\label{E:@@@}
(\tilde{U}_1)_{q(n-1)-p}=(\tilde{Q}^n)_{q(n-1)-p}=\tilde{Q}_{-p}\partial^{q(n-1)}+\partial^q \tilde{Q}_{-p}\partial^{q(n-2)}+\ldots + \partial^{q(n-1)}\tilde{Q}_{-p}+ \\
\sum_{i_1+\ldots +i_n=q(n-1)-p \atop i_1,\ldots ,i_n>-p}\tilde{Q}_{i_1}\cdots \tilde{Q}_{i_n}=
\frac{1}{d} x\partial^{q(n-1)+1-p}+\\
(n\sum_{j=0}^{d-1}e_{nj}A_{p;nj}) \partial^{q(n-1)-p} +
\sum_{i_1+\ldots +i_n=q(n-1)-p \atop i_1,\ldots ,i_n>-p}\tilde{Q}_{i_1}\cdots \tilde{Q}_{i_n},
\end{multline}
where the last equality holds because the HCP $(\tilde{Q}^n)_{q(n-1)-p}$ is totally free of $B_j$. Note that all terms in the last raw of this formula have $Sdeg_A\le 0$ by lemma \ref{L:HCPC}. 

4). Put $\tilde{\tilde{U}}_1:=\tilde{U}_1-\frac{1}{d} x\partial^{q(n-1)+1-p}$. In the first case (i.e. $\Ord (\tilde{U}_1)= q(n-1)-p+1$) this operator satisfies condition $A_p(0)$. Indeed, conditions 1,2 from definition \ref{D:conditionA} are obvious, the symbol $\sigma (\tilde{\tilde{U}}_1)$ does not contain $A_{p;i}$ and $Sdeg_A(\sigma (\tilde{\tilde{U}}_1))=0$ by 2), and $Sdeg_A((\tilde{\tilde{U}}_1)_{q(n-1)-p+1-t})<t$ for all $t>0$ by 3), 1).

Then 
$$
U_1'=S^{-1}\tilde{\tilde{U}}_1S+ \frac{1}{d} S^{-1} x\partial^{q(n-1)+1-p} S.
$$
By lemma \ref{L:conditionA} the operator $S^{-1}\tilde{\tilde{U}}_1S$ satisfies condition $A_p(0)$ and 
$S^{-1} x\partial^{q(n-1)+1-p} S$ satisfies condition $A_p(1)$. Therefore, by lemma \ref{L:obvious} 
$
Sdeg_A(U_1')_{q(n-1)-p-t}<t+1
$
for all $t>0$, and therefore 
$$
\ord (U_1')_{q(n-1)-p-t}= q(n-1)-p-t +Sdeg_A(U_1')_{q(n-1)-p-t}< q(n-1)-p +1
$$
for all $t>0$. On the other hand, since $(S^{-1})_{-1}=S_{-1}=0$ and $(S^{-1})_{0}=S_{0}=1$, we have 
$$
(U_1')_{q(n-1)-p}=\frac{1}{d} x\partial^{q(n-1)+1-p} + (\tilde{\tilde{U}}_1)_{q(n-1)-p},
$$
and therefore 
$$
U_1'=a_{q(n-1)-p+1}\partial^{q(n-1)-p+1} +a_{q(n-1)-p}\partial^{q(n-1)-p}+ \ldots +a_0
$$
with $a_{q(n-1)-p+1}=\varepsilon_1 +\frac{1}{d} x$ and $a_i\in \hat{R}$ for $i<q(n-1)-p+1$.

5). In the second case (i.e. $\Ord (\tilde{U}_1)= q(n-1)-p$) by steps 1), 3) we get for all $t\ge 0$
$$
Sdeg_A(\tilde{\tilde{U}}_1)_{q(n-1)-p-t}< t+1.
$$
Consider $U_1^*:=\tilde{\tilde{U}}_1 + \partial^{q(n-1)-p+1}$. Then $\Ord (U_1^*)=q(n-1)-p+1$, and $U_1^*$ satisfies condition $A_p(0)$. We have $\tilde{U}_1=U_1^*- \partial^{q(n-1)-p+1}+\frac{1}{d} x\partial^{q(n-1)-p+1}$, hence 
$$
U_1'=S^{-1}U_1^*S+ \frac{1}{d}S^{-1}x\partial^{q(n-1)-p+1}S - S^{-1}\partial^{q(n-1)-p+1}S. 
$$
As in step 4), by lemma \ref{L:conditionA} the operators $S^{-1}U_1^*S$, $S^{-1}\partial^{q(n-1)-p+1}S$ satisfy condition $A_p(0)$, and the operator $S^{-1}x\partial^{q(n-1)-p+1}S$ satisfies condition $A_p(1)$.    
Therefore, by lemma \ref{L:obvious}  $Sdeg_A(U_1')_{q(n-1)-p-t}< t+1$ for any $t>0$, and therefore for any $t>0$
$$
\ord (U_1')_{q(n-1)-p-t} = q(n-1)-p-t+ Sdeg_A (U_1')_{q(n-1)-p-t}<q(n-1)-p+1. 
$$
On the other hand, since $(S^{-1})_{-1}=S_{-1}=0$ and $(S^{-1})_{0}=S_{0}=1$, we have
$$
(U_1')_{q(n-1)-p}= \frac{1}{d} x\partial^{q(n-1)-p+1} + (\tilde{\tilde{U}}_1)_{q(n-1)-p}
$$
and therefore $\ord (U_1')_{q(n-1)-p} = q(n-1)-p+1$ (as $Sdeg_A (\tilde{\tilde{U}}_1)_{q(n-1)-p}\le 0$) and 
$$
U_1'=\frac{1}{d} x\partial^{q(n-1)-p+1} +a_{q(n-1)-p}\partial^{q(n-1)-p}+ \ldots +a_0
$$
with $a_i\in \hat{R}$ for $i<q(n-1)-p+1$. 

6). Recall that $\Phi_u^{-1} (U_1')\in A_1$ (as $\Phi_u^{-1}(U_1')=Q^n-P^m$). For convenience denote  $z:=q(n-1)-p$. Then in the second case (i.e. $\Ord (\tilde{U}_1)= q(n-1)-p$) the highest coefficient of the differential operator $\Phi_u^{-1} (U_1')$ will be equal to $H:=\frac{1}{d} f g^{(z+1)/d}\in x\cdot K[x]$, and in the first case (i.e. $\Ord (\tilde{U}_1)= q(n-1)-p+1$) it will be equal to $\tilde{H}:=(\varepsilon_1 + \frac{1}{d} f) g^{(z+1)/d}\in K[x]$. So, either $\frac{1}{d}f=H g^{-(z+1)/d}$ or 
$\varepsilon_1 + \frac{1}{d} f= \tilde{H} g^{-(z+1)/d}$. In both cases, differentiating these equalities with respect to $x$ we get (we'll use the same notation $H$ both for $H$ and $\tilde{H}$ below)
$$
\frac{1}{d} f'=\frac{1}{d} g^{-1/d}=H'g^{-(z+1)/d} - \frac{z+1}{d} H g^{-1-(z+1)/d}g',
$$ 
whence 
\begin{equation}
\label{E:difeq*}
\frac{1}{d} g^{(m-1)(n-1)}= H'g- \frac{z+1}{d} Hg'.
\end{equation}
Claim: $\deg_x (H)= (m-1)(n-1)l - (l-1)$. Proof: since $\deg_x g^{(m-1)(n-1)}=l(m-1)(n-1)$, we have 
$\deg_x (H) \ge (m-1)(n-1)l - (l-1)$. Obviously, $\deg_x H'g= \deg_x Hg'$, and the highest coefficients of $H'g$ and $Hg'$ are $\deg_x(H)c_Hc_g$ and $lc_Hc_g$ correspondingly, but $GCD(z+1,d)=1$ and $d\nmid l$, therefore $\deg_x(H)c_Hc_g -\frac{z+1}{d}lc_Hc_g \neq 0$ whence 
$\deg_x(H'g-\frac{z+1}{d}Hg')=\deg_x Hg'$ and thus $\deg_x (H)= (m-1)(n-1)l - (l-1)$.\footnote{Note that if $l=d$, these arguments imply that \eqref{E:difeq*} is impossible, because $(m-1)(n-1)d - (d-1)=z+1$,
i.e. $\Ord (\tilde{U}_1)>q(n-1)-p+1$.}

7). Let's show that the identity \eqref{E:difeq*} is possible only if $g$ has one root. This follows from the auxiliary lemma:

\begin{lemma}
\label{L:polynomials}
Suppose $\tilde{g}\in \dc [x]$ is a polynomial with more than one root. Suppose $\deg_x(\tilde{g})=\tilde{l}$, and $\tilde{d}, \tilde{z}, A\in \dn$,   satisfy the following relation: $\tilde{z}=(A-1)\tilde{d}$. 

Then there exists a solution $H\in \dc [x]$ of the differential equation
\begin{equation}
\label{E:difeq*0}
\tilde{c} \tilde{g}^{A}= H'\tilde{g}- \frac{\tilde{z}+1}{\tilde{d}} H\tilde{g}', \quad \tilde{c}\in \dc^*
\end{equation}
only if $\tilde{l}/\tilde{d}\in \dn$, $\tilde{l}/\tilde{d}>1$. In this case $\deg_x (H)=\tilde{l}(\tilde{z}+1)/\tilde{d}$. 
\end{lemma}

\begin{proof}
If $H$ exists, clearly we have 
$\deg_x (H) \ge A \tilde{l} - (\tilde{l}-1)= \tilde{l}(A-1)+1$.

Let, say 
$$
\tilde{g}=c (x-\alpha_1)^{j_1}\ldots (x-\alpha_k)^{j_k}(x-\beta_1)\ldots (x-\beta_p),
$$
where $k>1$ or  $p> 1$ or $k=p=1$ (if $p<1$ or $k<1$ we assume there are no factors with $\beta_j$ or $\alpha_j$) and all $j_q>1$. Then 
$$
g_1:=GCD(\tilde{g},\tilde{g}')=\begin{cases}
c(x-\alpha_1)^{j_1-1}\ldots (x-\alpha_k)^{j_k-1} & k\ge 1 \\
1 & k<1 .
\end{cases} 
$$
Assume $\tilde{g}=g_1r$, $\tilde{g}'=g_1s$, where 
$$
r=(x-\alpha_1)\ldots (x-\alpha_k)(x-\beta_1)\ldots (x-\beta_p),
$$
\begin{multline*}
 s=\sum_{q=1}^kj_q(x-\alpha_1)\ldots \widehat{(x-\alpha_q)}\ldots (x-\alpha_k)(x-\beta_1)\ldots (x-\beta_p) +\\
 \sum_{q=1}^p(x-\alpha_1)\ldots (x-\alpha_k)(x-\beta_1)\ldots \widehat{(x-\beta_q)}\ldots (x-\beta_p).
\end{multline*}
By our assumptions $\deg_x s>0$, and, obviously, $GCD(r,s)=1$ and $\deg_x(r)=\deg_x(s)+1$. Dividing equation \eqref{E:difeq*0} by $g_1$, we get the equation
$$
\tilde{c} g_1^{A-1}r^{A}=H'r-\frac{\tilde{z}+1}{\tilde{d}} Hs.
$$ 
Since $GCD(r,s)=1$ and the left hand side has the same roots as $r$, $H$ is divisible by all linear factors of $r$. Let's show that $H$ must be divisible by the whole left part, i.e. 
$$\deg_x(H)\ge \tilde{l}A-\deg_x(g_1)>\tilde{l}(A-1)+1.$$ 
Note that this is possible only if  $\deg_x(H'\tilde{g}-\frac{\tilde{z}+1}{\tilde{d}}H\tilde{g}')<\deg_x H\tilde{g}'$, i.e. only if $\deg_x (H)=\tilde{l}\frac{\tilde{z}+1}{\tilde{d}}$ and $\dn \ni\tilde{l}/\tilde{d}>1$. 

The proof is by induction on the degree: let, say, $H=(x-\alpha_1)H_1$ (so that $\deg_xH_1=\deg_x H-1$). Then, dividing the equation above by $(x-\alpha_1)$ we get the new equation on $H_1$:
$$
G=H_1'r-\frac{\tilde{z}+1}{\tilde{d}} Hs_1, \quad s_1=s-\frac{\tilde{d}}{\tilde{z}+1}r/(x-\alpha_1),
$$ 
where the set of roots of $G$ belong to the set of roots of $r$ again. Note that $s_1$ has the same shape as $s$:
\begin{multline*}
 s_1=(j_1-\frac{\tilde{d}}{\tilde{z}+1})(x-\alpha_2)\ldots (x-\alpha_k)(x-\beta_1)\ldots (x-\beta_p)+\\
 \sum_{q=2}^kj_q(x-\alpha_1)\ldots \widehat{(x-\alpha_q)}\ldots (x-\alpha_k)(x-\beta_1)\ldots (x-\beta_p) +\\
 \sum_{q=1}^p(x-\alpha_1)\ldots (x-\alpha_k)(x-\beta_1)\ldots \widehat{(x-\beta_q)}\ldots (x-\beta_p).
\end{multline*}
Since $\frac{\tilde{d}}{\tilde{z}+1}<1$, we have $GCD(r,s_1)=1$ and therefore we can repeat these arguments for other linear divisors of $G$. 

Continuing this line of reasoning, we'll get generic equations of the following type: 
$$
G_i=H_i'r-\frac{\tilde{z}+1}{\tilde{d}} H_is_i, 
$$ 
where 
\begin{multline*}
 s_i=\sum_{q=1}^k(j_q-\frac{\tilde{d}}{\tilde{z}+1}n_q)(x-\alpha_1)\ldots \widehat{(x-\alpha_q)}\ldots (x-\alpha_k)(x-\beta_1)\ldots (x-\beta_p) +\\
 \sum_{q=1}^p(1-\frac{\tilde{d}}{\tilde{z}+1}m_q)(x-\alpha_1)\ldots (x-\alpha_k)(x-\beta_1)\ldots \widehat{(x-\beta_q)}\ldots (x-\beta_p),
\end{multline*}
with $n_q\le j_qA-j_q$ (note that the multiplicity of $\alpha_q$ in $g_1^{A-1}r^{A}$ is $(j_q-1)(A-1)+A=j_qA-j_q+1$), $m_q\le A-1$, and 
$$
G_i=\tilde{c} \frac{g_1^{A-1}r^{A}}{(x-\alpha_1)^{n_1}\ldots (x-\alpha_k)^{n_k}(x-\beta_1)^{m_1}\ldots (x-\beta_p)^{m_p}},
$$
so that the set of roots of $G_i$ always belong to the set of roots of $r$.  Then, 
$$
j_q-\frac{\tilde{d}}{\tilde{z}+1}n_q = \frac{j_qd(A-1) +j_q -dn_q}{\tilde{z}+1}\ge \frac{j_q}{\tilde{z}+1}>0,
$$ 
$$
1-\frac{\tilde{d}}{\tilde{z}+1}m_q=\frac{\tilde{d}(A-1)+1-\tilde{d}m_q}{\tilde{z}+1}\ge \frac{1}{\tilde{z}+1}>0,
$$
whence $GCD(r,s_i)=1$. At the last step we get the equation 
$$
c (x-\alpha_1)\ldots (x-\alpha_k)(x-\beta_1)\ldots (x-\beta_p)=H_N'r-\frac{\tilde{z}+1}{\tilde{d}} H_Ns_N
$$
whence again $H_N$ must be divisible by the polynomial on the left hand side, and we are done.

\end{proof}

8). If $g$ has only one root, then $g=(1+\alpha x)^l$ for some $\alpha\in K$. In this case a solution of equation \eqref{E:difeq*} necessarily has the form $H=\tilde{c}(1+\alpha x)^{l(m-1)(n-1)-l+1}$, $\tilde{c}\in K$. But in the second case (i.e. $\Ord (\tilde{U}_1)= q(n-1)-p$) $H(0)=0$ - a contradiction, so item 2. holds. In the first case (i.e. $\Ord (\tilde{U}_1)= q(n-1)-p+1$) it maybe possible. Besides, in this case 
$$
f=\alpha^{-1}(1-\frac{l}{d})(1+\alpha x)^{1-l/d} + c_0,
$$
where $c_0=-\alpha^{-1}(1-\frac{l}{d})$ (as $f(0)=0$). Then necessarily 
$$
\varepsilon_1 = -\frac{c_0}{d}=\frac{\alpha^{-1}}{d}(1-\frac{l}{d}).
$$

{\it Item} 5. By the same arguments as in item 3. 2) we get $(U_1')_{\Ord (U_1')}=(\tilde{Q}^n)_{\Ord (U_1')}=(\tilde{U}_1)_{\Ord (\tilde{U}_1)}\in D_1$ and therefore $(U_1')_{\Ord (U_1')}=\varepsilon_1 \partial^{\Ord (\tilde{U}_1)}$ with $0\neq \varepsilon_1\in K$, because $Sdeg_A((\tilde{U}_1)_j)\le 0$ for all $j>q(n-1)-p$ (see \eqref{E:@@} and explanations in 3. 2)). 

By item 2. we know that $Sdeg_A([\tilde{P}, \tilde{U}_1]_{q(n-1)-t})<t$ for all $t>0$. Therefore,
$$
Sdeg_A(\tilde{U}_1)_{q(n-1)-p-t}< t+1
$$
for all $t>0$. Since $\chi := \Ord (\tilde{U}_1)>q(n-1)-p+1$, we get 
$$
Sdeg_A(\tilde{U}_1)_{\chi -t}< t
$$
for all $t>0$. In particular, the operator $\tilde{U}_1$ satisfies condition $A_p(0)$. 

{\it Item} 6. Since the operator $\tilde{U}_1$ satisfies condition $A_p(0)$, the operator $U_1'$ also satisfies condition $A_p(0)$ by proposition \ref{C:P'}. Therefore, $U_1'$ is a formally elliptic differential operator (cf. similar arguments in 3., 4),5)), i.e. $U_1'=\varepsilon_1 \partial^{\Ord (\tilde{U}_1)}+\ldots$, and therefore, $\ord (U_1')=\Ord (U_1')$. 

Analogously,  
$$
[P',U_1']=S^{-1}[\tilde{P},\tilde{U}_1]S\in D_1\quad\mbox{and}\quad \Ord ([P',U_1'])=q(n-1).
$$ 
By proposition \ref{C:P'} $[P',U_1']$ satisfies condition $A_p(0)$. Since $U_1', [P',U_1']$ are differential operators, their homogeneous components are also differential operators, i.e. they don't depend on $A_{p,i}$, cf. lemma \ref{L:correct_HPC}. Then, for any $t>0$ we have 
$$
\ord ([P',U_1']_{q(n-1)-t})=Sdeg_A([P',U_1']_{q(n-1)-t})+q(n-1)-t<q(n-1),
$$ 
and $\ord ([P',U_1']_{q(n-1)})=q(n-1)=\Ord ([P',U_1'])$, i.e. the highest term of the {\it differential operator}  $[P',U_1']$ coincides with $[P',U_1']_{q(n-1)}=-n \partial^{q(n-1)}$.

 Put $Q_1:=\Phi_u^{-1}(U_1')$. Note that $Q_1\in A_1$ and  $\ord (Q_1)=\Ord (Q_1)= \ord (U_1')=\Ord (U_1')= \Ord(\tilde{U}_1)$, because $\Phi_u^{-1}$ preserves the degree filtration on $D_1$ and $g(0)\neq 0$. Now we have 
$$
\ord ([P,Q_1])= \ord \Phi_u^{-1}([P',U_1'])=\ord ([P',U_1'])=\Ord ([P',U_1'])= \Ord ([\tilde{P},\tilde{U}_1]).
$$
Since by items 3., 2. we have $\ord (P)+ \ord (Q_1)-1 > q(n-1)= \Ord ([\tilde{P},\tilde{U}_1])$, it follows by \cite[L. 2.7]{Dixmier} that $f_{0,1}(P)$ is proportional with $f_{0,1}(Q_1)$, i.e. $f_{0,1}(Q_1)^p=\tilde{c}f_{0,1}(P)^{\ord (Q_1)}=\tilde{c}f_{0,1}(P)^{\Ord(\tilde{U}_1)}$. Therefore, $l\Ord(\tilde{U}_1)$ is divisible by $d$, i.e. $\Ord (\tilde{U}_1)$ is divisible by $d_2$.

{\it Items} 7., 8. immediately follow from item 3., 4. (cf. formulae \eqref{E:@@}, \eqref{E:@@@}). 

\end{proof}

{\it Step 3.} The proof of our theorem will follow  by induction from lemma below.

\begin{lemma}
\label{L:induction_step1}
In the notation above suppose $\tilde{U}_{k}\in \hat{D}_1^{sym}$ ($k\ge 1$) is an operator which satisfies the following properties (below we assume $n_0:=n$):
\begin{enumerate}
\item
$U_k':=S^{-1}\tilde{U}_{k}S\in D_1$ and $\Phi_u^{-1}(U_k')\in A_1$;
\item
the operator $\tilde{U}_k$ satisfies condition $A_p(0)$, i.e. in particular $(\tilde{U}_k)_{\Ord \tilde{U}_k}=\varepsilon_k \partial^{\Ord \tilde{U}_k}$ for some $\varepsilon_k\in K^*$;
\item
the operator $[\tilde{P},\tilde{U}_k]$ satisfies condition $A_p(0)$;
\item
$\Ord (\tilde{U}_k)> \Ord (\tilde{U}_{k-1})n_{k-1}-M_{k-1}+1$;
\item
The number $\Ord (\tilde{U}_k)$ is divisible by $d_2=EssGCD(P,Q)$;
\item
$Sdeg_A(\tilde{U}_k)_j\le 0$ for $\Ord (\tilde{U}_k)\ge j > \Ord (\tilde{U}_{k-1})n_{k-1}-M_{k-1}$, \\
and $Sdeg_A(\tilde{U}_k)_{\Ord (\tilde{U}_{k-1})n_{k-1}-M_{k-1}}=1$. Besides, 
$$
(\tilde{U}_k)_{\Ord (\tilde{U}_{k-1})n_{k-1}-M_{k-1}}=c_kx\partial^{\Ord (\tilde{U}_{k-1})n_{k-1}-M_{k-1}+1}+c_{k-1} \partial^{\Ord (\tilde{U}_{k-1})n_{k-1}-M_{k-1}},
$$
where $c_k\in K$ and $c_{k-1}$ is a HCP from $Hcpc (p)$ with $Sdeg_A(c_{k-1})\le 0$. 
\item
$M_k:=\Ord (\tilde{U}_k)- (\Ord (\tilde{U}_{k-1})n_{k-1} - M_{k-1})< M_{k-1}$.
\end{enumerate}

In the notation above put $n_k:=\frac{p}{GCD(p,\Ord \tilde{U}_k)}$, $m_k:=\frac{\Ord \tilde{U}_k}{GCD(p,\Ord \tilde{U}_k)}$,
$$
\tilde{U}_{k+1}:=\tilde{U}_k^{n_k}-
\varepsilon_k^{n_k}\tilde{P}^{m_k}.
$$ 
Then we have
\begin{enumerate}
\item
$U_{k+1}':=S^{-1}\tilde{U}_{k+1}S\in D_1$ and $\Phi_u^{-1}(U_{k+1}')\in A_1$;
\item
the operator $[\tilde{P},\tilde{U}_{k+1}]$ satisfies condition $A_p(0)$;
\item
$\Ord (\tilde{U}_{k+1})> \Ord (\tilde{U}_k)n_k - M_k$;
\item
$\Ord (\tilde{U}_{k+1})> \Ord (\tilde{U}_k)n_k - M_k+1$ if the polynomial $g$ has more than one root;

Moreover, if $g$ has only one root and $\Ord (\tilde{U}_{k+1})= \Ord (\tilde{U}_k)n_k - M_k+1$, then 
$\sigma (\tilde{U}_{k+1})=\varepsilon_{k+1} \partial^{\Ord (\tilde{U}_{k+1})}$, and 
$$
U_{k+1}'=a_{\Ord (\tilde{U}_k)n_k - M_k+1}\partial^{\Ord (\tilde{U}_k)n_k - M_k+1} +a_{\Ord (\tilde{U}_k)n_k - M_k}\partial^{\Ord (\tilde{U}_k)n_k - M_k}+ \ldots +a_0
$$
with $a_{\Ord (\tilde{U}_k)n_k - M_k+1}=\varepsilon_{k+1} + c_{k+1} x$ and $a_i\in \hat{R}$ for $i<\Ord (\tilde{U}_k)n_k - M_k+1$.
\item
if $\Ord (\tilde{U}_{k+1})> \Ord (\tilde{U}_k)n_k - M_k+1$, then the operator $\tilde{U}_{k+1}$ satisfies condition $A_p(0)$;
\item
if $\Ord (\tilde{U}_{k+1})> \Ord (\tilde{U}_k)n_k - M_k+1$, then the number $\Ord (\tilde{U}_{k+1})$ is divisible by $d_2$; in particular, $d_2$ divides $GCD(\Ord (\tilde{P}), \Ord (\tilde{Q}), \Ord (\tilde{U}_1), \ldots , \Ord (\tilde{U}_{k+1}))$, and this GCD divides $d$.
\item
$Sdeg_A(\tilde{U}_{k+1})_j\le 0$ for $\Ord (\tilde{U}_{k+1})\ge j > \Ord (\tilde{U}_{k})n_k-M_{k}$, \\
and $Sdeg_A(\tilde{U}_{k+1})_{\Ord (\tilde{U}_{k})n_k-M_{k}}=1$. Besides,
$$
(\tilde{U}_{k+1})_{\Ord (\tilde{U}_{k})n_k-M_{k}}=c_{k+1}x\partial^{\Ord (\tilde{U}_{k})n_k-M_{k}+1}+\tilde{c}_{k} \partial^{\Ord (\tilde{U}_{k})n_k-M_{k}},
$$
where $c_{k+1}\in K$ and $\tilde{c}_{k}$ is a HCP from $Hcpc (p)$ with $Sdeg_A(\tilde{c}_{k})\le 0$ and $Sdeg_B(\tilde{c}_{k})=-\infty$.
\item
$M_{k+1}:=\Ord (\tilde{U}_{k+1}) - (\Ord (\tilde{U}_k) n_k -M_k)< M_k$. 
\end{enumerate}

\end{lemma}

\begin{proof}
The proof is similar to the proof of lemma \ref{L:induction_step}, with minimal modifications. For convenience of the reader we describe details below. 

{\it Item} 1. follows from the property 1. for operators $\tilde{U}_k$, $\tilde{P}$. 

{\it Item} 2.   We have 
$$
[\tilde{P}, \tilde{U}_{k+1}]=[\tilde{P}, \tilde{U}_k]\tilde{U}_k^{n_k-1}+\tilde{U}_k[\tilde{P}, \tilde{U}_k]\tilde{U}_k^{n_k-2}+\ldots + \tilde{U}_k^{n_k-1}[\tilde{P}, \tilde{U}_k].
$$
Since $[\tilde{P}, \tilde{U}_k]$ and $\tilde{U}_k$ satisfy condition $A_p(0)$, all terms  in this sum satisfy condition $A_p(0)$ by lemma \ref{L:conditionA}. Besides, all terms have the same order, because the product of symbols of $\tilde{U}_k$ and $[\tilde{P}, \tilde{U}_k]$ is not zero, as $\sigma (\tilde{U}_k)=\varepsilon_k \partial^{\Ord (\tilde{U}_k)}$ and $\sigma ([\tilde{P}, \tilde{U}_k])=\tilde{\varepsilon}_1\partial^{\Ord [\tilde{P}, \tilde{U}_k]}$ (for some $\tilde{\varepsilon}_1\in K^*$).  Therefore, first three conditions from definition \ref{D:conditionA} are satisfied. Obviously, $\sigma ([\tilde{P}, \tilde{U}_{k+1}])=n_k\tilde{\varepsilon}_1\varepsilon_k^{n_k-1}\partial^{(n_k-1)\Ord (\tilde{U}_k)+\Ord [\tilde{P}, \tilde{U}_k]}$, so the last condition holds, and $[\tilde{P}, \tilde{U}_{k+1}]$ satisfies condition $A_p(0)$. 

Besides, by the property 6 of the operator $[\tilde{P}, \tilde{U}_k]$ we get 
$$
\Ord ([\tilde{P}, \tilde{U}_k])=\Ord (\tilde{U}_{k-1})n_{k-1}-M_{k-1}+p,
$$
whence 
$$
\Ord ([\tilde{P}, \tilde{U}_{k+1}])=\Ord (\tilde{U}_k)(n_k-1)+\Ord (\tilde{U}_{k-1})n_{k-1}-M_{k-1}+p=
\Ord (\tilde{U}_k) n_k-M_k+p.
$$

{\it Items} 3.,4. From item 1. we immediately get $\Ord (\tilde{U}_{k+1})\ge \Ord (\tilde{U}_k) n_k-M_k$. Assume that $\Ord (\tilde{U}_{k+1})\le \Ord (\tilde{U}_k) n_k-M_k+1$, i.e. $\Ord (\tilde{U}_{k+1})= \Ord (\tilde{U}_k) n_k-M_k+1$ or $\Ord (\tilde{U}_{k+1})= \Ord (\tilde{U}_k) n_k-M_k$. 

1). By item 1. we know that $Sdeg_A([\tilde{P}, \tilde{U}_{k+1}]_{\Ord (\tilde{U}_k) n_k-M_k+p-t})<t$ for all $t>0$. Therefore,
$$
Sdeg_A(\tilde{U}_{k+1})_{\Ord (\tilde{U}_k) n_k-M_k-t}< t+1
$$
for all $t>0$.

2). Put $U_{k+1}':=S^{-1}\tilde{U}_{k+1}S$. Note that  $\Ord (U_{k+1}')=\Ord (\tilde{U}_{k+1})$. In the first case (i.e. $\Ord (\tilde{U}_{k+1})= \Ord (\tilde{U}_k) n_k-M_k+1$) let's prove that 
$$
(U_{k+1}')_{\Ord (\tilde{U}_k) n_k-M_k+1}=(\tilde{U}_{k+1})_{\Ord (\tilde{U}_k) n_k-M_k+1}=\varepsilon_{k+1} \partial^{\Ord (\tilde{U}_k) n_k-M_k+1},
$$ 
where $0\neq \varepsilon_{k+1}\in K$. Indeed, for all $j> \Ord (\tilde{U}_k) n_k-M_k$ we have
\begin{equation}
\label{E:@@1}
(\tilde{U}_k^{n_k})_j= \sum_{i_1+\ldots +i_{n_k}=j}(\tilde{U}_k)_{i_1}\cdots (\tilde{U}_k)_{i_{n_k}},
\end{equation}
where $i_q\le \Ord (\tilde{U}_k)$. Since $j> \Ord (\tilde{U}_k) n_k-M_k$, we have $i_q>\Ord (\tilde{U}_{k-1})n_{k-1}-M_{k-1}$ for all $q$. Therefore, $Sdeg_A((\tilde{U}_k)_{i_q})\le 0$ for all $q$ by property 6, and by lemma \ref{L:HCPC} we get $Sdeg_A((\tilde{U}_k^{n_k})_j)\le 0$ for all $j> \Ord (\tilde{U}_k) n_k-M_k$.
Therefore $Sdeg_A((\tilde{U}_{k+1})_j)\le 0$ for all $j>\Ord (\tilde{U}_k) n_k-M_k$ too. Since $\Ord (\tilde{U}_{k+1})= \Ord (\tilde{U}_k) n_k-M_k+1$, we must have $Sdeg_A((\tilde{U}_{k+1})_{\Ord (\tilde{U}_k) n_k-M_k+1})=0$. 

Now note that 
$(\tilde{U}_{k+1})_{\Ord (\tilde{U}_k) n_k-M_k+1}=(\tilde{U}_k^{n_k})_{\Ord (\tilde{U}_k) n_k-M_k+1}=(U_{k+1}')_{\Ord (\tilde{U}_k) n_k-M_k+1}\in D_1$  (because $S^{-1}_{0}=S_{0}=1$), i.e. it is a differential operator, and by lemmas \ref{T:A.Z L 7.2}, \ref{L:correct_HPC} this HCP does not depend on $A_{p;i}$, i.e. it looks like $\varepsilon_{k+1} \partial^{\Ord (\tilde{U}_k) n_k-M_k+1}$ for $\varepsilon_{k+1}\in K^*$. 

3). For convenience denote $\tilde{M}_k:=\Ord (\tilde{U}_k) n_k-M_k$.  By lemma \ref{L:Q_p} we have
\begin{multline}
\label{E:@@@1}
(\tilde{U}_{k+1})_{\tilde{M}_k}=(\tilde{U}_k^{n_k})_{\tilde{M}_k}=\varepsilon^{\Ord (\tilde{U}_k) (n_k-1)} ((\tilde{U}_k)_{\tilde{M}_{k-1}}\partial^{\Ord (\tilde{U}_k) (n_k-1)}+ \\
\partial^{\Ord (\tilde{U}_k)} (\tilde{U}_k)_{\tilde{M}_{k-1}}\partial^{\Ord (\tilde{U}_k) (n_k-2)}+\ldots + \partial^{\Ord (\tilde{U}_k) (n_k-1)}(\tilde{U}_k)_{\tilde{M}_{k-1}}) + \\
\sum_{i_1+\ldots +i_{n_k}=\tilde{M}_k  \atop i_1,\ldots ,i_{n_k}>\tilde{M}_{k-1} }(\tilde{U}_k)_{i_1}\cdots (\tilde{U}_k)_{i_{n_k}}=
n_k \varepsilon^{\Ord (\tilde{U}_k) (n_k-1)} c_{k} x\partial^{\tilde{M}_k +1}+\tilde{c}_{k} \partial^{\tilde{M}_k} +\\
\sum_{i_1+\ldots +i_{n_k}=\tilde{M}_k  \atop i_1,\ldots ,i_{n_k}>\tilde{M}_{k-1} }(\tilde{U}_k)_{i_1}\cdots (\tilde{U}_k)_{i_{n_k}},
\end{multline}
where  $\tilde{c}_{k}$ is a HCP from $Hcpc (p)$ with $Sdeg_A(\tilde{c}_{k})\le 0$ and $Sdeg_B(\tilde{c}_{k})=-\infty$, because the property 6 for $\tilde{U}_k$ holds and the HCP $(\tilde{U}_k^{n_k})_{\tilde{M}_k}$ is totally free of $B_j$. Note that all terms in the last raw of this formula have $Sdeg_A\le 0$ by lemma \ref{L:HCPC}.

4). Denote $c_{k+1}:= n_k \varepsilon_k^{\Ord (\tilde{U}_k) (n_k-1)} c_{k}$. Put $\tilde{\tilde{U}}_{k+1}:=\tilde{U}_{k+1}-c_{k+1} x\partial^{\tilde{M}_k +1}$. In the first case (i.e. $\Ord (\tilde{U}_{k+1})= \tilde{M}_k+1$) this operator satisfies condition $A_p(0)$. Indeed, conditions 1,2 from definition \ref{D:conditionA} are obvious, the symbol $\sigma (\tilde{\tilde{U}}_{k+1})$ does not contain $A_{p;i}$ and $Sdeg_A(\sigma (\tilde{\tilde{U}}_{k+1}))=0$ by 2), and $Sdeg_A((\tilde{\tilde{U}}_{k+1})_{\tilde{M}_k+1-t})<t$ for all $t>0$ by 3), 1).

Then 
$$
U_{k+1}'=S^{-1}\tilde{\tilde{U}}_{k+1}S+ c_{k+1} S^{-1} x\partial^{\tilde{M}_k +1} S.
$$
By lemma \ref{L:conditionA} the operator $S^{-1}\tilde{\tilde{U}}_{k+1}S$ satisfies condition $A_p(0)$ and the operator $S^{-1} x\partial^{\tilde{M}_k +1} S$ satisfies condition $A_p(1)$. Therefore, by lemma \ref{L:obvious} 
$
Sdeg_A(U_{k+1}')_{\tilde{M}_k-t}<t+1
$
for all $t>0$, and therefore 
$$
\ord (U_{k+1}')_{\tilde{M}_k-t}= \tilde{M}_k-t +Sdeg_A(U_{k+1}')_{\tilde{M}_k-t}< \tilde{M}_k +1
$$
for all $t>0$. On the other hand, since $(S^{-1})_{-1}=S_{-1}=0$ and $(S^{-1})_{0}=S_{0}=1$, we have 
$$
(U_{k+1}')_{\tilde{M}_k}=c_{k+1} x\partial^{\tilde{M}_k+1} + (\tilde{\tilde{U}}_{k+1})_{\tilde{M}_k},
$$
and therefore 
$$
U_{k+1}'=a_{\tilde{M}_k+1}\partial^{\tilde{M}_k+1} +a_{\tilde{M}_k}\partial^{\tilde{M}_k}+ \ldots +a_0
$$
with $a_{\tilde{M}_k+1}=\varepsilon_{k+1} + c_{k+1} x$ and $a_i\in \hat{R}$ for $i<\tilde{M}_k+1$.

5). In the second case (i.e. $\Ord (\tilde{U}_{k+1})= \tilde{M}_k$) by steps 1), 3) we get for all $t\ge 0$
$$
Sdeg_A(\tilde{\tilde{U}}_{k+1})_{\tilde{M}_k-t}< t+1.
$$
Consider $U_{k+1}^*:=\tilde{\tilde{U}}_{k+1} + \partial^{\tilde{M}_k+1}$. Then $\Ord (U_{k+1}^*)=\tilde{M}_k+1$, and $U_{k+1}^*$ satisfies condition $A_p(0)$. We have $\tilde{U}_{k+1}=U_{k+1}^*- \partial^{\tilde{M}_k+1}+c_{k+1} x\partial^{\tilde{M}_k+1}$, hence 
$$
U_{k+1}'=S^{-1}U_{k+1}^*S+ c_{k+1}S^{-1}x\partial^{\tilde{M}_k+1}S - S^{-1}\partial^{\tilde{M}_k+1}S. 
$$
As in step 4), by lemma \ref{L:conditionA} the operators $S^{-1}U_{k+1}^*S$, $S^{-1}\partial^{\tilde{M}_k+1}S$ satisfy condition $A_p(0)$, and the operator $S^{-1}x\partial^{\tilde{M}_k+1}S$ satisfies condition $A_p(1)$.    
Therefore, by lemma \ref{L:obvious}  $Sdeg_A(U_{k+1}')_{\tilde{M}_k-t}< t+1$ for any $t>0$, and therefore for any $t>0$
$$
\ord (U_{k+1}')_{\tilde{M}_k-t} = \tilde{M}_k-t+ Sdeg_A (U_{k+1}')_{\tilde{M}_k-t}<\tilde{M}_k+1. 
$$
On the other hand, since $(S^{-1})_{-1}=S_{-1}=0$ and $(S^{-1})_{0}=S_{0}=1$, we have
$$
(U_{k+1}')_{\tilde{M}_k}= c_{k+1} x\partial^{\tilde{M}_k+1} + (\tilde{\tilde{U}}_{k+1})_{\tilde{M}_k}
$$
and therefore $\ord (U_{k+1}')_{\tilde{M}_k} = \tilde{M}_k+1$ (as $Sdeg_A (\tilde{\tilde{U}}_{k+1})_{\tilde{M}_k}\le 0$) and 
$$
U_{k+1}'=c_{k+1} x\partial^{\tilde{M}_k+1} +a_{\tilde{M}_k}\partial^{\tilde{M}_k}+ \ldots +a_0
$$
with $a_i\in \hat{R}$ for $i<\tilde{M}_k+1$. 

6). Note that $\Phi_u^{-1} (U_{k+1}')\in A_1$.  Then in the second case (i.e. $\Ord (\tilde{U}_{k+1})= \tilde{M}_k$) the highest coefficient of the differential operator $\Phi_u^{-1} (U_{k+1}')$ will be equal to $H:=c_{k+1} f g^{(\tilde{M}_k+1)/d}\in x\cdot K[x]$, and in the first case (i.e. $\Ord (\tilde{U}_{k+1})= \tilde{M}_k +1$) it will be equal to $\tilde{H}:=(\varepsilon_{k+1} + c_{k+1} f) g^{(\tilde{M}_k+1)/d}\in K[x]$. So, either $c_{k+1}f=H g^{-(\tilde{M}_k+1)/d}$ or 
$\varepsilon_{k+1} + c_{k+1} f= \tilde{H} g^{-(\tilde{M}_k+1)/d}$. In both cases, differentiating these equalities with respect to $x$ we get (we'll use again the same notation $H$ both for $H$ and $\tilde{H}$ below)
$$
c_{k+1} f'=c_{k+1} g^{-1/d}=H'g^{-(\tilde{M}_k+1)/d} - \frac{\tilde{M}_k+1}{d} H g^{-1-(\tilde{M}_k+1)/d}g',
$$ 
whence 
\begin{equation}
\label{E:difeq*1}
c_{k+1} g^{1+ \frac{\tilde{M}_k}{d}}= H'g- \frac{\tilde{M}_k+1}{d} Hg'
\end{equation}
(obviously, $d_2 | \tilde{M}_k$, but $d$ may not divide $\tilde{M}_k$; so, in particular, in this case all multiplicities of $g$ are divisible by a divisor $d_1$ of $GCD(d,l)$, i.e. $g=\tilde{g}^{d_1}$, $\tilde{g}\in \dc [x]$, $d=d_1d_3$ and $d_3 | \tilde{M}_k$). Just the same arguments as in lemma \ref{L:induction_step}, 3. 6) show that  $\deg_x (H)= (1+ \frac{\tilde{M}_k}{d})l - (l-1)$\footnote{Note again that if $l=d$, these arguments imply that \eqref{E:difeq*1} is impossible, $\Ord (\tilde{U}_{k+1})> \Ord (\tilde{U}_k)n_k - M_k+1$.}.

 Dividing the equation \eqref{E:difeq*1} by $\tilde{g}^{d_1-1}$, we get the equation
$$
c_{k+1} \tilde{g}^{1+ \frac{\tilde{M}_k}{d_3}}= H'\tilde{g}- \frac{\tilde{M}_k+1}{d_3} H\tilde{g}',
$$
and by lemma \ref{L:polynomials} (for $A=1+ \frac{\tilde{M}_k}{d_3}$, $\tilde{d}=d_3$, $\tilde{l}=l/d_1$, $\tilde{z}=\tilde{M}_k$) this equation has a solution only if $\tilde{g}$ has one root, because $l/d<1$.

7).  If $g$ has only one root, then $g=(1+\alpha x)^l$ for some $\alpha\in K$. In this case necessarily $H=\tilde{c}(1+\alpha x)^{(1+ \frac{\tilde{M}_k}{d})l-l+1}$, $\tilde{c}\in K$. But in the second case $H(0)=0$ - a contradiction, so item 2) holds. So, only the first case is possible. 

{\it Item} 5. By the same arguments as in item 2. 2) we get $(U_{k+1}')_{\Ord (U_{k+1}')}=(\tilde{U}_k^n)_{\Ord (U_{k+1}')}=(\tilde{U}_{k+1})_{\Ord (\tilde{U}_{k+1})}\in D_1$ and therefore $(U_{k+1}')_{\Ord (U_{k+1}')}=\varepsilon_{k+1} \partial^{\Ord (\tilde{U}_{k+1})}$ with $0\neq \varepsilon_{k+1}\in K$, because $Sdeg_A((\tilde{U}_{k+1})_j)\le 0$ for all $j>\tilde{M}_k$ (see \eqref{E:@@1} and explanations in 3. 2).). 

By item 2. we know that $Sdeg_A([\tilde{P}, \tilde{U}_{k+1}]_{\tilde{M}_k+p-t})<t$ for all $t>0$. Therefore,
$$
Sdeg_A(\tilde{U}_{k+1})_{\tilde{M}_k-t}< t+1
$$
for all $t>0$. Since $\chi := \Ord (\tilde{U}_{k+1})>\tilde{M}_k+1$, we get 
$$
Sdeg_A(\tilde{U}_{k+1})_{\chi -t}< t
$$
for all $t>0$. In particular, the operator $\tilde{U}_{k+1}$ satisfies condition $A_p(0)$. 

{\it Item} 6. Since the operator $\tilde{U}_{k+1}$ satisfies condition $A_p(0)$, the operator $U_{k+1}'$ also satisfies condition $A_p(0)$ by proposition \ref{C:P'}. Therefore, $U_{k+1}'$ is a formally elliptic differential operator (cf. similar arguments in 3., 4),5)), i.e. $U_{k+1}'=\varepsilon_{k+1} \partial^{\Ord (\tilde{U}_{k+1})}+\ldots$, and therefore, $\ord (U_{k+1}')=\Ord (U_{k+1}')$. 

Analogously,  
$$
[P',U_{k+1}']=S^{-1}[\tilde{P},\tilde{U}_{k+1}]S\in D_1\quad\mbox{and}\quad \Ord ([P',U_{k+1}'])=\tilde{M}_k+p.
$$ 
By proposition \ref{C:P'} $[P',U_{k+1}']$ satisfies condition $A_p(0)$. Since $U_{k+1}', [P',U_{k+1}']$ are differential operators, their homogeneous components are also differential operators, i.e. they don't depend on $A_{p,i}$, cf. lemma \ref{L:correct_HPC}. Then, for any $t>0$ we have 
$$
\ord ([P',U_{k+1}']_{\tilde{M}_k+p-t})=Sdeg_A([P',U_{k+1}']_{\tilde{M}_k+p-t})+\tilde{M}_k+p-t<\tilde{M}_k+p,
$$ 
and $\ord ([P',U_{k+1}']_{\tilde{M}_k+p})=\tilde{M}_k+p=\Ord ([P',U_{k+1}'])$, i.e. the highest term of the {\it differential operator}  $[P',U_{k+1}']$ coincides with $[P',U_{k+1}']_{\tilde{M}_k+p}=n_k\tilde{\varepsilon}_1\varepsilon_k^{n_k-1}\partial^{(n_k-1)\Ord (\tilde{U}_k)+\Ord [\tilde{P}, \tilde{U}_k]}$, see item 2.

 Put $Q_{k+1}:=\Phi_u^{-1}(U_{k+1}')$. Note that $\ord (Q_{k+1})=\Ord (Q_{k+1})= \ord (U_{k+1}')=\Ord (U_{k+1}')= \Ord(\tilde{U}_{k+1})$, because $\Phi_u^{-1}$ preserves the order filtration on $D_1$ and $g(0)\neq 0$. Now we have 
$$
\ord ([P,Q_{k+1}])= \ord \Phi_u^{-1}([P',U_{k+1}'])=\ord ([P',U_{k+1}'])=\Ord ([P',U_{k+1}'])= \Ord ([\tilde{P},\tilde{U}_{k+1}]).
$$
Since by items 3., 2. we have $\ord (P)+ \ord (Q_{k+1})-1 > \tilde{M}_k+p= \Ord ([\tilde{P},\tilde{U}_{k+1}])$, it follows by \cite[L. 2.7]{Dixmier} that $f_{0,1}(P)$ is proportional with $f_{0,1}(Q_{k+1})$, i.e. $f_{0,1}(Q_{k+1})^p=\tilde{c}f_{0,1}(P)^{\ord (Q_{k+1})}=\tilde{c}f_{0,1}(P)^{\Ord(\tilde{U}_{k+1})}$. Therefore, $l\Ord(\tilde{U}_{k+1})$ is divisible by $d$, i.e. $\Ord (\tilde{Q}_{k+1})$ is divisible by $d_2$.

{\it Items} 7., 8. immediately follow from item 3., 4. (cf. formulae \eqref{E:@@1}, \eqref{E:@@@1}).

\end{proof}

If $g$ has more than one root, then by induction lemma \ref{L:induction_step1} we get a sequence of operators $\tilde{U}_k$ satisfying properties 1-8 from lemma. But the sequence $M_k$ strictly decreases, and at the same time $M_k>0$ by property 3, a contradiction.\footnote{The same contradiction proves that $l\neq d$ (i.e. $l<d$). Thus we get as a byproduct an alternative proof of \cite[Th. 4.1 (3)]{GGV}.} 

If $g$ has one root, then the sequence of operators $\tilde{U}_k$ can stop, by item 4., on an operator $\tilde{U}_I$ with $\Ord (\tilde{U}_{I})= \Ord (\tilde{U}_{I-1})n_{I-1} - M_{I-1}+1$. Besides, we have the sequence of differential operators $U_k'\in D_1$ and $Q_k\in A_1$. From the proof of item 6. we know for $k=1, \ldots ,I-1$ that 
$$
\ord (Q_{k})=\Ord (Q_{k})= \ord (U_{k}')=\Ord (U_{k}')= \Ord(\tilde{U}_{k})
$$
and therefore  $\Ord (\tilde{U}_{I-1})n_{I-1} - M_{I-1}$ is divisible by $d_2$, so $\Ord (\tilde{U}_{I})=1$ mod $d_2$. Moreover, since $d|M_0$, by induction we get that 
$GCD\{\Ord (\tilde{P}), \Ord (\tilde{Q}), \Ord (\tilde{U}_1), \ldots , \Ord (\tilde{U}_{I-1})\}$ divides $M_{I-1}$, hence $\Ord (\tilde{U}_{I-1})n_{I-1} - M_{I-1}$ is divisible by this GCD, and therefore
$$
\Ord (\tilde{U}_{I})=1 \quad \mbox{mod} \quad GCD\{\Ord (\tilde{P}), \Ord (\tilde{Q}), \Ord (\tilde{U}_1), \ldots , \Ord (\tilde{U}_{I-1})\}.
$$

By item 4. we also know that 
$$
U_{I}'=a_{\Ord (\tilde{U}_{I-1})n_{I-1} - M_{I-1}+1}\partial^{\Ord (\tilde{U}_{I-1})n_{I-1} - M_{I-1}+1} +a_{\Ord (\tilde{U}_{I-1})n_{I-1} - M_{I-1}}\partial^{\Ord (\tilde{U}_{I-1})n_{I-1} - M_{I-1}}+ \ldots +a_0
$$
with $a_{\Ord (\tilde{U}_{I-1})n_{I-1} - M_{I-1}+1}=\varepsilon_{k+1} + c_{k+1} x$ and $a_i\in \hat{R}$ for $i<\Ord (\tilde{U}_{I-1})n_{I-1} - M_{I-1}+1$. Then $Q_I:=\Phi_u^{-1}(U_I')\in A_1$ and 
$$
\ord (Q_I)=\Ord (\tilde{U}_{I-1})n_{I-1} - M_{I-1}+1,
$$
because $\Phi_u^{-1}$ preserves the order filtration on $D_1$ and $g(0)\neq 0$. 

So, the sequence of operators $Q_i$ is the needed sequence $F_i(P,Q)$ from theorem. 

The equality \eqref{E:distance} follows from
$$
\sum_{i=0}^{I-1} (n_i\ord F_i(P,Q)-\ord F_{i+1}(P,Q))=\sum_{i=0}^{I-1} (M_i-M_{i+1})=M_0-M_{I}=
p+q-1.
$$
\end{proof} 

We can say more about the operators $F_i(P,Q)$ from theorem \ref{T:DC-pairs}. Namely, we'll show below that they are of subrectangular type. Before we prove it, let's make several preliminary comments.  

Assume there exists a  DC-pair $P,Q\in A_1$ that provides a counterexample to the Dixmier conjecture and $\varphi$ is the corresponding endomorphism with $\varphi (x)=P$, $\varphi (\partial )=Q$. Then, as it was noted in section \ref{S:DC-pairs} there exists a DC-pair of subrectangular type satisfying conditions of theorem \ref{T:DC-pairs} and conditions of lemma \ref{L:order_of_polynomial} (remark \ref{R:order_of_polynomial}).  Moreover, without loss of generality and by remark \ref{R:DC-pais} we can assume 
$$f_{0,1}(P)=(1+\alpha x)^{ln}y^{dn}, \quad f_{0,1}(Q)=(1+\alpha x)^{lm}y^{dm}, \quad 1<l<d, \quad (m,n)=1, \quad \epsilon (\varphi)\notin \dn .
$$ 
Let $f_{1,0}(P)=\tilde{h}x^{ln}$, $\tilde{h}\in \dc [y]$, then $\tilde{h}=\tilde{g}^n$ for some $\tilde{g}\in \dc [y]$ of degree $d$ and $f_{1,0}(Q)=\tilde{g}^mx^{lm}$ (cf. arguments in section \ref{S:DC-pairs}).

For convenience, let's introduce one notation (from \cite{GGV}): for an operator $B\in A_1$ with $f_{1,0}(B)=hx^{v_{1,0}(B)}$, $h\in \dc [y]$, $f_{0,1}(B)=gy^{v_{0,1}(B)}$, $g\in \dc [x]$, we define  
$$en_{1,0}(B):=(v_{1,0}(B),\deg_y(h(y))), \quad st_{1,0}(B):=(v_{1,0}(B),-\Ord (h(x))),$$ 
$$en_{0,1}(B):=(-\Ord (g(x)),v_{0,1}(B)),\quad st_{0,1}(B):=(\deg_x(g(x)),v_{0,1}(B)).$$

\begin{theorem}
\label{T:subrectangular_sequence}
Under assumptions of theorem \ref{T:DC-pairs},  the operators  $F_i(P,Q)$, $i=0, \ldots ,I$ are of subrectangular type.

Let $f_{1,0}(Q)=\tilde{g}^mx^{lm}$, $f_{1,0}(Q_I)=H_1x^{v_{1,0}(Q_I)}$, where $\tilde{g}=\tilde{g}_1^{d_1}$, $d=d_1d_3$, $l=d_1l_3$, $H_1\in K[y]$. If  the polynomial $\tilde{g}$ has one root, say $\tilde{g}=(b+\tilde{\alpha} y)^d$,\footnote{$\tilde{\alpha}^d=\alpha^l$} then $H_1=c(b+\tilde{\alpha} y)^{v_{0,1}(Q_I)}$ for some $c\in K$. 

If  the polynomial $\tilde{g}$ has more than one root, then $H_1$ satisfies the differential equation
\begin{equation}
\label{E:dif_eq_H_1}
\frac{\partial H_1}{\partial y}\tilde{g}_1- \frac{(l(M-n)+1)n_{I-1}}{l_3}H_1\frac{\partial \tilde{g}_1}{\partial y} =\tilde{c}_1H^{n_{I-1}-1}\tilde{g}_1^{d_1(M-n)+1}, \quad \tilde{c}_1\in K, 
\end{equation}
where $H\in K[y]$, $f_{1,0}(Q_{I-1})=Hx^{v_{1,0}(Q_{I-1})}$, satisfies the differential equation
\begin{equation}
\label{E:dif_eq_H}
\frac{\partial H}{\partial y}\tilde{g}_1- \frac{l(M-n)+1}{l_3}H\frac{\partial \tilde{g}_1}{\partial y} =\tilde{c}\tilde{g}_1^{d_1(M-n)+1}, \quad \tilde{c}\in K,
\end{equation}
and 
$$
M=  \frac{(n_{I-2}-1)\ord (Q_{I-2})+(n_{I-3}-1)\ord (Q_{I-3})+\ldots +(n-1)dm}{d}
$$
(note that $H_1, H$ are uniquely determined by equations \eqref{E:dif_eq_H_1}, \eqref{E:dif_eq_H}).

Moreover, the polynomial $\tilde{g}$ has more that one root only if $l|d$, and in this case 
$$\deg H_1=((l(M-n)+1)n_{I-1}-1)d_2+1, \quad \deg H=(l(M-n)+1)d_2.$$ 
\end{theorem}

\begin{proof}

The proof is by induction. For $F_0(P,Q)$ the claim is trivial. 

For $i=1$ we have the following alternative: 

{\it Case} 1. We have $\{f_{1,0}(Q_1), f_{1,0}(P)\}= 0$. 

Then by theorem \ref{T:Dixmier2.7} we have $f_{1,0}(Q_1)^{ln}=cf_{1,0}(P)^{v_{1,0}(Q_1)}$, i.e. $f_{1,0}(Q_1)=c\tilde{g}^Nx^{v_{1,0}(Q_1)}$, where $N=v_{1,0}(Q_1) /l$. 
So, $en_{1,0}(Q_1)=(l,d)v_{1,0}(Q_1)/l$.
By lemma  \ref{L:induction_step} we have two possible cases:

a) if $v_{0,1}(Q_1)=\ord (Q_1)> dm(n-1)-dn+1$, then $\ord (Q_1)$ is divisible by $d_2$ and $f_{0,1}(Q_1)^{dn}=\tilde{c}f_{0,1}(P)^{\ord (Q_1)}$ (see the proof of item 6 in lemma).  So, in particular, 
$st_{0,1}(Q_1)dn=st_{0,1}(P)\ord (Q_1)$, i.e. $st_{0,1}(Q_1)=(l,d) \ord (Q_1)/d$ and $st_{0,1}(Q_1)$, $en_{1,0}(Q_1)$ are proportional. Since $v_{1,0}(Q_1)\ge l\ord (Q_1)/d$ and $dv_{1,0}(Q_1)/l\le \ord (Q_1)$, we must have $v_{1,0}(Q_1)= l\ord (Q_1)/d$ and therefore $st_{0,1}(Q_1)=en_{1,0}(Q_1)$, i.e. $Q_1$ is of subrectangular type. Moreover, in this case $f_{0,1}(Q_1)$ is a (fractional) power of $f_{0,1}(P)$ and 
$f_{1,0}(Q_1)$ is a (fractional) power of $f_{1,0}(P)$, i.e. $st_{0,1}(Q_1)$ is proportional with $st_{0,1}(P)$ and $en_{1,0}(Q_1)$ is proportional with $en_{1,0}(P)=st_{0,1}(P)$. 

b) if $v_{0,1}(Q_1)=\ord (Q_1)= dm(n-1)-dn+1$, then $st_{0,1}(Q_1)=(l(mn-m-n)+1,d(mn-m-n)+1)$ (see the proof of item 8 in lemma \ref{L:induction_step}). But we have $v_{1,0}(Q_1)\ge l(mn-m-n)+1$ and $d(mn-m-n)+1\ge d v_{1,0}(Q_1)/l$, whence $l/d\ge 1$ - a contradiction, as $l<d$. So, this case is impossible. 

{\it Case} 2. We have $\{f_{1,0}(Q_1), f_{1,0}(P)\}\neq 0$. Consider the following two cases:

a) $\tilde{g}$ has more than one root. 

Since $[Q,P]=1$, by theorem \ref{T:Dixmier2.7} we have $\{f_{1,0}(Q), f_{1,0}(P)\}= 0$. Recall that in step 2 of the proof of lemma \ref{L:induction_step} we get 
$[P,Q_1]=-nQ^{n-1}$. So, the operator $[P,Q_1]$ is of subrectangular type, and by theorem \ref{T:Dixmier2.7} we have $f_{1,0}([Q_1,P])=n f_{1,0}(Q)^{n-1}=n \tilde{g}^{m(n-1)}x^{lm(n-1)}$. 

Then by theorem \ref{T:Dixmier2.7} we have 
$$
\{f_{1,0}(Q_1), f_{1,0}(P)\}=f_{1,0}([Q_1,P]), \quad v_{1,0}(Q_1)+v_{1,0}(P)-1=v_{1,0}([Q_1,P])=(n-1)lm.  
$$
Let $f_{1,0}(Q_1)=Hx^{v_{1,0}(Q_1)}$, $H\in K[y]$. Then 
$$
\{f_{1,0}(Q_1), f_{1,0}(P)\}=-ln\frac{\partial H}{\partial y}\tilde{g}^nx^{v_{1,0}(Q_1)+ln-1}+
n v_{1,0}(Q_1)H\tilde{g}^{n-1}\frac{\partial \tilde{g}}{\partial y}x^{v_{1,0}(Q_1)+ln-1},
$$
so we get the equation
$$
-l\frac{\partial H}{\partial y}\tilde{g}^nx^{v_{1,0}(Q_1)+ln-1}+
 v_{1,0}(Q_1)H\tilde{g}^{n-1}\frac{\partial \tilde{g}}{\partial y}x^{v_{1,0}(Q_1)+ln-1}=
\tilde{g}^{m(n-1)}x^{lm(n-1)}, 
$$
whence $v_{1,0}(Q_1)=lm(n-1)-ln+1$ and 
\begin{equation}
\label{E:difeq*2}
-l\frac{\partial H}{\partial y}\tilde{g}+ v_{1,0}(Q_1)H\frac{\partial \tilde{g}}{\partial y} =\tilde{g}^{(m-1)(n-1)}.
\end{equation}

By lemma \ref{L:polynomials} (for $A=(m-1)(n-1)$, $\tilde{d}=l$, $\tilde{l}=d$, $\tilde{z}=l(mn-m-n)$) we have $\deg_y(H)= d(mn-m-n)+\frac{d}{l}$ (note that it is possible only if $d/l=d_2$). Hence, $\ord (Q_1)>d(mn-m-n)+1$ and then by lemma \ref{L:induction_step} $\ord (Q_1)$ is divisible by $d_2$ and $st_{0,1}(Q_1)$ is proportional with $st_{0,1}(P)=(ln,dn)$, i.e. $st_{0,1}(Q_1)=(l\ord (Q_1)/d, \ord (Q_1))$.  On the other hand, we have in this case also
$$
en_{1,0}(Q_1)=(l(mn-m-n)+1, d(mn-m-n)+d_2)
$$
and therefore $en_{1,0}(Q_1)$ is also  proportional with $st_{0,1}(P)=(ln,dn)\sim (1,d_2)$. But then $st_{0,1}(Q_1)=en_{1,0}(Q_1)$ (as $\ord (Q_1)\ge d(mn-m-n)+d_2$ and $l(mn-m-n)+1\ge l\ord (Q_1)/d$), and $Q_1$ is of subrectangular type. 

Now note that $M_1=d_2$ (in the notation of lemma \ref{L:induction_step}). So, $0<M_2<d_2$ and therefore $M_2$ can not be divisible by $d_2$, hence by lemma \ref{L:induction_step1} $M_2=1$. This means that $I=2$, and $Q_2$ is the last operator in the sequence for $P,Q$. By item 3.4.7) of lemma \ref{L:induction_step1} we have in this case $n_1=ln$, $\tilde{M}_1=(d(mn-m-n)+d_2)n_1-d_2$, 
$\ord (Q_2)=\tilde{M}_1+1$ and 
\begin{multline*}
st_{0,1}(Q_2)=((\frac{(d(mn-m-n)+d_2)n_1-d_2}{d})l+1, (d(mn-m-n)+d_2)n_1-d_2+1)=\\
((l(mn-m-n)+1)n_1, ((l(mn-m-n)+1)n_1-1)d_2+1).
\end{multline*}
So, $st_{0,1}(Q_2)$ is not proportional with $st_{0,1}(P)=(ln,dn)\sim (1,d_2)$. Now we have again two cases:

aa) $\{f_{1,0}(Q_2), f_{1,0}(P)\}= 0$. Then by theorem \ref{T:Dixmier2.7} we have $en_{1,0}(Q_2)$ is proportional with $en_{1,0}(P)=(ln,dn)\sim (1,d_2)$. Let, say, $en_{1,0}(Q_2)=(A,Ad_2)$. Then we have 
$$((l(mn-m-n)+1)n_1-1)d_2+1\ge Ad_2),\quad A\ge (l(mn-m-n)+1)n_1,$$ 
whence $-d_2+1\ge 0$ - a contradiction. So, this case is impossible.

ab) $\{f_{1,0}(Q_2), f_{1,0}(P)\}\neq 0$. Then by theorem \ref{T:Dixmier2.7} we have $\{f_{1,0}(Q_2), f_{1,0}(P)\}=f_{1,0}([Q_2,P])$. Note that 
\begin{multline*}
[Q_2,P]=[Q_1,P]Q_1^{n_1-1}+Q_1[Q_1,P]Q_1^{n_1-2}+\ldots +Q_1^{n_1-1}[Q_1,P]= \\
nQ^{n-1}Q_1^{n_1-1}+nQ_1Q^{n-1}Q_1^{n_1-2}+\ldots + nQ_1^{n_1-1}Q^{n-1},
\end{multline*}
and by remark \ref{R:subrectangular} the operator $[Q_2,P]$ is of subrectangular type (as all summands are of subrectangular type with the same highest monomials). In particular, $v_{1,0}([Q_2,P])=(n_1-1)v_{1,0}(Q_1)+(n-1)lm$. 
Moreover, by theorem \ref{T:Dixmier2.7} we also have
$$
v_{1,0}(Q_2)+v_{1,0}(P)-1=v_{1,0}([Q_2,P])=(n_1-1)v_{1,0}(Q_1)+(n-1)lm=(n_1-1)(l(mn-m-n)+1)+(n-1)lm,
$$
hence 
$$
v_{1,0}(Q_2)=(l(mn-m-n)+1)n_1,
$$
and the vertex $st_{0,1}(Q_2)$ belong to the $(1,0)$-top line of $Q_2$ (as it is vertical and passes through the point $((l(mn-m-n)+1)n_1,0)$). But then $en_{1,0}(Q_2)=st_{0,1}(Q_2)$ and $Q_2$ is of subrectangular type. 

b) $\tilde{g}$ has one root. 

Then $\tilde{g}=(b+\tilde{\alpha} y)^d$ for some $b\in K$. In this case the solution of the equation \eqref{E:difeq*2} is of the form $H=c(b+\tilde{\alpha} y)^{d(mn-m-n)+1}$, hence $en_{1,0}(Q_1)=(l(mn-m-n)+1, d(mn-m-n)+1)$. According to lemma \ref{L:induction_step} we have two possibilities: 

ba) $\ord (Q_1)$ is divisible by $d_2$, and therefore $st_{0,1}(Q_1)$ is proportional with $st_{0,1}(P)=(ln,dn)$ (see the proof of item 6 in lemma). Let, say, $st_{0,1}(Q_1)=(lA,dA)$. Then we have 
$l(mn-m-n)+1\ge lA$ and $dA\ge d(mn-m-n)+1$, whence
$$
l(mn-m-n)+1\ge lA\ge l(mn-m-n)+l/d.
$$ 
Since $l<d$, we must have $lA=l(mn-m-n)+1$, hence $dA=d(mn-m-n)+d/l$. Again it is possible only if $d/l=d_2$. But then the vertex  $st_{0,1}(Q_1)$ belongs to the $(1,0)$-top line of $Q_1$, and lies above the vertex  $en_{1,0}(Q_1)$ -- a contradiction. So, this case is impossible. 

bb) $\ord (Q_1)$ is not divisible by $d_2$. Then $\ord (Q_1)=d(mn-m-n)+1$, and $st_{0,1}(Q_1)=(l(mn-m-n)+1, d(mn-m-n)+1)$ (see item 3.4. 8) of lemma). So, $st_{0,1}(Q_1)= en_{1,0}(Q_1)$ and therefore $Q_1$ is of subrectangular type. Besides, in this case $I=1$. 

\bigskip

Summing up, note that only in the case 1 a) we need to continue our research in order to prove the assertion. So, continuing this line of reasoning, by induction we can assume  for  $1<i<I$ that $Q_i$ is an operator such that $[Q_i,P]$ is of subrectangular type and that $f_{0,1}([Q_i,P])$ is a (fractional) power of $f_{0,1}(P)$ and 
$f_{1,0}([Q_i,P])$ is a (fractional) power of $f_{1,0}(P)$ (cf. item 1 a)). Indeed, this follows from the induction hypothesis, remark \ref{R:subrectangular}, theorem \ref{T:Dixmier2.7}, and since 
$$
[Q_i,P]=[Q_{i-1},P]Q_{i-1}^{n_{i-1}-1}+Q_{i-1}[Q_{i-1},P]Q_{i-1}^{n_{i-1}-2}+\ldots +Q_{i-1}^{n_{i-1}-1}[Q_{i-1},P].
$$
In particular, $st_{0,1}([Q_i,P])$ is proportional with $st_{0,1}(P)=(ln,dn)$ and $en_{1,0}([Q_i,P])$ is proportional with $en_{1,0}(P)=st_{0,1}(P)=(ln,dn)$. 

For this induction step we need to repeat all arguments given above for the operator $Q_1$. They are practically the same. For convenience of the reader we give them here. 

{\it Case} 1. We have $\{f_{1,0}(Q_i), f_{1,0}(P)\}= 0$. Then by theorem  \ref{T:Dixmier2.7} we have $f_{1,0}(Q_i)^{ln}=cf_{1,0}(P)^{v_{1,0}(Q_i)}$, i.e. $f_{1,0}(Q_i)=c\tilde{g}^{N_i}x^{v_{1,0}(Q_i)}$, where $N_i=v_{1,0}(Q_i) /l$. 
So, $en_{1,0}(Q_i)=(l,d)v_{1,0}(Q_i)/l$.
By lemma  \ref{L:induction_step1} we have two possible cases:

a) if $v_{0,1}(Q_i)=\ord (Q_i)> \ord (Q_{i-1})n_{i-1}-M_{i-1}+1$, then $\ord (Q_i)$ is divisible by $d_2$ and $f_{0,1}(Q_i)^{dn}=\tilde{c}f_{0,1}(P)^{\ord (Q_i)}$ (see the proof of item 6 in lemma \ref{L:induction_step1}).  So, in particular, 
$st_{0,1}(Q_i)dn=st_{0,1}(P)\ord (Q_i)$, i.e. $st_{0,1}(Q_i)=(l,d) \ord (Q_i)/d$ and $st_{0,1}(Q_i)$, $en_{1,0}(Q_i)$ are proportional. Since $v_{1,0}(Q_i)\ge l\ord (Q_i)/d$ and $dv_{1,0}(Q_i)/l\le \ord (Q_i)$, we must have $v_{1,0}(Q_i)= l\ord (Q_i)/d$ and therefore $st_{0,1}(Q_i)=en_{1,0}(Q_i)$, i.e. $Q_i$ is of subrectangular type. Moreover, in this case $f_{0,1}(Q_i)$ is a (fractional) power of $f_{0,1}(P)$ and 
$f_{1,0}(Q_i)$ is a (fractional) power of $f_{1,0}(P)$. 

b) if $v_{0,1}(Q_i)=\ord (Q_i)= \ord (Q_{i-1})n_{i-1}-M_{i-1}+1=: \tilde{M}_i+1$, then 
$st_{0,1}(Q_i)=(l\tilde{M}_i/d +1,\tilde{M}_i+1)$ (see the proof of item 3.4.7) in lemma \ref{L:induction_step1}). But we have $v_{1,0}(Q_i)\ge l\tilde{M}_i/d +1$ and $\tilde{M}_i+1\ge d v_{1,0}(Q_i)/l$, whence $d/l\le 1$ - a contradiction, as $l<d$. So, this case is impossible. 

{\it Case} 2. We have $\{f_{1,0}(Q_i), f_{1,0}(P)\}\neq 0$. Consider the following two cases:

a) $\tilde{g}$ has more than one root. 

Then by theorem \ref{T:Dixmier2.7} we have 
\begin{multline*}
\{f_{1,0}(Q_i), f_{1,0}(P)\}=f_{1,0}([Q_i,P])=c_0\tilde{g}^Mx^{v_{1,0}([Q_i,P])}, \quad \mbox{where}\\ 
v_{1,0}(Q_i)+v_{1,0}(P)-1=v_{1,0}([Q_i,P])=
(n_{i-1}-1)v_{1,0}(Q_{i-1})+(n_{i-2}-1)v_{1,0}(Q_{i-2})+\ldots +(n-1)lm, \\
dM=  (n_{i-1}-1)\ord (Q_{i-1})+(n_{i-2}-1)\ord (Q_{i-2})+\ldots +(n-1)dm,
\end{multline*}
because $[Q_i,P]$ is of subrectangular type.  Since $en_{1,0}([Q_i,P])$ is proportional to $(l,d)$, we get $v_{1,0}([Q_i,P])=lM$. Obviously, $d_2$ divides the right hand side of the expression for $dM$ (cf. the proof of item 3.4. 6) of lemma \ref{L:induction_step1}), but $d$ may not divide it; so, in particular, in this case all multiplicities of $\tilde{g}$ are divisible by a divisor $d_1$ of $GCD(d,l)$, i.e. $\tilde{g}=\tilde{g}_1^{d_1}$, $\tilde{g}_1\in \dc [x]$, $d=d_1d_3$ and $d_3$ divides the right hand side of the expression for $dM$. Moreover, $d_1$ divides $l$, $l:=d_1l_3$.

Let $f_{1,0}(Q_i)=Hx^{v_{1,0}(Q_i)}$, $H\in K[y]$. Then 
$$
\{f_{1,0}(Q_i), f_{1,0}(P)\}=-ln\frac{\partial H}{\partial y}\tilde{g}^nx^{v_{1,0}(Q_i)+ln-1}+
n v_{1,0}(Q_i)H\tilde{g}^{n-1}\frac{\partial \tilde{g}}{\partial y}x^{v_{1,0}(Q_i)+ln-1},
$$
so we get the equation
$$
-l\frac{\partial H}{\partial y}\tilde{g}^nx^{v_{1,0}(Q_i)+ln-1}+
 v_{1,0}(Q_i)H\tilde{g}^{n-1}\frac{\partial \tilde{g}}{\partial y}x^{v_{1,0}(Q_i)+ln-1}=
c_0\tilde{g}^{M}x^{v_{1,0}([Q_i,P])}, 
$$
whence $v_{1,0}(Q_i)=v_{1,0}([Q_i,P])-ln+1$ (so, $(v_{1,0}(Q_i),l)=1$) and 
\begin{equation}
\label{E:difeq*3}
-l\frac{\partial H}{\partial y}\tilde{g}+ v_{1,0}(Q_i)H\frac{\partial \tilde{g}}{\partial y} =c\tilde{g}^{M-n+1}, \quad c=c_0/n.
\end{equation}
Dividing the equation \eqref{E:difeq*3} by $-l\tilde{g}_1^{d_1-1}$, we get the equation
$$
\frac{\partial H}{\partial y}\tilde{g}_1- \frac{l(M-n)+1}{l_3}H\frac{\partial \tilde{g}_1}{\partial y} =\tilde{c}\tilde{g}_1^{d_1(M-n)+1}, \quad \tilde{c}=-c/l.
$$

By lemma \ref{L:polynomials} (for $A=d_1(M-n)+1$, $\tilde{d}=l_3$, $\tilde{l}=d/d_1$,  $\tilde{z}=l(M-n)$) we have $\deg_y(H)= d(M-n)+\frac{d}{l}$ (note that again it is possible only if $d/l=d_2$). Hence, $\ord (Q_i)>d_1(M-n)+1$ and then by lemma \ref{L:induction_step1} $\ord (Q_i)$ is divisible by $d_2$ and $st_{0,1}(Q_i)$ is proportional with $st_{0,1}(P)=(ln,dn)$, i.e. $st_{0,1}(Q_i)=(l\ord (Q_i)/d, \ord (Q_i))$.  On the other hand, we have in this case also
$$
en_{1,0}(Q_i)=(l(M-n)+1, d(M-n)+d_2)
$$
and therefore $en_{1,0}(Q_i)$ is also  proportional with $st_{0,1}(P)=(ln,dn)\sim (1,d_2)$. But then $st_{0,1}(Q_i)=en_{1,0}(Q_i)$ (as $\ord (Q_i)\ge d(M-n)+d_2$ and $l(M-n)+1\ge l\ord (Q_i)/d$), and $Q_i$ is of subrectangular type. 

Now note that $M_i=d_2$ (in the notation of lemma \ref{L:induction_step1}). So, $0<M_{i+1}<d_2$ and therefore $M_{i+1}$ can not be divisible by $d_2$, hence by lemma \ref{L:induction_step1} $M_{i+1}=1$. This means that $I=i+1$, and $Q_{i+1}$ is the last operator in the sequence for $P,Q$. By item 3.4.7) of lemma \ref{L:induction_step1} we can put in this case $n_i=dn$, $\tilde{M}_i=(d(M-n)+d_2)n_i-d_2$, 
$\ord (Q_{i+1})=\tilde{M}_i+1$ and 
\begin{multline*}
st_{0,1}(Q_{i+1})=((\frac{(d(M-n)+d_2)n_i-d_2}{d})l+1, (d(M-n)+d_2)n_i-d_2+1)=\\
((l(M-n)+1)n_i, ((l(M-n)+1)n_i-1)d_2+1).
\end{multline*}
So, $st_{0,1}(Q_{i+1})$ is not proportional with $st_{0,1}(P)=(ln,dn)\sim (1,d_2)$. Now we have again two cases:

aa) $\{f_{1,0}(Q_{i+1}), f_{1,0}(P)\}= 0$. Then by theorem \ref{T:Dixmier2.7} we have $en_{1,0}(Q_{i+1})$ is proportional with $en_{1,0}(P)=(ln,dn)\sim (1,d_2)$. Let, say, $en_{1,0}(Q_{i+1})=(A,Ad_2)$. Then we have 
$$((l(M-n)+1)n_i-1)d_2+1\ge Ad_2,\quad A\ge (l(M-n)+1)n_i,$$ 
whence $-d_2+1\ge 0$ - a contradiction. So, this case is impossible.

ab) $\{f_{1,0}(Q_{i+1}), f_{1,0}(P)\}\neq 0$. Then by theorem \ref{T:Dixmier2.7} we have 
$\{f_{1,0}(Q_{i+1}), f_{1,0}(P)\}=f_{1,0}([Q_{i+1},P])$. Note that 
$$
[Q_{i+1},P]=[Q_i,P]Q_i^{n_i-1}+Q_i[Q_i,P]Q_i^{n_i-2}+\ldots +Q_i^{n_i-1}[Q_i,P],
$$
and by remark \ref{R:subrectangular} the operator $[Q_{i+1},P]$ is of subrectangular type (as all summands are of subrectangular type with the same highest monomials by the induction hypothesis). In particular, $v_{1,0}([Q_{i+1},P])=(n_i-1)v_{1,0}(Q_i)+v_{1,0}([Q_i,P])$. 
Moreover, by theorem \ref{T:Dixmier2.7} we also have
$$
v_{1,0}(Q_{i+1})+v_{1,0}(P)-1=v_{1,0}([Q_{i+1},P])=(n_i-1)v_{1,0}(Q_i)+v_{1,0}([Q_i,P])=(n_i-1)(l(M-n)+1)+lM,
$$
hence 
$$
v_{1,0}(Q_{i+1})=(l(M-n)+1)n_i,
$$
and the vertex $st_{0,1}(Q_{i+1})$ belongs to the $(1,0)$-top line of $Q_{i+1}$ (as it is vertical and passes through the point $((l(M-n)+1)n_i,0)$). But then $en_{1,0}(Q_{i+1})=st_{0,1}(Q_{i+1})$ and $Q_{i+1}$ is of subrectangular type. 

Put $f_{1,0}(Q_{i+1})=H_1x^{v_{1,0}(Q_{i+1})}$, $H_1\in K[y]$. Repeating calculations as above, we get another differential equation on $H_1$:
$$
\frac{\partial H_1}{\partial y}\tilde{g}_1- \frac{(l(M-n)+1)n_i}{l_3}H_1\frac{\partial \tilde{g}_1}{\partial y} =\tilde{c}_1H^{n_i-1}\tilde{g}_1^{d_1(M-n)+1}, \quad \tilde{c}_1\in K.
$$

b) $\tilde{g}$ has one root. 

Then $\tilde{g}=(b+\tilde{\alpha} y)^d$ for some $b\in K$. In this case the solution of the equation \eqref{E:difeq*3} is of the form $H=c(b+\tilde{\alpha} y)^{d(M-n)+1}$, hence $en_{1,0}(Q_i)=(l(M-n)+1, d(M-n)+1)$. According to lemma \ref{L:induction_step1} we have two possibilities: 

ba) $\ord (Q_i)$ is divisible by $d_2$, and therefore $st_{0,1}(Q_i)$ is proportional with $st_{0,1}(P)=(ln,dn)$ (see the proof of item 6 in lemma). Let, say, $st_{0,1}(Q_i)=(lA,dA)$. Then we have 
$l(M-n)+1\ge lA$ and $dA\ge d(M-n)+1$, whence
$$
l(M-n)+1\ge lA\ge l(M-n)+l/d.
$$ 
Since $l<d$, we must have $lA=l(M-n)+1$, hence $dA=d(M-n)+d/l$. Again it is possible only if $d/l=d_2$. But then the vertex  $st_{0,1}(Q_i)$ belongs to the $(1,0)$-top line of $Q_{i}$, and lies above the vertex  $en_{1,0}(Q_i)$ -- a contradiction. So, this case is impossible. 

bb) $\ord (Q_i)$ is not divisible by $d_2$. Then $\ord (Q_i)=d(M-n)+1$, and $st_{0,1}(Q_i)=(l(M-n)+1, d(M-n)+1)$ (see item 3.4. 7) of lemma). So, $st_{0,1}(Q_i)= en_{1,0}(Q_i)$ and therefore $Q_i$ is of subrectangular type. Besides, in this case $I=i$. 

\end{proof}

\section{DC-pairs and the string equation}
\label{S:string}

In this section we study solutions to the {\it string equation} $[Q,P]=1$, where $Q,P\in D_1$ (the DC-pairs, or more generally, arbitrary endomorphisms of the first Weyl algebra, provide partial solutions to this equation).  The natural question is how to classify, if possible, such solutions, and how  they are controlled by their normal forms. We'll show that solutions of this equation  are connected with pairs of commuting ordinary differential operators generating a ring of rank one, and vice versa, any pair of commuting differential operators generating a ring of rank one, whose spectral curve is of some special type, provide a solution to the string equation. In the last section we'll apply these results to show that there are no DC-pairs, i.e. the Dixmier conjecture is true. In this section we assume $K=\dc$ ($K$ may be assumed to be any other algebraically closed field of characteristic zero).

Theorem \ref{T:DC-pairs} motivates the following definition.

\begin{Def}
\label{D:distance}
Let $P,Q\in D_1$ be two differential operators such that $\ord (P)=\Ord (P)=p$, $\sigma (P)=\partial^{\Ord (P)}$,  $\ord (Q)=\Ord (Q)=q$, and $[Q,P]=1$ or $[Q,P]=0$. Assume that there exists a sequence of operators $Q_i:=F_i(P,Q)$, $i=0, \ldots ,I$ like in theorem \ref{T:DC-pairs} (i.e. defined by recursion in the same way) such that $\ord (Q_i)= \Ord (Q_i)$ for all $i$. In this case we'll define a sequence of numbers, called {\it sequential GCD}, as\footnote{We don't assume here that $SeqGCD_I(P,Q)=1$.}   
$$
SeqGCD_i(P,Q):=GCD (\ord (P), \ord (Q_0), \ldots ,\ord (Q_i)), \quad i=0, \ldots ,I
$$
and a {\it distance} of $P,Q$  by the formula \eqref{E:distance}:
$$
Dist(P,Q):= \sum_{i=0}^{I-1} (n_i\ord F_i(P,Q)-\ord F_{i+1}(P,Q)).
$$ 
The number $SeqGCD_I(P,Q)$ will be called as the {\it end of sequential GCD sequence} if $SeqGCD_I(P,Q)<SeqGCD_{I-1}(P,Q)$ .

The condition on orders imply that $\sigma (Q_i)=c_i\partial^{\Ord (Q_i)}$, $c_i\in K$ (see remark \ref{R:ord=deg}). We'll call the pairs of constants $(\Ord (Q_0), c_0), \ldots , (\Ord (Q_I), c_I)$ as the {\it sequence} of $P,Q$. 
\end{Def}

We formulated definition above in a slightly more general way than it was mentioned in Introduction. In most cases we'll use it to describe the head-chopping process; in this case we'll usually assume that $SeqGCD_I(P,Q)=1$ and $SeqGCD_{I-1}(P,Q)>1$ as in theorem \ref{T:DC-pairs}. As the sequences $F_i$ are not uniquely determined, it is natural to ask if the distance depends on the choice of a sequence with the same end of sequential GCD sequence. It is not difficult to see that distance does not depend on such a choice, namely, if $F_i'$, $i=0, \ldots ,I'$ is another sequence of polynomials such that $SeqGCD_{I'}(P,Q)=SeqGCD_I(P,Q)$ and $SeqGCD_{I'}(P,Q)<SeqGCD_{I'-1}(P,Q)$, $SeqGCD_{I}(P,Q)<SeqGCD_{I-1}(P,Q)$, then both distances coincide. In particular, this follows from lemma below (an improved analogue of lemma \ref{L:conjugation in E_q}):

\begin{lemma}
\label{L:invariant_conjugation}
Let $\tilde{P}=\partial^p, \tilde{Q}\in \hat{D}_1^{sym}$ be two operators such that $\sigma (\tilde{Q})=c\partial^q$, $c\in K$, and $[\tilde{Q},\tilde{P}]=1$ or $[\tilde{Q},\tilde{P}]=0$. Assume that there exists a sequence of operators $\tilde{Q}_i:=F_i(\tilde{P},\tilde{Q})$, $i=0, \ldots ,I$ like in theorem \ref{T:DC-pairs} (i.e. defined by recursion in the same way) such that $\sigma (\tilde{Q}_i)=\tilde{c}_i\partial^{\Ord \tilde{Q}_i}$, $\tilde{c}_i\in K^*$ for all $i$. Let 
$
\tilde{d}:=SeqGCD_I(\tilde{P},\tilde{Q})=GCD (p, \Ord (\tilde{Q}_0), \ldots ,\Ord (\tilde{Q}_I)).
$  
Then 
\begin{enumerate}
\item
 there exists a monic invertible operator $\tilde{S}\in K^{\oplus p}((\tilde{D}^{-1}))\subset E_p$ with $\ord_{\tilde{D}}(\tilde{S})=0$ such that all coefficients  of the operator $\tilde{S}^{-1}\hat{\Phi}(\tilde{Q})\tilde{S}$ 
satisfy the following condition (we'll call such operator as {\it auxiliary  operator}):  
$$
\tilde{S}^{-1}\hat{\Phi}(\tilde{Q})\tilde{S}=c_{q}\tilde{D}^q+ c_{q-k_1}\tilde{D}^{q-k_1} +\ldots + c_{q-k_1-\ldots -k_I}\tilde{D}^{q-k_1-\ldots -k_I}+\sum_{i< q-k_1-\ldots -k_I}q_i\tilde{D}^{i},
$$
where $\quad c_i\in K$\footnote{Recall that $K$ is embedded to $K^{\oplus p}$ as $a\mapsto (a,\ldots ,a)$}, $\quad q_i\in K^{\oplus p}$ for $i\neq -p$ (and for $i=-p$ $q_i$ has a shape $c\Gamma+ \tilde{q}_i$, $c\in K$,  $\tilde{q}_i\in  K^{\oplus p}$), $\quad [q_i,\tilde{D}^{\tilde{d}}]=0$ (and $\quad [\tilde{q}_i,\tilde{D}^{\tilde{d}}]=0$)\footnote{We'll call these coefficients, allowing for some freedom of speech, {\it $\tilde{D}^{\tilde{d}}$-invariant}}, $k_i=n_{i-1}\Ord (\tilde{Q}_{i-1})-\Ord (\tilde{Q}_i)$. 

Moreover, there exists  such an operator $\tilde{S}$ with additional conditions on its coefficients $s_l$: $s_{l,i}=0$ for $i=0, \ldots ,\tilde{d}-1$ for all $l$.

In particular, if $\tilde{d}=1$, then all coefficients $q_i$ are constant (i.e. $q_i$ or $\tilde{q}_i$ belong to $K$, in the same sense as above), so that $\tilde{S}^{-1}\hat{\Phi}(\tilde{Q})\tilde{S}\in K((\tilde{D}^{-1}))$ if $[\tilde{P},\tilde{Q}]=0$ and $\tilde{S}^{-1}\hat{\Phi}(\tilde{Q})\tilde{S}\in K[x]((\tilde{D}^{-1}))\subset K[\Gamma ]((\tilde{D}^{-1}))$ if $[\tilde{P},\tilde{Q}]=1$\footnote{The ring of classical Schur's pseudo-differential operators $E=K[[x]]((\tilde{D}^{-1}))$ is naturally embedded to $E_p$ by the following receipt: any series $P$ from $E$ can be represented as a sum of homogeneous components (HCPs) with respect to the $\Ord$ function. Rewriting them in the G-form, we get a series from $E_p$. Besides, the homogeneous components of this series don't depend on $A_{p,i}$ and $B_j$ and $Sdeg_A(P_{\ord (P)-i})\le i$ for all $i>0$ (if $P$ is monic, then even $Sdeg_A(P_{\ord (P)-i})< i$ and $\ord (P)=\Ord (P)$). 

It's easy to see that the image of the subring $D_1\subset E$ under this embedding is characterised by the properties of theorem \ref{T: P in hat(D) is a dif_op}, where we adopt definition \ref{D:conditionA} to the ring $E_p$ in the following sense: condition 2 from definition (totally freeness) is replaced by an equivalent system of equations \eqref{E:system_on_B_j} on coefficients, all other conditions are the same. In particular, since the homomorphism $\hat{\Phi}$ is an embedding, this image of $D_1$ coincides with the image $\hat{\Phi}(D_1)$ of the subring $D_1\subset \hat{D}_1^{sym}$ }. 

The constants $c_i$ don't depend on the choice of a sequence $F_i$ with the same end of sequential GCD sequence. 
\item
in case $[\tilde{Q},\tilde{P}]=1$ we have 
$$
Dist(\tilde{P},\tilde{Q}):=\sum_{i=0}^{I-1} (n_i\Ord F_i(\tilde{P},\tilde{Q})-\Ord F_{i+1}(\tilde{P},\tilde{Q}))= k_1+\ldots +k_I <p+q,
$$
and if $SeqGCD_{I-1}(\tilde{P},\tilde{Q})>\tilde{d}$, then $c_{q-k_1-\ldots -k_I}=c\tilde{c}_I\neq 0$. 

In particular, the distance does not depend on the choice of a sequence $F_i$ with the same end of sequential GCD sequence.
\end{enumerate}
\end{lemma}

\begin{rem}
\label{R:application_of_lemma_inv_conj}
This lemma can be applied e.g. to a normal form $\tilde{Q}$ of a differential operator $Q$ with respect to $P\in D_1$, satisfying similar conditions. Namely, let $P,Q\in D_1$ be two differential operators such that $\ord (P)=\Ord (P)=p$, $\ord (Q)=\Ord (Q)=q$, and $[Q,P]=1$ or $[Q,P]=0$. Assume that there exists a sequence of operators $Q_i:=F_i(P,Q)$, $i=0, \ldots ,I$ like in theorem \ref{T:DC-pairs} (i.e. defined by recursion in the same way) such that $\ord (Q_i)= \Ord (Q_i)$ for all $i$. Then the normal form of $Q$ with respect to $P$ satisfies conditions of lemma. 

For, by remark \ref{R:ord=deg} all operators $P,Q,Q_i$ have invertible highest coefficients, hence they are regular. Therefore by proposition \ref{T:A.Z 7.2} and by basic properties of the order function from remark \ref{R:ord_properties} normal forms of operators $Q_i$ with respect to $P$ have the same orders as $Q_i$, and their highest symbols have shape $c\partial^{\Ord (Q_i)}$, $c\in K$\footnote{Clearly, without loss of generality $P$ can be assumed to be normalized}. 

For a pair of operators $\tilde{P}, \tilde{Q}$ from lemma  the same notions as in definition \ref{D:distance} (i.e. sequential GCD, distance and sequence) are defined, and we'll use them below.

\end{rem}

\begin{proof}
1. By proposition \ref{T:A.Z 7.1} and lemma \ref{L:Q_p} all homogeneous components of $\tilde{Q}$ except maybe one (if $[\tilde{Q},\tilde{P}]=1$) belong to $C(\partial^p)$, and if $[\tilde{Q},\tilde{P}]=1$, then the only one homogeneous component not from the centralizer is $\tilde{Q}_{-p}$ of $Sdeg_A(\tilde{Q}_{-p})=1$ and it has a shape $-(x\partial )/p +a$, where $a\in  C(\partial^p)$ (its explicit shape, up to a constant, is determined in lemma \ref{L:Q_p}).  Then all coefficients of the series $\hat{\Phi}(\tilde{Q})$ except maybe one belong to $K^{\oplus p}$, and if $[\tilde{Q},\tilde{P}]=1$, then the only one coefficient not from the centralizer has a shape $-(x\partial )/p +a$, where $a\in  K^{\oplus p}$. 

To prove our claim we apply lemma \ref{L:conjugation in E_q} to the operator  $\hat{\Phi}(\tilde{Q})$ several ($I$) times. First, by this lemma and remark \ref{R:S_over_K} there exists an invertible monic operator $\tilde{S}_0$ from $K^{\oplus p}((\tilde{D}^{-1}))$ such that all coefficients  of the operator $\tilde{S}_0^{-1}\hat{\Phi}(\tilde{Q})\tilde{S}_0$ commute with $\tilde{D}^{d_1}$ (in the same sense as it is formulated), where $d_1=SeqGCD_0(\tilde{P},\tilde{Q})$. For, it is claimed there in the case if $[\tilde{Q},\tilde{P}]=0$, and in the case $[\tilde{Q},\tilde{P}]=1$ the proof is the same as in loc. cit. (and we'll repeat it below for generic step), because by relations from lemma \ref{L:1} and analogous relations in the ring $E_p$ we have for any such $\tilde{S}\in K^{\oplus p}((\tilde{D}^{-1}))$ 
\begin{equation}
\label{E:gamma-congugation}
\tilde{S}^{-1}\Gamma \tilde{D}^{-p}\tilde{S}=\Gamma  \tilde{D}^{-p}+ T, \quad T\in K^{\oplus p}((\tilde{D}^{-1})), \quad \Ord (T)<-p
\end{equation}
and therefore all coefficients can be made, step by step, $\tilde{D}^{d_1}$-invariant. Then we claim that 
$$
\tilde{S}_0^{-1}\hat{\Phi}(\tilde{Q})\tilde{S}_0=c_{q}\tilde{D}^q+ c_{q-k_1}\tilde{D}^{q-k_1} + \mbox{terms of lower order}, \quad c_i\in K,
$$
where $k_1= n_0 \Ord (\tilde{Q})-\Ord (\tilde{Q}_1)<p+q$. Indeed, 
$\tilde{Q}_1=\tilde{Q}^{n_0}- \varepsilon_0\partial^{p m_0}$, then $\tilde{S}_0^{-1}\hat{\Phi}(\tilde{Q}_1)\tilde{S}_0=(\tilde{S}_0^{-1}\hat{\Phi}(\tilde{Q})\tilde{S}_0)^{n_0}-\varepsilon_0^{n_0}\tilde{D}^{p m_0}$ has the highest symbol of the shape $c_{q n_0-k_1}\tilde{D}^{q n_0-k_1}$, $c_{q n_0-k_1}=\tilde{c}_1\in K^*$, whence
$$
(\tilde{S}_0^{-1}\hat{\Phi}(\tilde{Q})\tilde{S}_0)^{n_0}=\varepsilon_0^{n_0}\tilde{D}^{q n_0}+ c_{q n_0-k_1}\tilde{D}^{q n_0-k_1} + \mbox{terms of lower order}, \quad c_i\in K.
$$
Since all coefficients of $\tilde{S}_0^{-1}\hat{\Phi}(\tilde{Q})\tilde{S}_0$ are $\tilde{D}^{d_1}$-invariant, it must have the  shape as claimed above, i.e. $c_{q-k_1}=\tilde{c}_1/(\varepsilon_0^{n_0}n_0)\in K$. 

Now we can find a monic invertible operator $\tilde{S}_1\in K^{\oplus p}((\tilde{D}^{-1}))\subset E_p$  whose coefficients are {\it $\tilde{D}^{d_1}$-invariant} and such that all coefficients of the operator $\tilde{S}_1^{-1}\tilde{S}_0^{-1}\hat{\Phi}(\tilde{Q})\tilde{S}_0\tilde{S}_1$ are $\tilde{D}^{d_2}$-invariant, where $d_2=SeqGCD_1(\tilde{P},\tilde{Q})$. The proof is the same as in lemma \ref{L:conjugation in E_q} (in both cases, $[\tilde{Q},\tilde{P}]=0$ or $[\tilde{Q},\tilde{P}]=1$): 

We will find the operator $\tilde{S}_1$ as the limit of a Cauchy sequence\footnote{with respect to the topology defined by the pseudo-valuation $-\ord_{\tilde{D}}$}. Assume 
$$
\tilde{S}_0^{-1}\hat{\Phi}(\tilde{Q})\tilde{S}_0=c_{q}\tilde{D}^q+ c_{q-k_1}\tilde{D}^{q-k_1} +\sum_{i< q-k_1}q_i\tilde{D}^{i}, 
$$
where $\quad q_i\in K^{\oplus p}$ for $i\neq -p$ (and for $i=-p$ $q_i$ has the shape $c\Gamma +\tilde{q}_i$, $\tilde{q}_i \in K^{\oplus p}$), and all coefficients are $\tilde{D}^{d_1}$-invariant. Let $q_l=(q_{l,0}, \ldots , q_{l,p-1})$ be the first coefficient not commuting with $\tilde{D}^{d_2}$\footnote{Here we again consider the coefficients in the sense of formulation above, namely, if $Sdeg_A(q_l)=1$, then the shape of $q_l$ is $c\Gamma +\tilde{q}_l$, $\tilde{q}_l\in K^{\oplus p}$, and we consider $\tilde{q}_l$ instead of $q_l$. The term $c\Gamma$ is inessential for further calculations by the formula \eqref{E:gamma-congugation} given above.}. Consider the operator $\tilde{\tilde{S}}_l:=1+s_l \tilde{D}^{-l}$, where $s_l=(s_{l,0},\ldots , s_{l,p-1})$ and $[s_l, \tilde{D}^{d_1}]=0$\footnote{The last condition just means that $s_{l,i}=s_{l,i+d_1\mbox{mod $p$}}$ for all $i$. }. Then an easy direct calculation shows that
\begin{multline*}
\tilde{\tilde{S}}_l^{-1}\tilde{S}_0^{-1}\hat{\Phi}(\tilde{Q})\tilde{S}_0 \tilde{\tilde{S}}_l=c_{q}\tilde{D}^q+ c_{q-k_1}\tilde{D}^{q-k_1}+\ldots + (q_{q-k_1-l}+c_{q-k_1}(-s_l+\sigma^{q-k_1}(s_l)))\tilde{D}^{q-k_1-l}+ \\
\mbox{terms of lower order},
\end{multline*}
where $\ldots$ are the same terms of $\tilde{S}_0^{-1}\hat{\Phi}(\tilde{Q})\tilde{S}_0$ (as $\tilde{\tilde{S}}_l$ is $\tilde{D}^{d_1}$-invariant). Consider the following $d_2$ systems of linear equations: for $i=0, \ldots , d_2-1$ [mod $d_2$] set $b_i:=\sum_{k=0}^{m-1}\tilde{q}_{l,[(i+(q-k_1)k) \mbox{mod $d_1$}] }$, where $m$ is a minimal natural number such that $(q-k_1)m=0 \mbox{ mod $d_1$}$ and $\tilde{q}_l:=q_{q-k_1-l}/c_{q-k_1}$,  then the $i$-th system is (below all indices are considered modulo $d_1$ )
\begin{equation}
\label{E:invariant_conjugation}
\begin{cases}
\tilde{q}_{l,i}-s_{l,i}+s_{l,i+(q-k_1)}=b_i/m \\
\tilde{q}_{l,i+(q-k_1)}-s_{l,i+(q-k_1)}+s_{l,i+2(q-k_1)}=b_i/m \\
\ldots \\
\tilde{q}_{l,i+(m-1)(q-k_1)}-s_{l,i+(m-1)(q-k_1)}+s_{l,i}=b_i/m
\end{cases}
\end{equation} 
This system is solvable and give explicit formulae for unknown variables $s_{l,j}$: for $r=1,\ldots ,(m-1)$ from the first $(m-1)$ equations we get 
\begin{equation}
\label{E:solution_invar_conjugation}
s_{l,i+r(q-k_1)}=r b_i/m+s_{l,i}-(\sum_{j=0}^{r-1} \tilde{q}_{l,i+j(q-k_1)})
\end{equation}
($s_{l,i}$ is a free parameter), and the last equation becomes an identity under substitution of these values. Taking such $s_l$ we get
$$
\tilde{\tilde{q}}_l:= \tilde{q}_l-s_l+\sigma^{(q-k_1)}(s_l)= (b_0/m, \ldots , b_{d_2-1}/m, b_0/m, \ldots , b_{d_2-1}/m, \ldots )
$$ 
which obviously commutes with $\tilde{D}^{d_2}$. Note that $s_l$ is uniquely determined if we put the free parameters $s_{l,i}=0$ for $i=0, \ldots , d_2-1$. 

It's easy to see that the system $\{\prod_{i=1}^j \tilde{\tilde{S}}_i\}$, $j\ge 1$ is a Cauchy system in $E_p$, and its limit $\tilde{S}_1$ is the needed operator. Note that all coefficients $c_i$ for $i\le q-k_1-k_2$ must be constant, because
$$
\tilde{S}_1^{-1}\tilde{S}_0^{-1}\hat{\Phi}(\tilde{Q}_2)\tilde{S}_0\tilde{S}_1= \tilde{S}_1^{-1}\tilde{S}_0^{-1}(((\hat{\Phi}(\tilde{Q}))^{n_0}-\varepsilon_0^{n_0}\tilde{D}^{p m_0})^{n_1})\tilde{S}_0\tilde{S}_1-\varepsilon_1^{n_1}\tilde{D}^{p m_1}
$$
has the highest symbol of the shape $\tilde{c}_2\tilde{D}^{\Ord \tilde{Q}_2}$ (if there was a non-constant coefficient $c_i$ with $i\le q-k_1-k_2$, this would be impossible). 

\smallskip

Continuing this line of reasoning, after $I$ steps we'll find $I$ operators $\tilde{S}_0, \tilde{S}_1, \ldots , \tilde{S}_{I-1}$ such that the coefficients of the operator $\tilde{S}_{I-1}^{-1}\ldots \tilde{S}_0^{-1}\hat{\Phi}(\tilde{Q})\tilde{S}_0\ldots \tilde{S}_{I-1}$ are $\tilde{D}^{\tilde{d}}$-invariant, and 
\begin{equation}
\label{E:form_of_Q}
\tilde{S}_{I-1}^{-1}\ldots \tilde{S}_0^{-1}\hat{\Phi}(\tilde{Q})\tilde{S}_0\ldots \tilde{S}_{I-1}=c_{q}\tilde{D}^q+ c_{q-k_1}\tilde{D}^{q-k_1} +\ldots + c_{q-k_1-\ldots -k_I}\tilde{D}^{q-k_1-\ldots -k_I}+\sum_{i<q-k_1-\ldots -k_I}q_i\tilde{D}^{i}, 
\end{equation}
where $\quad c_i\in K$, $\quad q_i\in K^{\oplus p}$ for $i\neq -p$ (and for $i=-p$ $q_i$ has a shape $c\Gamma+ \tilde{q}_i$), $\quad [q_i,\tilde{D}^{\tilde{d}}]=0$ (and $\quad [\tilde{q}_i,\tilde{D}^{\tilde{d}}]=0$), and $k_i=n_{i-1}\Ord (\tilde{Q}_{i-1})-\Ord (\tilde{Q}_i)$.

Moreover, we can find these operators $\tilde{S}_k$ such that their coefficients satisfy additional conditions $s_{l,i}=0$ for $i=0, \ldots ,d_k-1$ and all $l$, so the product $\tilde{S}:=\tilde{S}_0\ldots \tilde{S}_{I-1}$ will satisfy the claimed condition on its coefficients from lemma. 

\begin{rem}
\label{R:relation_between_coeff} 

The relation between coefficients $\tilde{c}_i$ of the sequence $(\Ord (\tilde{Q}_i), \tilde{c}_i)$ and coefficients $c_i$ of the series $\tilde{S}^{-1}\hat{\Phi}(\tilde{Q})\tilde{S}$ is non-trivial, but let's point out some observations. 

First, obviously, the coefficients $\tilde{c}_i$ can be uniquely reconstructed by the coefficients $c_i$ if we know the exponents $n_i$ of the sequence $F_i$. Vice versa, the coefficients $c_i$ can be uniquely reconstructed by the coefficients $\tilde{c}_i$ if we know  the sequence $F_i$ by solving recurrent system of equations (by comparing the corresponding highest coefficients of expressions $F_i(\tilde{D}^p, \tilde{S}^{-1}\hat{\Phi}(\tilde{Q})\tilde{S})$ with  $\tilde{c}_i$). 

Second, by solving the system of equations on $c_i$, non-zero values of $c_i$ with $i\le q-k_1-\ldots -k_l$ may appear only at exponents belonging to the lattice $\Lambda \subset \dz$ generated by $p$ and $\Ord (\tilde{Q}_j)$ with $j\le l$. 

Third, the last observation shows that if $l$ is such that $SeqGCD_l(\tilde{P},\tilde{Q})< SeqGCD_{l-1}(\tilde{P},\tilde{Q})$, then 
$$
GCD\{p,q, q-k_1, \ldots, q-k_1-\ldots - k_l\}< GCD\{p,q, q- k_1, \ldots, q-k_1-\ldots -k_{l-1}\},
$$
$q-k_1-\ldots - k_l$ is the first exponent of the series from lemma with such a property, and 
\begin{equation}
\label{E:formula_for_coeff}
c_{q-k_1-\ldots -k_l}=\tilde{c}_{q-k_1-\ldots -k_l}/(n_0\cdots n_{l-1}\varepsilon_0^{n_0-1}\cdots \varepsilon_{l-1}^{n_{l-1}-1}).
\end{equation}
\end{rem}

By construction and remark \ref{R:relation_between_coeff} we see that coefficients $c_i$ don't depend on the choice of a sequence $F_i$ with the same end of sequential GCD sequence. 

2. To prove the last assertion of lemma let's note that the distance is, by definition, a sum of distances between old and new head during the head-chopping process. In case of $[\tilde{Q},\tilde{P}]=1$ this process is restricted by the distance between the head of $\tilde{Q}$ and its tail $\tilde{Q}_{-p}$, because the highest symbol of the tail can not be of the shape $\tilde{c}\partial^k$. 

The last assertion is therefore obvious in view of remark \ref{R:relation_between_coeff}.

\end{proof}

Now we are ready to prove the first theorem relating a solution to the string equation with a pair of commuting ordinary differential operators generating a ring of rank one. We will need a partial analogue of definition \ref{D:distance}

\begin{Def}
\label{D:BC-distance}
Let $f(X,Y)\in K[X,Y]$ be an irreducible polynomial of Burchnall-Chaundy type (cf. remark \ref{R:classif_ODO}),
$$
f(X,Y) = (\alpha X^n + Y^m)^d+\ldots ,
$$
where $\ldots$ mean terms of lower weighted degree, $0\neq\alpha\in K$, $(m,n)=1$, $d\ge 1$. Let $Y=z^{-nd}$, $X=c_{-md}z^{-md}+ c_{n_1}z^{n_1}+c_{n_2}z^{n_2}+\ldots \in K((z))$ be a solution (a Puiseux series) of equation $f(X,Y)=0$\footnote{ Thus, by classification theorem \ref{T:classif2}, there exist a pair of commuting differential operators $P,Q\in D_1$, such that $\rk K[P,Q]=1$ and $K[P,Q]\simeq K[X,Y]/(f)$. By Wilson's theorem (see remark \ref{R:classif_ODO}) their BC-polynomial $\partial Res(P-X, Q-Y)$ is irreducible, therefore is proportional to $f$, whence 
$\ord (P)=p=nd$, $\ord (Q)=q=md$.   

Then, if $S_{cl}$ is a classical Schur operator such that $S_{cl}^{-1}PS_{cl}=\partial^{nd}$, the operator $S_{cl}^{-1}QS_{cl}$ has constant coefficients and has the shape of $X$ (after formal replacement $\partial \mapsto z^{-1}$). Here $z$ is a local parameter at the point "at infinity" of the one-point compactification of the curve $f(X,Y)=0$.}. Assume $n_l$ is the first exponent such that $GCD(d, n_1, \ldots ,n_l)=1$. 

Then we'll call the number $nd-n_l$ as {\it the distance of $f(X,Y)$.} For a solution $Y,X$ other notions from definition \ref{D:distance} are defined (i.e. the sequential GCD and the sequence) in the same way. 
\end{Def}

\begin{rem}
\label{R:distinguished_coefficient}
Note that for fixed $Y=z^{-nd}$ all other solutions of this equation can be obtained by the scale transform $z\mapsto \xi z$, where $\xi$ is a $nd$-th primitive root of unity. Hence $c_{nd}=0$. Indeed, the sum of roots of the polynomial $f(X, z^{-nd})$, $\sum_{i=1}^{nd} X|_{z\mapsto z\xi^i}$, must be a polynomial in $z^{-1}$, hence $nd c_{nd}=0$. 

\end{rem}

\begin{theorem}[cutting the tail]
\label{T:string_equation}
Let $P,Q\in D_1$ be two differential operators such that $\ord (P)=\Ord (P)=p=nd$, $\sigma (P)=\partial^{\Ord (P)}$, $\ord (Q)=\Ord (Q)=q=md$, where $(m,n)=1$, and $[Q,P]=1$. Assume that there exists a sequence of operators $Q_i:=F_i(P,Q)$, $i=0, \ldots ,I$ like in theorem \ref{T:DC-pairs} (i.e. defined by recursion in the same way) such that $\ord (Q_i)= \Ord (Q_i)$ for all $i$. Assume $SeqGCD_I(P,Q)=1$ and $SeqGCD_{I-1}(P,Q)>1$. 
Let $\tilde{Q}=S^{-1}QS$ be a normal form of $Q$ with respect to $P$, where $S$ is a monic Schur operator (see remark \ref{R:normal_forms}). 

Then the pair $(\partial^p, \tilde{Q}-\tilde{Q}_{-p})$ is a normal form of a commuting pair of formally elliptic differential operators $P',Q'$, where $P'$ is monic and normalized, $\ord (P')=p$, $\ord (Q')=q$, such that $rk (K[P',Q'])=1$ (i.e. $\tilde{Q}-\tilde{Q}_{-p}$ is a normal form of $Q'$ with respect to $P'$). 

Moreover, this pair satisfies conditions from lemma \ref{L:invariant_conjugation}, it has the same sequence as $(P,Q)$ and its affine spectral curve is given by a BC-polynomial from definition \ref{D:BC-distance} of distance $\le p+q-1$.
\end{theorem}

\begin{proof}
By lemma \ref{L:invariant_conjugation} there exists  a monic invertible operator $\tilde{S}\in K^{\oplus p}((\tilde{D}^{-1}))\subset E_p$ with $\ord_{\tilde{D}}(\tilde{S})=0$ such that 
$$
\tilde{S}^{-1}\hat{\Phi}(\tilde{Q})\tilde{S}=c_{q}\tilde{D}^q+ c_{q-k_1}\tilde{D}^{q-k_1} +\ldots + c_{q-k_1-\ldots -k_I}\tilde{D}^{q-k_1-\ldots -k_I}+\sum_{i< q-k_1-\ldots -k_I}q_i\tilde{D}^{i}\in K[x]((\tilde{D}^{-1})),
$$
and $GCD(p,q,q-k_1, \ldots ,q-k_1-\ldots -k_I)=1$. 
Since $Dist(P,Q)<p+q$, we have 
\begin{multline*}
\tilde{S}^{-1}\hat{\Phi}(\tilde{Q})\tilde{S}=\tilde{S}^{-1}\hat{\Phi}(\tilde{Q}-\tilde{Q}_{-p})\tilde{S}+\tilde{S}^{-1}\hat{\Phi}(\tilde{Q}_{-p})\tilde{S}=\\
c_{q}\tilde{D}^q+ c_{q-k_1}\tilde{D}^{q-k_1} +\ldots + c_{q-k_1-\ldots -k_I}\tilde{D}^{q-k_1-\ldots -k_I}+\sum_{i< q-k_1-\ldots -k_I}\tilde{q}_i\tilde{D}^{i}+\tilde{S}^{-1}\hat{\Phi}(\tilde{Q}_{-p})\tilde{S},
\end{multline*}
where $\tilde{q}_i\in K^{\oplus p}$. Therefore we can find, following the proof of lemma \ref{L:invariant_conjugation}, another monic  operator $\tilde{S}'\in K^{\oplus p}((\tilde{D}^{-1}))\subset E_p$ with $\ord_{\tilde{D}}(\tilde{S}')=0$ such that 
\begin{equation}
\label{E:final_conjug_form}
(\tilde{S}')^{-1}\hat{\Phi}(\tilde{Q}-\tilde{Q}_{-p})\tilde{S}'= c_{q}\tilde{D}^q+ c_{q-k_1}\tilde{D}^{q-k_1} +\ldots + c_{q-k_1-\ldots -k_I}\tilde{D}^{q-k_1-\ldots -k_I}+\ldots 
\in K((\tilde{D}^{-1})).
\end{equation}
Since both operators $\tilde{S}, \tilde{S}'$ are monic, the  pair $(\partial^p, \tilde{Q}-\tilde{Q}_{-p})$ satisfies conditions from lemma \ref{L:invariant_conjugation} and has the same sequence as $(P,Q)$.

The rest of the proof is similar to \cite[Th.3.3]{GZ1}. We define the space $W':= F'\cdot (\tilde{S}\tilde{S}')\subset K((\tilde{D}^{-1}))$, where $F'=K[\tilde{D}]$. Set $B':=\hat{\Phi}(K[\partial^p, \tilde{Q}-\tilde{Q}_{-p}])$, and $A':=(\tilde{S}')^{-1}\tilde{S}^{-1}B'\tilde{S}\tilde{S}'\subset K((\tilde{D}^{-1}))$. Note that $(W',A')$ is an embedded Schur pair of {\it rank one}\footnote{Since $K[\partial^p, \tilde{Q}-\tilde{Q}_{-p}]\in C(\partial^p)\subset \hat{D}_1^{sym}$, we have $B'\subset \hat{\Phi}(C(\partial^p))$ and consists of Laurent polynomials whose coefficients at negative powers of $\tilde{D}$ are not constants (see remark \ref{R:vector_form}). Since $\tilde{S}',\tilde{S}$ are monic and $A'\subset K((\tilde{D}^{-1}))$, it implies that $A'\cap K[[\tilde{D}^{-1}]]=K$. The rank is one because $SeqGCD_I(P,Q)=1$. The operators $\partial^p$, $\tilde{Q}-\tilde{Q}_{-p}$ are algebraically dependent, and the polynomial $\det (\psi (\hat{\Phi}(\tilde{Q}-\tilde{Q}_{-p}))-X)\in K[X,\tilde{D}^p]$ is a BC polynomial (after replacement $\tilde{D}^p\mapsto Y$ (see \cite[Rem.3.3]{GZ1}).  }. Therefore, by theorem \ref{T:classif_Schur_pairs} it corresponds to a normalized ring $B\subset D_1$ of commuting differential operators of rank one (see remark \ref{R:classif_ODO}), $B\simeq B'$, so there are two formally elliptic commuting operators $P',Q'\in D_1$ of the same orders $p,q$ such that $B=K[P',Q']$. By theorem \ref{T:classif2} the ring $K[P',Q']$ determines the spectral curve $C$ -- the one-point compactification (with the help of the point $p$) of the affine spectral curve $C_0$ determined by the above mentioned equation $f(X,Y)=0$ (by Wilson's theorem the polynomial $f$ is irreducible, see remark \ref{R:classif_ODO}), and also a spectral sheaf $\cf$ - a torsion free sheaf of rank one. By \cite[Th. 3.3]{GZ1} the corresponding (normalized) normal form of $B$ is uniquely determined by this datum, and the correspondence is established via the embedded Schur pair of $B$. Since $B'$ and $B$ have the same embedded Schur pairs, the operator  
$\tilde{Q}-\tilde{Q}_{-p}$ is a normal form of $Q'$ with respect to $P'$ (see also remark below). 

The last assertion follows immediately from formula \eqref{E:final_conjug_form}.

\end{proof}

\begin{rem}
\label{R:spectral_data_for_string}
Note that the conditions of theorem \ref{T:string_equation} are satisfied for a DC pair $P,Q$ of subrectangular pair from theorem \ref{T:DC-pairs}, and by this theorem $Dist(P,Q)=p+q-1$.  

It can be shown that if $\varphi \in \Aut_K (A_1)$, and $P=\varphi (x)$, $Q=\varphi (\partial )$, then these conditions hold as well. Moreover, in this case $Dist(P,Q)<p+q-1$, and the operators $P',Q'\in K[\partial ]$, i.e. the corresponding spectral curve $C$ from theorem \ref{T:classif2} is rational.

However, it is an interesting question whether any solution of the string equation from $D_1$ satisfy these conditions. 

 We'll call the datum $(C,p,\cf ,z)$ as the {\it spectral datum} of a solution $P,Q$ to the string equation. Here $z$ is the local parameter at $p$ mentioned above. The isomorphism class of this datum uniquely determines the equivalence class $[K[P',Q']]$, cf. \cite[Rem. 3.5]{GZ1}. 
\end{rem}

For any solution of the string equation $P,Q\in D_1$ satisfying conditions of theorem \ref{T:string_equation} we can construct another solution $\bar{P}, \bar{Q}\in D_1$ such that $\bar{P}$ is monic and normalized. We'll call such pair a {\it normalized solution}. If the orders of operators $P,Q$ are fixed, such a pair is uniquely determined in the conjugacy class  by an invertible function from $K[[x]]^*$. Analogously, for any pair of commuting operators $P',Q'$ as above we can construct a normalized pair (cf. remark \ref{R:classif_ODO}). 

The subring $B'=K[\partial^p, \tilde{Q}']\subset \hat{D}_1^{sym}$ generated by a normal form $\tilde{Q}'$ of $Q'$ with respect to $P'$, can be also normalized in its conjugacy class (recall that a normal form is defined up to conjugation by an invertible element in $C(\partial^p)$), e.g. by choosing a special set of generators as in definition \ref{D:normalised_normal_form} and applying lemma \ref{L:normalisation}. We'll call the corresponding pair $(\partial^p, \tilde{Q}')$ as a {\it normalized normal pair}. Analogously, since the normal form of the normalized solution $(P,Q)$ is uniquely defined, by lemma \ref{L:Q_p} and theorem \ref{T:string_equation}, by the normal form of the corresponding normalized pair $(P',Q')$, we'll call the normal form $(\partial^p, \tilde{Q})$ of a normalized solution  $(P,Q)$ as {\it normalized normal string pair} if the corresponding pair $(\partial^p, \tilde{Q}-\tilde{Q}_{-p})$ (from theorem \ref{T:string_equation}) is the normalized normal pair. More precisely:

\begin{Def}
\label{D:normalized_normal_pair}
A pair of elements $(\partial^p, \tilde{Q}')\in C(\partial^p)\subset \hat{D}_1^{sym}$ is called a {\it normal pair} of order $(p,q)$, where $q=\Ord (\tilde{Q}')$, if it satisfies conditions of lemma \ref{L:invariant_conjugation}, $SeqGCD_I(\partial^p, \tilde{Q}')=1$, $SeqGCD_{I-1}(\partial^p, \tilde{Q}')>1$ and $Dist(\partial^p, \tilde{Q}')<p+q$.\footnote{ 
In this case the same arguments as in the proof of theorem \ref{T:string_equation} show that the polynomial  $f(X,Y):=\det (\psi (\hat{\Phi}(\tilde{Q}'))-X)|_{\tilde{D}^p \mapsto Y}\in K[X,Y]$ is irreducible of distance $<p+q$ (by Wilson's theorem, because the corresponding ring of commuting differential operators has rank one).} 

This pair is called a {\it normalized normal pair} of order $(p,q)$ if the ring $B':=K[\partial^p, \tilde{Q}']$ is a normalized normal form of rank one with respect to some set of generators (see definition \ref{D:normalised_normal_form}). 
\end{Def}

\begin{Def}
\label{D:normalized_string_pair}
A pair of elements $(\partial^p, \tilde{Q})\in  \hat{D}_1^{sym}$ is called a {\it normal string pair}  of order $(p,q)$, where $q=\Ord (\tilde{Q})$, if  $[\tilde{Q},\partial^p]=1$ and the pair 
$(\partial^p, \tilde{Q}-\tilde{Q}_{-p})$ is a normal pair of order $(p,q)$.

This pair is called a {\it normalized normal string pair} of order $(p,q)$ if the pair $(\partial^p, \tilde{Q}-\tilde{Q}_{-p})$ is a normalized normal pair of order $(p,q)$.
\end{Def}

The normalisation depends on the choice of two specific generators in the ring $B'$. Let's fix such generators  $\partial^p$, $F_I(\partial^p, \tilde{Q}')$, where $F_i$, $i=0, \ldots ,I$ is the sequence of polynomials like in theorem \ref{T:string_equation} (it exists by definition of distance), i.e. $F_I(\partial^p, \tilde{Q}')$ is the first operator in this sequence of coprime order with $p$. 

Let's note that the normalisation can be made with the help of monic zeroth order invertible operator $S$ from the centralizer $C(\partial^p)$, according to the proof of lemma \ref{L:normalisation} in (\cite[Lem 3.4]{GZ1}). So, this is an operator satisfying condition $A_p(0)$ according to the description of $C(\partial^p)$ given in proposition \ref{T:A.Z 7.1}. Note also that any such operator commutes with the tail $\tilde{Q}_{-p}$. For, $[S^{-1}\tilde{Q}_{-p}S, \partial^p]=1$, and $\Ord (S^{-1}\tilde{Q}_{-p}S)=-p$. So, by proposition \ref{T:A.Z 7.1} and lemma \ref{L:Q_p} $S^{-1}\tilde{Q}_{-p}S= \tilde{Q}_{-p}$.

\begin{theorem}
\label{T:one-to-one-correspondence}
There is a one-to-one correspondence between
\begin{enumerate}
\item
Normalized solutions $P, Q\in D_1$ to the string equation such that  $\ord (P)=p=nd$, $\ord (Q)=q=md$, where $(m,n)=1$, and conditions of theorem \ref{T:string_equation} are satisfied;
\item
Normalized  normal  string pairs $(\partial^p, \tilde{Q})\in  \hat{D}_1^{sym}$ of order $p=nd ,q=md$ with respect to the fixed normalisation (e.g. as above);
\end{enumerate}
with the same sequences. This correspondence is established with the help of conjugation by a unique monic invertible Schur operator $S\in \hat{D}_1^{sym}$ of order zero satisfying condition $A_p(0)$. 
\end{theorem}

\begin{proof}
If $P,Q$ is a normalized solution from item 1, then by proposition \ref{P:Si} there exists a monic Schur operator satisfying condition $A_p(0)$ such that $S^{-1}PS=\partial^p$. Then,  
clearly, the normal form $(\partial^p, \tilde{Q})$ with respect to $P$ (and $S$) is a normal string pair with the same sequence by theorem \ref{T:string_equation}. Taking the normalisation, we get a normalized  normal  string pair. Since the normalisation can be made with the help of monic zeroth order invertible operator from the centralizer $C(\partial^p)$, therefore, their product is a monic zeroth order invertible operator  satisfying condition $A_p(0)$ by lemma \ref{L:conditionA}. By lemma \ref{L:normalisation} the resulting Schur operator is uniquely defined. 

Vice versa, if we have a normalized  normal  string pair $(\partial^p, \tilde{Q})\in  \hat{D}_1^{sym}$ of order $p=nd ,q=md$, we can construct a unique normalised solution of order $p,q$ from item 1 as follows. 

{\it Step} 1. {\it First, we claim that there is a {\it unique} invertible operator $\tilde{S}\in K^{\oplus p}((\tilde{D}^{-1}))$ of order zero from lemma \ref{L:invariant_conjugation} such that $\tilde{S}^{-1}(\hat{\Phi}(\tilde{Q}))\tilde{S}\in K[x]((\tilde{D}^{-1}))$ and for all $i$ the coefficients $s_i$ of $\tilde{S}$ satisfy the condition $s_{i,0}=0$.} 

We know from lemma \ref{L:invariant_conjugation} that such an operator exists (as $SeqGCD_I(\partial^p, \tilde{Q})=1$ according to definition), so we need only to prove its uniqueness. 

Assume there are two such operators, $\tilde{S}_1, \tilde{S}_2$, and $\bar{Q}_1:=\tilde{S}_1^{-1}(\hat{\Phi}(\tilde{Q}))\tilde{S}_1$, $\bar{Q}_2:=\tilde{S}_2^{-1}(\hat{\Phi}(\tilde{Q}))\tilde{S}_2$, $\bar{Q}_i\in K[x]((\tilde{D}^{-1}))$. Then $\tilde{S}_2^{-1}\tilde{S}_1\bar{Q}_1\tilde{S}_1^{-1}\tilde{S}_2=\bar{Q}_2$. Then, if $F(X,Y)\in K[X,Y]$ is any polynomial,  we have also
\begin{equation}
\label{E:one-to-one}
\tilde{S}_2^{-1}\tilde{S}_1F(\tilde{D}^p, \bar{Q}_1)\tilde{S}_1^{-1}\tilde{S}_2=F(\tilde{D}^p, \bar{Q}_2). 
\end{equation}
Note that $F(\tilde{D}^p, \bar{Q}_j)\in K[x]((\tilde{D}^{-1}))$ and there exists a polynomial $F$ such that the operator $F(\tilde{D}^p, \bar{Q}_j)$ have order $M$ coprime with $p$. 

An easy direct calculation shows that the operators $\tilde{S}_i^{-1}, \tilde{S}_1^{-1}\tilde{S}_2$ satisfy the same property on coefficients (i.e. the $0$th components are zero). As in the proof of lemma \ref{L:invariant_conjugation} we can represent the operator $\tilde{S}_1^{-1}\tilde{S}_2$ as a product of elementary operators $\tilde{S}_l$. Assume $l$ is a minimal number such that $\tilde{S}_l$ is non-trivial. Then the system of equations on $l$-th coefficients of operators in \eqref{E:one-to-one} will look like system 
\eqref{E:invariant_conjugation}  with $i=0$, $k_1=0$, $m=p$ where $q_{l}$ is a constant, since the operators $\bar{Q}_i$ have constant coefficients (i.e. all coefficients of polynomials in $x$ belong to $K\subset K^{\oplus p}$), namely:
 \begin{equation}
\label{E:invariant_conjugation1}
\begin{cases}
q_{l}-s_{l,0}+s_{l,M}=b_l \\
q_{l}-s_{l,M}+s_{l,2M}=b_l \\
\ldots \\
q_{l}-s_{l,(p-1)M}+s_{l,0}=b_l
\end{cases}
\end{equation} 
(here $q_l$ and $b_l$ are corresponding {\it constant} coefficients of $F(\tilde{D}^p, \bar{Q}_j)$). But, since $s_{l,0}=0$, this is possible only if all $s_{l,i}$ are zero and $q_l=b_l$. Therefore, $\tilde{S}_1=\tilde{S}_2$. 

\begin{rem}
\label{R:Step2,3}
Any monic invertible zeroth order operator $S\in K^{\oplus p}((\tilde{D}^{-1}))$ from lemma \ref{L:invariant_conjugation} can be decomposed into a product $\tilde{S}\bar{S}$, where $\tilde{S}$ is the operator as above, i.e. its coefficients satisfy the condition $s_{i,0}=0$, and $\bar{S}\in K((\tilde{D}^{-1}))$. 

This is almost evident: $\bar{S}$ can be constructed again as a product of elementary operators $\bar{S}_l=1+\bar{s}_l\tilde{D}^{-i}$, $\bar{s}_l\in K$, where each next coefficient $\bar{s}_l$ is determined by the first non-zero component of coefficients of the operator $\tilde{S}\bar{S}_1\ldots \bar{S}_{l-1}$.  

Next, let's note that if we fix the operator $L:=\tilde{S}^{-1}(\hat{\Phi}(\tilde{Q}))\tilde{S}\in K[x]((\tilde{D}^{-1}))$ in Step 1, there is a unique monic zeroth order operator $S\in K^{\oplus p}((\tilde{D}^{-1}))$ such that $S^{-1}(\hat{\Phi}(\tilde{Q}))S=L$. 

Indeed, we know that $S=\tilde{S}\bar{S}$, and $L$ is a sum of the term $\Gamma \tilde{D}^{-p}/p= x \tilde{D}^{-p+1}/p$ and a series with constant coefficients (which commutes with $\bar{S}$). From formula \eqref{E:gamma-congugation} it follows that $\bar{S}$ commutes with the first term only if $\bar{S}=1$. 
\end{rem}

{\it Step} 2. {\it Let's prove that there exists a unique monic operator $S\in E_p$ of order zero satisfying condition $A_p(0)$ such that $T:=\tilde{S}S$, where $\tilde{S}$ is the unique operator constructed above in Step 1., is a monic zeroth order operator such that all its homogeneous components don't depend on $A_i$ and $B_j$ and $Sdeg_A(T_{-i})<i$ for all $i>0$\footnote{ (clearly, all components of $T$ will be HCPs by lemma \ref{C:HCPC(k)}, not depending on $B_j$ since components of $\tilde{S},S$ don't depend on $B_j$). Here we adopt definition \ref{D:conditionA} to our situation in the following sense: condition 2 from definition (totally freeness) is replaced by an equivalent system of equations \eqref{E:system_on_B_j} on coefficients, all other conditions are the same. Since $\hat{\Phi}$ is an injective homomorphism, any operator in $E_p$ satisfying condition $A_p(0)$ in this sense is an image of an operator satisfying condition $A_p(0)$, and vice versa}.} 

In view of corollary \ref{C:Si^(-1) no B} this claim is equivalent to the following: there is a unique monic zeroth order operator $T$ with all its homogeneous components don't depending on $A_i$ and $B_j$ and $Sdeg_A(T_{-i})<i$ for all $i>0$,  such that $\tilde{S}T$ is a monic zeroth order operator satisfying condition $A_p(0)$. 

The proof is by induction on the index of coefficients of the operator $S=\tilde{S}T$. Clearly, it is monic and all its homogeneous components are HCP. For $s_{-1}$ we have $s_{-1}=\tilde{s}_{-1}+t_{-1}$, whence $Sdeg_A(S_{-1})=0$ iff $Sdeg_A(T_{-1})=0$. So, $s_{-1}, t_{-1}\in K^{\oplus p}$ and  $s_{-1}$ is totally free of $B_j$ iff $s_{-1,0}=0$, because $s_{-1}=\Phi (\sum_{i=0}^{p-1} a_iA_i)$ and $s_{-1,0}=\sum_{i=0}^{p-1} a_i$ and the system \eqref{E:system_on_B_j} reduces to just one equation
$$
\sum_{i=0}^{p-1}a_i=0.
$$
Analogously, $t_{-1}$ is totally free of $B_j$ iff $t_{-1,0}=0$. Since $\tilde{s}_{-1,0}=0$, we get $s_{-1}$ is totally free of $B_j$ iff $t_{-1}$ is totally free of $B_j$. Since $t_{-1}$ must not contain $A_i$, there is a unique value of $t_{-1}$ with all these properties, $t_{-1}=0$. 

Now suppose we have proved that there are uniquely determined coefficients $t_{-1}, \ldots , t_{-l+1}$ such that $s_{-1}, \ldots , s_{-l+1}$ satisfy condition $A_p(0)$. Now 
\begin{equation}
\label{E:s_l}
s_{-l}=t_{-l}+\tilde{s}_{-l}+\sum_{i=1}^{l-1}\tilde{s}_{-i}\sigma^i(t_{-l+i}). 
\end{equation}
We need to show that there is a unique $t_{-l}$ with $Sdeg_A(T_{-l})<l$ and  $t_{-l}$ does not contain $A_i$ and $B_j$ such that $s_{-l}$ is totally free of $B_j$ and $Sdeg_A(S_{-l})<l$. So, $t_{-l}$ should be of shape $t_{-l}=\sum_{m=0}^{l-1}a_{m,0}\Gamma_m$. Clearly, $Sdeg_A(s_{-l})<l$ if $Sdeg_A(t_{-l})<l$, by induction hypothesis. The system \eqref{E:system_on_B_j} for the coefficients of $s_{-l}$ has the shape
$$
\sum_{m=0}^{l-1}(j-1)^m a_{m,0}=*
$$ 
for $1\le j\le l$, where $*$ means the known coefficients coming from the right hand side of \eqref{E:s_l}. For $j=1$ this system implies $a_{0,0}=0$, because $\tilde{s}_{-i,0}=0$ for all $i$; so it reduces to analogous system
$$
\sum_{m=1}^{l-1}(j-1)^m a_{m,0}=*
$$
for $2\le j\le l$. Now observe that the matrix of this linear system has a shape of the Vandermonde matrix with powers of natural numbers, i.e. it is invertible, and there is a unique solution $a_{m,0}$ for any values $*$. Hence $t_l$ is uniquely determined. 

Finally, note that $(\tilde{S}S)^{-1}=S^{-1}\tilde{S}^{-1}$ is also a monic zeroth order operator such that all its homogeneous components don't contain $A_i$ (this easily follows from lemmas \ref{L:obvious} and \ref{L:HCPC}), and $\tilde{S}^{-1}$ is an operator with the same properties as $\tilde{S}$. So, analogous statement for the product $S\tilde{S}$ holds. 

\begin{rem}
\label{R:small_remark}
If we take instead of $\tilde{S}$ any operator $\tilde{S}\bar{S}$ from Step 2., then the same arguments work, and there exists a unique monic zeroth order operator $S$ satisfying condition $A_p(0)$ such that $S\tilde{S}\bar{S}$ satisfy properties of this step. Clearly, we can just take the same $S$ as for $\tilde{S}$ (as $\bar{S}$ has constant coefficients).  So, the claim of this step holds for {\it any} operator $\tilde{S}$ from lemma \ref{L:invariant_conjugation}.  

Note that $T$ belongs to the image of the embedding $E\subset E_p$, i.e. it is a pseudo-differential operator; besides our calculations imply $1\cdot T=T(0)=1$, i.e. $T$ is normalised. Since $F'\cdot S-F'$, we get from the Sato theorem \ref{T:Sato} that $T=S_{cl}^{-1}$, where $S_{cl}$ is the classical Sato operator for the space $W=F'\cdot\tilde{S}$, i.e. $W=F'\cdot S_{cl}$. It is not difficult to see that $S_{cl}$ is also normalised. 
\end{rem}

{\it Step} 3. Put $\tilde{Q}':=\tilde{S}^{-1}\hat{\Phi}(\tilde{Q})\tilde{S}$, where $\tilde{S}$ is from Step 1.  Let $S$ be an operator from Step 4. such that all homogeneous components of $T:=S\tilde{S}$ don't contain $A_i$ and $B_j$ and $Sdeg_A(T_{-i})<i$ for all $i>0$. As it was noticed in 4., $S$ is the image of some uniquely defined operator $S'\in \hat{D}_1^{sym}$ with the same properties under the map $\hat{\Phi}$. 
{\it We claim that $P:=S'\partial^p(S')^{-1}, Q:=S'(\tilde{Q})(S')^{-1} \in D_1$ is the normalised solution from item 1. } 

Indeed, $S\hat{\Phi}(\tilde{Q})S^{-1}=(S\tilde{S})\tilde{Q}'(\tilde{S}^{-1}S^{-1})=T\tilde{Q}'T^{-1}$.  Note that all homogeneous components of $T^{-1}$  don't contain $A_i$ and $B_j$. Therefore, the product $T\tilde{Q}'T^{-1}$ satisfy the same property.  Finally, $S\hat{\Phi}(\tilde{Q})S^{-1}$ satisfy condition $A_p(0)$, because $S,S^{-1}, \tilde{Q}$ satisfy it (see example \ref{Ex:Schur_op}, lemma  \ref{L:conditionA}, proposition \ref{C:P'}). Then by theorem \ref{T: P in hat(D) is a dif_op} $S\hat{\Phi}(\tilde{Q})S^{-1}$ is the image of the differential operator $Q$ under the map $\hat{\Phi}$. It is monic, since $S$ is monic. The same arguments work for $P$; moreover, $P$ is normalized\footnote{So, $T$ is a classical (normalised) Schur operator.}. 

At last, note that the map from item 1 to item 2 and vice versa are inverse to each other. In one way because of uniqueness of the operator $S'$, in another way because of lemma \ref{L:normalisation}. 
 
\end{proof}

\begin{rem}
\label{R:one-to-one}
There is a natural question what can be said about the correspondence from theorem \ref{T:one-to-one-correspondence} if we allow conjugation by an arbitrary invertible operator. 

By propositions \ref{P:Si} and \ref{T:A.Z 7.1} any such operator $S\in \hat{D}_1^{sym}$ with the property $S^{-1}PS=\partial^p$ is a product of a monic zeroth order operator satisfying condition $A_p(0)$ and an invertible operator from the centralizer $C(\partial^p)$. If we demand the normal form $(\partial^p, \tilde{Q})$ satisfies conditions of lemma \ref{L:invariant_conjugation}, then it can be shown by arguments from the proof of \cite[Lem. 3.4]{GZ1}  that $\Ord (S)=0$ (the main reason is that the sequence $\tilde{Q}_i$ contains an operator of the order coprime with $p$). Any invertible operator of order zero from $C(\partial^p)$ is a product of an invertible homogeneous operator of order 0, say $S_0$, and a monic zeroth order operator satisfying condition $A_p(0)$. Again, conditions of lemma \ref{L:invariant_conjugation} imply that  $S_0$ should be of special form. Namely, conjugation by $S_0$ should be a scale transform $\partial \mapsto \xi_{nd}^l\partial$, where $\xi_{nd}$ is a primitive $nd$-th  root of unity. Of course, such a scale transform does not preserve the sequence of $P,Q$. 
\end{rem}

\begin{rem}
\label{R:one-to-one1}
Analogous theorem (just with the same formulation) holds also for a normalized pair of commuting operators $(P',Q')$ of orders $p,q$ and satisfying conditions from remark \ref{R:application_of_lemma_inv_conj} and theorem \ref{T:string_equation}, and normalized normal pairs of order $(p,q)$:

\begin{theorem}
\label{T:one-to-one-correspondence1}
There is a one-to-one correspondence between
\begin{enumerate}
\item
Normalized pairs $P', Q'\in D_1$ of commuting operators such that  $\ord (P)=p=nd$, $\ord (Q)=q=md$, where $(m,n)=1$, and conditions of theorem \ref{T:string_equation} and remark \ref{R:application_of_lemma_inv_conj} are satisfied;
\item
Normalized  normal  pairs $(\partial^p, \tilde{Q}')\in  \hat{D}_1^{sym}$ of order $p=nd ,q=md$ with respect to the fixed normalisation; 
\end{enumerate}
with the same sequences. This correspondence is established with the help of conjugation by a unique monic invertible Schur operator $S\in \hat{D}_1^{sym}$ of order zero satisfying condition $A_p(0)$. 
\end{theorem}

This theorem can be deduced from the classification theory of commuting ordinary differential operators like in \cite[Th. 3.3]{GZ1} or just by the same arguments as in the proof of theorem \ref{T:string_equation}. 

In both cases (normalized string pairs or normalized commutative pairs) we can choose different normalisations. Directly from the proof we see that the corresponding normalized solutions of the string equation or normalized commuting operators are uniquely determined and in fact don't depend on normalisation. 
\end{rem}

Vice versa, any pair of commuting ordinary differential operators of rank one with a specific spectral curve lead to a solution of the string equation:

\begin{theorem}[glueing the tail]
\label{T:inverse_string_equation}
Let $P',Q'\in D_1$, $[P',Q']=0$ be monic operators such that $\rk (K[P',Q'])=1$, $\ord (P')=p=nd$, $\ord (Q')=q=md$, where $(m,n)=1$. Let $f(X,Y)$ be their irreducible BC-polynomial. Assume $f$ is of distance $\le p+q-1$. Let $\tilde{Q}'=S^{-1}Q'S$ be a normal form of $Q'$ with respect to $P'$, where $S$ is a monic Schur operator (see remark \ref{R:normal_forms}). 

Then the pair $(\partial^p, \tilde{Q'}+\tilde{Q}_{-p})$, where $\tilde{Q}_{-p}$ is a "tail" defined in lemma \ref{L:Q_p}, is a normal form of some solution $(P,Q)\in D_1$ to the string equation $[Q,P]=1$, where $P,Q$ are monic and $P$ is normalized (i.e. $\tilde{Q'}+\tilde{Q}_{-p}$ is a normal form of $Q$ with respect to $P$). Moreover, this pair satisfies conditions from lemma \ref{L:invariant_conjugation}, it has the same sequence as $(P',Q')$.

In particular, $P,Q$ satisfy conditions of theorem \ref{T:string_equation}.
\end{theorem}

\begin{proof}
Since the distance of $f$ is less than $p+q$, the operators $P',Q'$ satisfy conditions of remark \ref{R:application_of_lemma_inv_conj} with $SeqGCD_I(P',Q')=1$ and $Dist(P',Q')<p+q$ (see the footnote to definition \ref{D:BC-distance}). Hence the operators $\partial^p,  \tilde{Q}'$ satisfy conditions of lemma \ref{L:invariant_conjugation} with the same sequential GCD and distance. 

Now note that the operators $\partial^p,  \tilde{Q}'+\tilde{Q}_{-p}$ also satisfy these conditions and therefore by lemma \ref{L:invariant_conjugation} there exists a monic invertible operator $\tilde{S}$ of order zero such that $\tilde{S}^{-1}\hat{\Phi}(\tilde{Q}'+\tilde{Q}_{-p})\tilde{S}\in K[x]((\tilde{D}^{-1}))\subset E_p$. 

The key observation is that the ring $K[\partial^p,  \tilde{Q}'+\tilde{Q}_{-p}]$ preserves $F=K[\partial ]$ (see section \ref{S:prelim} for notation), i.e. $F\cdot K[\partial^p,  \tilde{Q}'+\tilde{Q}_{-p}]\subseteq F$, because of special form of the "tail". Therefore the subring $A':=\tilde{S}^{-1}\hat{\Phi}(K[\partial^p,  \tilde{Q}'+\tilde{Q}_{-p}])\tilde{S}\subset E\subset E_p$ preserves the space $W':=F'\cdot \tilde{S} \subset E$. Note that $W'$ has support equal to $F'$, because $\tilde{S}$ is a monic invertible operator of order zero. Then by Sato's theorem \ref{T:Sato} there exists a classical Sato operator $S_{cl}\in E\subset E_p$ such that $W'=F'\cdot S_{cl}$. Hence the subring $G:=S_{cl}A'S_{cl}^{-1}$ stabilises  $F'$, and therefore by another Sato theorem \ref{P:Sato} $G\subset D_1\subset E_p$. The subring $D_1\subset E_p$ is characterised by the properties of theorem \ref{T: P in hat(D) is a dif_op}, and therefore it coincides with the image of the subring $D_1\subset \hat{D}_1^{sym}$ under the embedding $\hat{\Phi}$.  
In particular, 
$S_{cl}\tilde{S}^{-1}\tilde{D}^p\tilde{S}S_{cl}^{-1}$ and $S_{cl}\tilde{S}^{-1}\hat{\Phi}(\tilde{Q}'+\tilde{Q}_{-p})\tilde{S}S_{cl}^{-1}$ are the images of differential operators $P,Q$ providing a solution to the string equation. Note that by construction $P$ is a normalized operator, the orders of $P,Q$ are $p,q$ correspondingly. 

The last assertion of theorem is clear. Besides, by theorem \ref{T:one-to-one-correspondence} we showed that the normalized pair of differential operators $P,Q$ (i.e. with $P$ normalised) with $[Q,P]=0$ or $[Q,P]=1$ is uniquely determined by its normalized (string) pair (in particular, the conjugating operator $\tilde{S}S_{cl}^{-1}$ for which $P,Q$ are differential operators, is uniquely determined). So, the operator  $\tilde{Q'}+\tilde{Q}_{-p}$ is a normal form of $Q$ with respect to $P$.
\end{proof}

\begin{cor}
\label{C:one-to-one2}
There is a one-to-one correspondence between 
\begin{enumerate}
\item
Normalized solutions $P, Q\in D_1$ of the string equation such that  $\ord (P)=p=nd$, $\ord (Q)=q=md$, where $(m,n)=1$, and conditions of theorem \ref{T:string_equation} are satisfied;
\item
Normalized  normal  string pairs $(\partial^p, \tilde{Q})\in  \hat{D}_1^{sym}$ of order $p=nd ,q=md$ with respect to the fixed normalisation as in theorem \ref{T:one-to-one-correspondence};
\item
Normalized pair of commuting operators $P', Q'\in D_1$ such that  $\ord (P')=p=nd$, $\ord (Q')=q=md$, where $(m,n)=1$, and conditions of theorem \ref{T:inverse_string_equation} are satisfied;
\item
Normalized  normal  pairs $(\partial^p, \tilde{Q}')\in  \hat{D}_1^{sym}$ of order $p=nd ,q=md$ with respect to the fixed normalisation as in theorem \ref{T:one-to-one-correspondence};
\item 
Spectral datum $(C,p,\cf ,z)$ as in remark \ref{R:spectral_data_for_string}, where $C$ is the one point compactification of the affine spectral curve $C_0$ determined by an equation $f(X,Y)=0$ from definition \ref{D:BC-distance}, $\cf$ is a torsion free sheaf of rank one on $C$, $z$ is a formal local parameter at the point $p$;
\end{enumerate}
with the same sequences. Here the sequence for the spectral datum is taken for some solution $Y=z^{-nd}$, $X=c_{-md}z^{-md}+\ldots $ as in definition \ref{D:BC-distance}.
\end{cor}

The proof is clear in view of theorems \ref{T:string_equation}, \ref{T:one-to-one-correspondence}, \ref{T:inverse_string_equation} and remark \ref{R:one-to-one1}.

The last result related to the solutions of the string equation  is a correspondence for adjoint differential operators - adjoint solutions of the string equation. 

Recall that there is an anti-isomorphism of the ring $D_1$, denoted by $*: D_1 \rightarrow D_1$, $*(x):=x$, $*(\partial ):= -\partial$. If $P,Q$ is a normalized solution to the string equation, then, clearly, $*(P),-*(Q)$ is another normalized solution to the string equation. Now a natural question appears: how are the corresponding normalized pairs of commuting operators from corollary \ref{C:one-to-one2} related?

\begin{Prop}
\label{P:adjoint}
Let $P, Q\in D_1$ be a normalized solution of the string equation such that  $\ord (P)=p=nd$, $\ord (Q)=q=md$, where $(m,n)=1$, and conditions of theorem \ref{T:string_equation} are satisfied. Let $P', Q'\in D_1$ be the corresponding normalized pair of commuting operators from corollary \ref{C:one-to-one2}. 
Let $*(P),-*(Q)\in D_1$ be the adjoint normalized solution of the string equation. 

Then the corresponding normalized pair of commuting operators from corollary \ref{C:one-to-one2} is $*(P'), -*(Q')$, in other words the diagram \eqref{E:*diagram}  holds.
\end{Prop}  

\begin{proof}
The anti-isomorphism $*$ can not be extended to the whole ring $\hat{D}_1^{sym}$, but it can be extended to the ring $E_p$ as follows: $*(A_{p,i}):=A_{p,p-i}$, $*(\tilde{D})=-\tilde{D}$, $*(\Gamma )=-\partial x=-\Gamma -1$. A direct check show that such extended $*$ is indeed an anti-isomorphism of $E_p$. 

Using this extension we can prove our claim in the ring $E_p$, since $\hat{\Phi}|_{D_1}$ is the embedding of rings compatible with $*$. 

Let $(\partial^p, \tilde{Q})$ be the corresponding to $P,Q$  normalized normal string pair. 
By theorem \ref{T:one-to-one-correspondence} we know that $\tilde{D}^p=S^{-1}\hat{\Phi}(P)S$, 
$\hat{\Phi}(\tilde{Q})= S^{-1}\hat{\Phi}(Q)S$ for a monic zeroth order operator $S\in E_p$ satisfying condition $A_p(0)$. Then $*(\hat{\Phi}(\tilde{Q}))= *(S)*(\hat{\Phi}(Q))(*(S))^{-1}$ and $*(\tilde{D}^p)=*(S)*(\hat{\Phi}(P))(*(S))^{-1}$. Clearly, the pair $(*(\tilde{D}^p),- *(\hat{\Phi}(\tilde{Q})))$ is a normal string pair of order $(p,q)$ with the same sequence as the adjoint normalized solution $*(P),-*(Q)$ (by remark \ref{R:one-to-one1} the operator $*(S)$ is monic zeroth order and satisfy condition $A_p(0)$), and the pair $(*(\tilde{D}^p),- *(\hat{\Phi}(\tilde{Q}-\tilde{Q}_{-p})))$ is a normal pair of order $(p,q)$ with the same sequence as well. 

According to the proof of theorem \ref{T:one-to-one-correspondence} and remark \ref{R:one-to-one1} there is an operator $\tilde{S}\in K^{\oplus p}((\tilde{D}^{-1}))$ such that $\Omega:=\tilde{S}^{-1}\hat{\Phi}(\tilde{Q}-\tilde{Q}_{-p})\tilde{S}\in K((\tilde{D}^{-1}))$. Then also $*(\Omega )=*(\tilde{S})*(\hat{\Phi}(\tilde{Q}-\tilde{Q}_{-p}))*(\tilde{S}^{-1})\in K((\tilde{D}^{-1}))$. 

From the proof of theorem \ref{T:one-to-one-correspondence} we know that the operator $T=S_1\tilde{S}\in E\subset E_p$ (where $S_1$ is the monic zeroth order operator from Step 2 of the proof of theorem \ref{T:one-to-one-correspondence} such that $\tilde{D}^p=S_1^{-1}\hat{\Phi}(P')S_1$, $\hat{\Phi}(\tilde{Q}-\tilde{Q}_{-p})= S_1^{-1}\hat{\Phi}(Q')S_1$) is a classical Schur operator  for the pair $P',Q'$, i.e. $\hat{\Phi}(P')=T\tilde{D}^pT^{-1}$, $\hat{\Phi}(Q')=T\Omega T^{-1}$. 

Then $*(T)^{-1}=*(S_1)^{-1}*(\tilde{S})^{-1}$ is the classical Schur operator for the pair $*(P'),-*(Q')$\footnote{i.e. $*(T)^{-1}*(\tilde{D}^p)*(T)=\hat{\Phi}(*(P'))$, $\hat{\Phi}(-*(Q'))=-*(T)^{-1}*(\Omega )*(T)$.}, and the pair $(*(\tilde{D}^p),- *(\hat{\Phi}(\tilde{Q}-\tilde{Q}_{-p})))$ is a normal pair of order $(p,q)$ with the same sequence as $*(P'),-*(Q')$. Obviously, the pair $*(P'),-*(Q')$ is a normalized pair of commuting operators. So, by corollary \ref{C:one-to-one2} the pair $*(P'),-*(Q')$ is the corresponding normalized pair of commuting operators for $*(P),-*(Q)$.

\end{proof}

\begin{rem}
\label{R:adjoint}
In \cite[\S 6]{HM} it was shown that adjoint rings of commuting operators correspond to the {\it Serre dual} spectral datum, which in our case means that the pair $*(P'),-*(Q')$ corresponds to the datum $(C,p, \cf^*\otimes K_C, z)$, where $K_C$ is the canonical sheaf on $C$. 
\end{rem}

\section{Proof of theorem \ref{T:main}}
\label{S:proof}

In this section we give the proof of theorem \ref{T:main}. The scheme of the proof is as follows. Assuming that there is a DC pair $P,Q$ of subrectangular type we construct first another DC-pair 
$P_n, Q_n$ of subrectangular type with additional beautiful properties: the right vertical sides of their rectangulars are monomials, and the upper horizontal sides have specific coefficients preceding the highest coefficients. 

After that we construct another DC-pair $\Phi_{N,1}(P_n), \Phi_{N,1}(Q_n)$ by applying additional automorphism $\Phi_{N,1}$ with appropriate $N$ for this pair. Then we show (lemma \ref{L:Phi_N(P)}) that there are two cases, depending on the original DC-pair, such that in the first case the new DC pair has a sequence (of operators $Q_{N,i}$ etc.) of length one more than the length of the sequence  for the original DC pair, and the second case of the same length. In both cases we calculate sequential GCDs, distances and the orders of the last operators in the sequences. In both cases this DC-pair satisfies the termination condition for the head-chopping process. Lemma \ref{L:Phi_N(P)} relies heavily on the preparatory theorems \ref{T:DC-pairs}, \ref{T:subrectangular_sequence}. 

After that in section \ref{S:Schur_adj_DC_pairs} we prove theorem \ref{T:main} with the help of one explicit calculation. Namely, according to corollary \ref{C:one-to-one2} the pair of commuting operators $P_{N,n}', Q_{N,n}'$ corresponding to the pair $\Phi_{N,1}(P_n), \Phi_{N,1}(Q_n)$ satisfy an algebraic relation given by some polynomial $f(X,Y)$ of BC-type (see definition \ref{D:BC-distance} and remark \ref{R:classif_ODO}). Going to the normal form $\partial^{\ord P_{N,n}'}, \tilde{Q}_{N,n}'$ of the pair $P_{N,n}', Q_{N,n}'$ and conjugating the operator $\hat{\Phi}(\tilde{Q}_{N,n}')$ by the auxiliary operator $\tilde{S}$, we get a Puiseux series $X_{N,n}$ which gives a solution of the equation $f(\tilde{D}^{\ord P_{N,n}'}, X)=0$. This series has a {\it distinguished coefficient} at $\tilde{D}^{-\ord P_{N,n}'}$ which is equal to zero according to remark \ref{R:distinguished_coefficient} (this is an important property which holds for  series obtained from commuting operators generating the ring of rank one due to the Newton-Puiseux theorem combined with the Wilson theorem). Because of this property it is possible to find distinguished coefficients of the sequence $X_{N,i}=F_i(\tilde{D}^{\ord P_{N,n}'}, X_{N,n})$ -- these are coefficients of homogeneous terms corresponding to the beginning of the "tail" of operators $\tilde{Q}_{N,i}$. 

We get a contradiction by showing that, on the one hand side, these distinguished coefficients don't depend on $N$ (lemma \ref{L:distinguished coefficients}), and on the other hand side by calculating the distinguished coefficient of the series $X_{N,I}$ in Step 7 of section \ref{S:Schur_adj_DC_pairs}  (this calculation shows that it depends on $N$).

\subsection{Preliminary lemmas}
\label{S:prelim_lemmas}

Assume there exists a  DC-pair $P,Q\in A_1$ that provides a counterexample to the Dixmier conjecture and $\varphi$ is the corresponding endomorphism with $\varphi (x)=P$, $\varphi (\partial )=Q$. Then, as it was noted in section \ref{S:DC-pairs}, there exists a DC-pair $P,Q$ of subrectangular type satisfying conditions of theorem \ref{T:DC-pairs} and conditions of remark \ref{R:order_of_polynomial}.  Moreover, without loss of generality and by remark \ref{R:DC-pais} we can assume 
$$
f_{0,1}(P)=(1+\alpha x)^{ln}y^{dn}, \quad f_{0,1}(Q)=(1+\alpha x)^{lm}y^{dm},\quad 1<l<d, \quad (m,n)=1, \quad \epsilon (\varphi)\notin \dn .
$$ 
Let $f_{1,0}(P)=\tilde{h}x^{ln}$, $\tilde{h}\in \dc [y]$, then $\tilde{h}=\tilde{g}^n$ for some $\tilde{g}\in \dc [y]$ of degree $d$ and highest coefficient $\alpha^{l}$, 
and $f_{1,0}(Q)=\tilde{g}^mx^{lm}$ (cf. arguments in section \ref{S:DC-pairs}).

\begin{rem}
\label{R:another_subrectangular}
Consider  operators 
$$\hat{P}=\varphi\circ \Phi_{N,\lambda}(P),\quad \hat{Q}=\varphi\circ \Phi_{N,\lambda}(Q)$$ 
from corollary \ref{C:order_of_polynomial}. Then, they give another DC-pair, moreover, they also satisfy conditions of theorem \ref{T:DC-pairs}, as it immediately follows from the proof of lemma \ref{L:order_of_polynomial} and corollary \ref{C:order_of_polynomial}. Indeed, we need to check only the condition $\ord (\hat{P})=\Ord (\hat{P})$ and $\ord (\hat{Q})=\Ord (\hat{Q})$ (other conditions are trivial). From the proof of lemma \ref{L:order_of_polynomial} we know that 
$$HT(\hat{P})=HT(\varphi (\Phi_{N,\lambda}(\varphi (x)))=
cHT(\varphi (\partial^{nd+Nln}))=cg^{m(nd+Nln)}$$
for some $c\in K^*$ (recall that here $g=(1+\alpha x)^l$), so, $HT(\hat{P})(0)\neq 0$ and therefore $\ord (\hat{P})=\Ord (\hat{P})$ by remark \ref{R:ord=deg}. Analogous arguments work for $\hat{Q}$, since the highest terms of  $\hat{Q}$ and  $\hat{P}$ are proportional by \cite[L.2.7]{Dixmier}. 

By corollary \ref{C:order_of_polynomial} we have 
$$
EssGCD(\hat{P},\hat{Q})=EssGCD(P,Q)=:d_2>1, \quad \ord (\hat{P})=dmn(d+Nl), \quad \ord (\hat{Q})=dm^2(d+Nl)
$$ 
(so that for this pair $d$ is replaced by $dm(d+Nl)$, and $l$ is replaced by $lm(d+Nl)$). So, in particular, we can find a DC-pair with $d\gg 0$. 
\end{rem}

By theorem \ref{T:DC-pairs} there is a sequence of polynomials $Q_i=F_i(X,Y)$, $i=0, \ldots ,I$,   
such that $SeqGCD_I(P,Q)=1$ and $SeqGCD_{I-1}(P,Q)\ge d_2$. By lemma \ref{L:induction_step1} we know that 
$\ord (Q_I)= \tilde{M}_{I-1}+1$, where  
$$
\tilde{M}_{I-1}=\ord (Q_{I-1})n_{I-1}-M_{I-1}=\ord (Q_{I-1})(n_{I-1}-1)+ \ord (Q_{I-2})(n_{I-2}-1)+\ldots +\ord (Q_0)(n-1)-dn. 
$$
From Items 3.4. 7) of this lemma we also know that $f_{0,1}(Q_I)=\varepsilon (1+\alpha x)^{\tilde{M}_{I-1}l/d +1}y^{\tilde{M}_{I-1}+1}$ for some $\varepsilon\in K^*$. Then $f_{1,0}(Q_I)=H_1x^{\tilde{M}_{I-1}l/d +1}$, where $H_1\in K[y]$ is of degree $\tilde{M}_{I-1}+1$ with the highest coefficient $\varepsilon \alpha^{\tilde{M}_{I-1}l/d +1}$.

\begin{lemma}
\label{L:simple_automorphisms}
There exists an automorphism $\varphi_t\in \Aut(A_1)$ ($t\in K$): $\varphi_t(\partial)=\partial+c$, $\varphi_t (x)=x$, where $c\in K$ depends on $t$, such that 
$$
\varphi_t (\tilde{g})=\alpha^{l}y^d+ d\alpha^{l}ty^{d-1}+\ldots ,\quad \varphi_t (H_1)=\varepsilon \alpha^{\tilde{M}_{I-1}l/d +1}y^{\tilde{M}_{I-1}+1}+(\tilde{M}_{I-1}+1)\varepsilon \alpha^{\tilde{M}_{I-1}l/d +1}ty^{\tilde{M}_{I-1}}+\ldots
$$ 
where $\ldots$ means terms of lower degree.
\end{lemma}

\begin{proof}
By theorem \ref{T:subrectangular_sequence} we have two possibilities: either $\tilde{g}$ has one root, say  $\tilde{g}=(b+\tilde{\alpha} y)^d$, or $\tilde{g}$ has more than one root. 

In the first case $H_1=c(b+\tilde{\alpha} y)^{\tilde{M}_{I-1}+1}$, $c\in K$, and our claim is trivial. In the second case $d=ld_2$ and $H_1$ satisfies differential equation \eqref{E:dif_eq_H_1}. Obviously, there is an automorphism  $\varphi(\partial)=\partial+c$, $\varphi (x)=x$ such that $\varphi (\tilde{g})=\alpha^{l}y^d + c_{d-2}y^{d-2}+\ldots$. So, it is sufficient to show that $\varphi (H_1)$ also has zero coefficient at the second place (in this case we can find another automorphism which satisfies our claim). 

So, assume $\tilde{g}=\alpha^{l}y^d + c_{d-2}y^{d-2}+\ldots$, $H_1=h_0y^{\tilde{M}_{I-1}+1}+h_1y^{\tilde{M}_{I-1}}+\ldots $. From theorem \ref{T:subrectangular_sequence} we know that $\tilde{M}_{I-1}=((l(M-n)+1)n_{I-1}-1)d_2$, and the degree of polynomial $H$ from equation \eqref{E:dif_eq_H_1} is $(l(M-n)+1)d_2$. Suppose  
$H=\bar{h}_0y^{(l(M-n)+1)d_2}+\bar{h}_1y^{(l(M-n)+1)d_2-1}+\ldots $. Then the equation on the coefficient $\bar{h}_1$ from equation \eqref{E:dif_eq_H} is (here we use that $d_3=l_3d_2$, as $d=ld_2$)
$$
\bar{h}_1((l(M-n)+1)d_2-1)\alpha^{l_3}-\bar{h}_1((l(M-n)+1)d_2\alpha^{l_3}=0,
$$
whence $\bar{h}_1=0$, and the equation on the coefficient $h_1$ from equation \eqref{E:dif_eq_H_1} is
$$
h_1(((l(M-n)+1)n_{I-1}-1)d_2)\alpha^{l_3}-h_1((l(M-n)+1)n_{I-1}d_2\alpha^{l_3}=0,
$$
whence $h_1=0$, and we are done.
\end{proof}

Consider the automorphism  $\psi\in \Aut(A_1)$: $\psi (\partial)=\partial$, $\psi (x)=x-\alpha^{-1}$. 
Then, obviously, $f_{0,1}(\psi (P))=(\alpha x)^{ln}y^{dn}$, $f_{0,1}(\psi (Q))=(\alpha x)^{lm}y^{dm}$,
and $f_{1,0}(\psi (P))=f_{1,0}(P)$, $f_{1,0}(\psi (Q))=f_{1,0}(Q)$. Then from theorem \ref{T:DC-pairs} and  remark \ref{R:DC-pais} we know that  $f_{0,1}(\psi (Q_i))$ are also monomials of shape $(\alpha x)^{\tilde{M}}y^{\tilde{N}}$ for all $i=0, \ldots ,I$. 

Let $\varphi_t$ be an automorphism  from lemma \ref{L:simple_automorphisms}, and $Fo:x\mapsto \partial$, $\partial \mapsto x$ be the Fourier transform. Put 
\begin{equation}
\label{E:P_n,Q_n}
P_n:=Fo\circ \varphi_t\circ \psi (P), \quad Q_n:=Fo\circ \varphi_t\circ \psi (Q).
\end{equation} 
Then, clearly,  
$$
f_{1,0}(P_n)=(\alpha y)^{ln}x^{dn},\quad f_{1,0}(Q_n)=(\alpha y)^{lm}x^{dm}, \quad f_{0,1}(P_n)=\varphi_t(\tilde{g})(x)^ny^{ln}, \quad f_{0,1}(Q_n)=\varphi_t(\tilde{g})(x)^my^{lm}
$$ 
for any $t$. 
Note that $F_i(P_n,Q_n)=Fo\circ \varphi_t\circ \psi (F_i(P,Q))$, $i=0, \ldots ,I$ are of subrectangular type (for any $t$), because  the automorphism $Fo\circ \varphi_t\circ \psi$ preserves the subrectangular form of the Newton polygon. Moreover, $f_{1,0}(F_i(P_n,Q_n))$ are monomials for all $i=0, \ldots ,I$. Making the scale transform $\partial \mapsto c\partial$, $x\mapsto c^{-1}x$ such that $\alpha^{l(m+n)}c^{(l-d)(m+n)}=1$, and taking the operators $P_n/(\alpha^{ln}c^{(l-d)n})$, $(\alpha^{ln}c^{(l-d)n})Q_n$ instead of $P_n, Q_n$, we can assume without loss of generality that $\alpha =1$, i.e. $Hm(P_n)=x^{dn} \partial^{ln}$, $Hm(Q_n)= x^{dm} \partial^{lm}$ (for any $t$), and the highest coefficient $\tilde{\alpha}=\alpha^l$ of the polynomial $\varphi_t(\tilde{g})$ is also equal to 1. 

\begin{lemma}
\label{L:Phi_N(P)} 
In the notation above (for operators $P_n, Q_n$) assume $d=d_2 GCD(l,d)$, $l=l_2 GCD(l,d)$.

\begin{enumerate}
\item
Assume $d_2-l_2>1$ or $SeqGCD_{I-1}(P,Q)>d_2$. Let $N=(d_2-l_2)k-1$ for $k\gg 0$, $k$ divisible by $SeqGCD_{I-1}(P,Q)$. 

Then for any $t\neq 0$ the operators $\Phi_{N,1}(P_n)$, $\Phi_{N,1}(Q_n)$ satisfy conditions of definition \ref{D:distance}, more precisely, there exists a sequence of operators $Q_{N,i}$, $i=0, \ldots , I+1$, 
where $Q_{N,i}=F_i(\Phi_{N,1}(P_n),\Phi_{N,1}(Q_n))$ for $i=0, \ldots ,I$ (i.e. the polynomials $F_i$ for $i<I+1$ are the same as polynomials for $P, Q$) such that  $SeqGCD_{I+1}(\Phi_{N,1}(P_n),\Phi_{N,1}(Q_n))=1$, $SeqGCD_{I}(\Phi_{N,1}(P_n),\Phi_{N,1}(Q_n))>1$ and 
$$
Dist(\Phi_{N,1}(P_n),\Phi_{N,1}(Q_n))=\ord (\Phi_{N,1}(P_n))+\ord (\Phi_{N,1}(Q_n))-1=(Nd+l)(n+m)-1.
$$

\item 
If $d_2-l_2=1$ and $SeqGCD_{I-1}(P,Q)=d_2$,\footnote{$SeqGCD_{I-1}(P,Q)\ge d_2$ by 
theorem \ref{T:DC-pairs}.} then the same is true for any $t\in K$ and any $N$ for the sequence $Q_{N,i}$, $i=0, \ldots , I$, and in this case $SeqGCD_{I}(\Phi_{N,1}(P_n),\Phi_{N,1}(Q_n))=1$, $SeqGCD_{I-1}(\Phi_{N,1}(P_n),\Phi_{N,1}(Q_n))>1$,
$$
Dist(\Phi_{N,1}(P_n),\Phi_{N,1}(Q_n))= \ord (\Phi_{N,1}(P_n))+\ord (\Phi_{N,1}(Q_n))-1-N.
$$ 
\end{enumerate}

In both cases the sequences for $P,Q$ and for  $\Phi_{N,1}(P_n)$, $\Phi_{N,1}(Q_n)$ are similar in the following sense: if $(\ord (Q_i),\varepsilon_i)$ is the sequence for $P,Q$, then  $(\ord (Q_{N,i}),\varepsilon_i)$ for $i\le I$ is a part of the sequence for $\Phi_{N,1}(P_n)$, $\Phi_{N,1}(Q_n)$ (it is the whole sequence in the second case, and the sequence without the last element in the first case), where we moreover have the following equalities:
$$
\ord (Q_{N,i})=(d_2N+l_2)\ord (Q_i)/d_2 \quad \mbox{for $i<I$}.
$$
\end{lemma} 

\begin{proof}
{\it Step} 1. Obviously, the operators $\Phi_{N,1}(P_n)$, $\Phi_{N,1}(Q_n)$ are monic, 
$$\sigma (\Phi_{N,1}(P_n))=\partial^{dnN+ln}, \quad \ord (\Phi_{N,1}(P_n))=\Ord (\Phi_{N,1}(P_n)), \quad \ord (\Phi_{N,1}(Q_n))=\Ord (\Phi_{N,1}(Q_n)).
$$ 
Let's denote by $\hat{F}_{i}(X,Y)$ a sequence of polynomials for these operators. 
Let's show that for $i=0,\ldots , I$ we can take $\hat{F}_i(X,Y)=F_i(X,Y)$. 

The proof is by induction. For generic $i< I$ we need to show that $f_{0,1}(\Phi_{N, 1 }(F_i(P_n,Q_n)))^{n_i}=\varepsilon_i^{n_i}f_{0,1}(\Phi_{N, 1 }(P_n))^{m_i}$, where $n_i, m_i, \varepsilon_i$ are the sequences from theorem \ref{T:DC-pairs} for polynomials $F_i(X,Y)$ -- in this case we can take the polynomial $\hat{F}_{i+1}$ coinciding with $F_{i+1}$. 

From theorem \ref{T:subrectangular_sequence} we know that $F_i(P,Q)=Q_i$ are of subrectangular type,  
$$
f_{0,1}(Q_i)=\varepsilon_i(1+\alpha x)^{\tilde{M}}y^{\tilde{N}}
$$ 
and also $en_{1,0}(Q_i)$ is proportional with $en_{1,0}(P)=(ln,dn)$ for $i<I$. We also know by lemma \ref{L:induction_step1} that $f_{0,1}(Q_i)^{n_i}=\varepsilon_i^{n_i}f_{0,1}(P)^{m_i}$, so $(\tilde{M}, \tilde{N}) n_i= (ln, dn) m_i$. Then for $P_n$ and $F_i(P_n,Q_n)=Fo\circ \varphi\circ \psi (Q_i)$ we have  $en_{1,0}(P_n)=(dn,ln)$, $en_{1,0}(F_i(P_n,Q_n))=(\tilde{N},\tilde{M})$  and again $en_{1,0}(P_n)m_i=en_{1,0}(F_i(P_n,Q_n))n_i$. 

Then for $i<I$ 
$$
f_{0 ,1}(\Phi_{N, 1}(F_i(P_n,Q_n)))=\varepsilon_i y^{\tilde{M}+N\tilde{N}}, \quad f_{0 ,1}(\Phi_{N, 1}(P_n))= y^{ln+dnN}
$$ 
(as $\alpha =1$). So, $\ord (\Phi_{N, 1}(F_i(P_n,Q_n))) =\Ord (\Phi_{N, 1}(F_i(P_n,Q_n)))$ and $f_{0,1}(\Phi_{N, 1 }(F_i(P_n,Q_n)))^{n_i}=\varepsilon_i^{n_i}f_{0,1}(\Phi_{N, 1 }(P_n))^{m_i}$.

Besides, since $F_i(P_n,Q_n)$  are of subrectangular type and $f_{1,0}(F_i(P_n,Q_n))$ are monomials for $i=0, \ldots ,I$ by theorem \ref{T:DC-pairs} and remark \ref{R:DC-pais},  we have for $i=0, \ldots ,I$
\begin{equation}
\label{E:properties_of_Q_N,i}
({Q}_{N,i})_q=0 \quad \mbox{for $\Ord (Q_{N,i})-N<q<\Ord (Q_{N,i})$},
\end{equation}
and $({Q}_{N,i})_{\Ord (Q_{N,i})-N}=c\partial^{\Ord (Q_{N,i})-N}$, $c\in K$, where $c$ is a constant depending on $t$. 

{\it Step} 2. From theorem \ref{T:subrectangular_sequence} we know that for $i<I$ 
$$en_{1,0}(Q_i)=(l,d)\ord (Q_i)/d=(l_2,d_2)\ord (Q_i)/d_2 .
$$ 
So, for $i<I$ we have $en_{1,0}(F_i(P_n,Q_n))=(d_2,l_2)\ord (Q_i)/d_2$ and therefore 
\begin{equation}
\label{E:formula_for_orders}
\ord (Q_{N,i})=(d_2N+l_2)\ord (Q_i)/d_2
\end{equation} 
(by theorem \ref{T:DC-pairs} we know that $\ord (Q_i)$ is divisible by $d_2$ for $i<I$). Therefore, for $i<I$ we have 
$$
SeqGCD_i(\Phi_{N,1}(P_n),\Phi_{N,1}(Q_n))=SeqGCD_i(P,Q)(d_2N+l_2)/d_2.
$$
Now let's calculate $SeqGCD_I(\Phi_{N,1}(P_n),\Phi_{N,1}(Q_n))$. By lemma \ref{L:induction_step1} we know (see also the formulae above) that $en_{1,0}(Q_I)=(\tilde{M}_{I-1}l/d+1, \tilde{M}_{I-1}+1)$. So, 
we have $en_{1,0}(F_I(P_n,Q_n))=(\tilde{M}_{I-1}+1, \tilde{M}_{I-1}l/d+1)$ and therefore 
$$
\ord (Q_{N,I})=N(\tilde{M}_{I-1}+1)+\tilde{M}_{I-1}l/d+1.
$$
Recall the definition of $\tilde{M}_i$ from lemma \ref{L:induction_step1}: $\tilde{M}_i=\ord (Q_i)n_i-M_i$. So, 
\begin{multline}
\label{E:ordQ_N,i}
\ord (Q_{N,I})=\ord (Q_{I-1})n_{I-1}(N+\frac{l_2}{d_2})-M_{I-1}(N+\frac{l_2}{d_2})+N+1=\\
\ord (Q_{I-1})n_{I-1}(N+\frac{l_2}{d_2})-(d_2-l_2)(k(M_{I-1}-1)-\frac{M_{I-1}}{d_2})
\end{multline}
(here we put $N=(d_2-l_2)k-1$ for both cases of lemma). Therefore,
\begin{multline*}
SeqGCD_I(\Phi_{N,1}(P_n),\Phi_{N,1}(Q_n))=GCD(SeqGCD_{I-1}(\Phi_{N,1}(P_n),\Phi_{N,1}(Q_n)), \ord (Q_{N,I}))= \\
GCD(SeqGCD_{I-1}(P,Q)(d_2N+l_2)/d_2, (d_2-l_2)(k(M_{I-1}-1)-\frac{M_{I-1}}{d_2}))=\\
(d_2-l_2)GCD(SeqGCD_{I-1}(P,Q)(kd_2-1)/d_2, (k(M_{I-1}-1)-\frac{M_{I-1}}{d_2}))=\\
(d_2-l_2)SeqGCD_{I-1}(P,Q)/d_2,
\end{multline*}
where the last equality hold because $M_{I-1}/d_2$ is divisible by $SeqGCD_{I-1}(P,Q)/d_2$ by theorem \ref{T:DC-pairs} and $k$ is divisible by $SeqGCD_{I-1}(P,Q)/d_2$ by condition. Hence it is $>1$ if conditions of item 1) hold. If conditions of item 2) hold, then this GCD is $1$ for any $k$. Moreover, in this case 
\begin{multline}
\label{E:promezh_distance}
\sum_{i=0}^{I-1}(n_i\ord (Q_{N,i})-\ord (Q_{N,i+1}))=
\sum_{i=0}^{I-2}(n_i\ord (Q_{i})-\ord (Q_{i+1}))(N+\frac{l_2}{d_2})+\\
(n_{I-1}\ord (Q_{N,I-1})-\ord (Q_{N,I}))=
(N+\frac{l_2}{d_2})(Dist(P,Q)+1)-N-1=\\
(N+\frac{l_2}{d_2})(p+q)-N-1 =
\ord (\Phi_{N,1}(P_n))+\ord (\Phi_{N,1}(Q_n))-1-N, 
\end{multline} 
which proves Item 2) of lemma. Note that all differences $(n_i\ord (Q_{N,i})-\ord (Q_{N,i+1})$ are big if $k\gg 0$. 

{\it Step} 3. At last, to prove Item 1), we need to calculate the polynomial $\hat{F}_{I+1}$, the operator $Q_{N,I+1}=\Phi_{N,1}(\hat{F}_{I+1}(P_n,Q_n))$ and its order, distance and sequential GCD. Let's take for simplicity of calculations $\hat{n}_I=\ord (\Phi_{N,1}(P_n))=(ln+dnN)$, and $\hat{m}_I=\ord (Q_{N,I})$. 

Recall that from Items 3.4. 7) of  lemma \ref{L:induction_step1} we  know that $f_{0,1}(Q_I)=\varepsilon (1+\alpha x)^{\tilde{M}_{I-1}l/d +1}y^{\tilde{M}_{I-1}+1}$ for some $\varepsilon\in K^*$, $Q_I=F_I(P,Q)$. 
Since $Q_I$ is of subrectangular type, we have therefore 
$$f_{0 ,1}(\Phi_{N, 1 }(F_I(P_n,Q_n)))=\varepsilon y^{(\tilde{M}_{I-1}l/d +1)+(\tilde{M}_{I-1}+1)N}.$$
So, we can find the constant $\hat{\varepsilon}_{I}$ of $\hat{F}_{I+1}$\footnote{from the equation $f_{0,1}(Q_{N,I})^{\hat{n}_I}- \hat{\varepsilon}_{I} f_{0,1}(\Phi_{N, \lambda }(P_n))^{\ord (Q_{N,I})}=0$}: $\hat{\varepsilon}_{I}=\varepsilon^{(ln+dnN)}$, and so
$$
\hat{F}_{I+1}(X,Y)=F_I(X,Y)^{(ln+dnN)}-\varepsilon^{(ln+dnN)}  X^{(\tilde{M}_{I-1}l/d +1)+(\tilde{M}_{I-1}+1)N}. 
$$
By  remark \ref{R:subrectangular} we know that $F_I(P_n,Q_n)^{ln+dnN}$ and $P_n^{(\tilde{M}_{I-1}l/d +1)+(\tilde{M}_{I-1}+1)N}$ are operators of subrectangular type. Moreover we have
\begin{multline*}
en_{1,0}(F_I(P_n,Q_n)^{ln+dnN})=(\tilde{M}_{I-1}+1, \tilde{M}_{I-1}l/d +1)\cdot (ln+dnN)\\
en_{1,0}(P_n^{(\tilde{M}_{I-1}l/d +1)+(\tilde{M}_{I-1}+1)N})=(dn,ln)\cdot ((\tilde{M}_{I-1}l/d +1)+(\tilde{M}_{I-1}+1)N).
\end{multline*}
From this we immediately get that 
\begin{multline*}
\ord (P_n^{(\tilde{M}_{I-1}l/d +1)+(\tilde{M}_{I-1}+1)N})-\ord F_I(P_n,Q_n)^{(ln+dnN)} =(l-d)nN<0 \\
\ord_x (P_n^{(\tilde{M}_{I-1}l/d +1)+(\tilde{M}_{I-1}+1)N})-\ord_x F_I(P_n,Q_n)^{(ln+dnN)} = (d-l)n >0,
\end{multline*}
that is 
$$
\ord (\hat{F}_{I+1}(P_n,Q_n))=\ord F_I(P_n,Q_n)^{\hat{n}_I}=(\tilde{M}_{I-1}l/d +1)(ln+dnN)
$$
and the highest monomials of the operators $F_I(P_n,Q_n)^{(ln+dnN)}$ and $P_n^{(\tilde{M}_{I-1}l/d +1)+(\tilde{M}_{I-1}+1)N}$ belong to the $(N,1)$-top line of the operator $\hat{F}_{I+1}(P_n,Q_n)$, and this top line contains exactly these two vertices. Now direct calculation of the polynomial $f_{N,1}(\Phi_{N,1 }(\hat{F}_{I+1}(P_n,Q_n)))$ gives
\begin{multline*}
f_{N,1}(\Phi_{N,1 }(\hat{F}_{I+1}(P_n,Q_n)))=
(\varepsilon ((x+  y^N)^{\tilde{M}_{I-1}+1}y^{\tilde{M}_{I-1}l/d +1})^{(ln+dnN)} 
-\\
\varepsilon^{ln+dnN}  (( ( x+ y^N)^{dn}y^{ln})^{(\tilde{M}_{I-1}l/d +1)+(\tilde{M}_{I-1}+1)N}=\\
\varepsilon^{(ln+dnN)}( x+ y^N)^{(\tilde{M}_{I-1}+1)(ln+dnN)}
y^{ln((\tilde{M}_{I-1}l/d +1)+(\tilde{M}_{I-1}+1)N)}\cdot 
(y^{N(d-l)n}-( x+ y^N)^{(d-l)n}).
\end{multline*}
Therefore the $(N,1)$-top line of the operator $\Phi_{N,1 }(\hat{F}_{I+1}(P_n,Q_n))$ contains no vertex on the $y$-axis, and contains the vertex $(1, (ln+dnN)((\tilde{M}_{I-1}l/d +1)+(\tilde{M}_{I-1}+1)N)-N)$. Besides, the coefficient at this vertex is 
\begin{equation}
\label{E:**}
c:=-\varepsilon^{(ln+dnN)}(d-l)n.
\end{equation} 

{\it Step} 4. Now we claim that 
\begin{multline*}
\ord (\Phi_{N,1 }(\hat{F}_{I+1}(P_n,Q_n)))= \Ord (\Phi_{N,1 }(\hat{F}_{I+1}(P_n,Q_n)))= \\
(ln+dnN)((\tilde{M}_{I-1}l/d +1)+(\tilde{M}_{I-1}+1)N)-N= (ln+dnN)\ord (Q_{N,I}) -N.
\end{multline*}
Since $f_{1,0}(F_I(P_n,Q_n)^{(ln+dnN)})$ and $f_{1,0}(P_n^{(\tilde{M}_{I-1}l/d +1)+(\tilde{M}_{I-1}+1)N})$ are monomials, belonging to the $(N,1)$-top line above, the next monomials of these operators belonging to a lower $(N,1)$-line parallel to the top line, will be the second monomials\footnote{here we mean the next after the highest monomial of the polynomial $H_1(x)^{(ln+dnN)}$ multiplied with $y^{(\tilde{M}_{I-1}l/d +1)(ln+dnN)}$ and, respectively, the next after the highest monomial of the polynomial $\tilde{g}(x)^{n((\tilde{M}_{I-1}l/d +1)+(\tilde{M}_{I-1}+1)N)}$ multiplied with $y^{ln((\tilde{M}_{I-1}l/d +1)+(\tilde{M}_{I-1}+1)N)}$ } of 
\begin{multline*}
f_{0,1}(F_I(P_n,Q_n)^{(ln+dnN)})=H_1(x)^{(ln+dnN)}y^{(\tilde{M}_{I-1}l/d +1)(ln+dnN)}\quad \mbox{and}\quad  \\ 
f_{0,1}(P_n^{(\tilde{M}_{I-1}l/d +1)+(\tilde{M}_{I-1}+1)N})=\tilde{g}(x)^{n((\tilde{M}_{I-1}l/d +1)+(\tilde{M}_{I-1}+1)N)}y^{ln((\tilde{M}_{I-1}l/d +1)+(\tilde{M}_{I-1}+1)N)},
\end{multline*}
where $H_1, \tilde{g}$ are polynomials from lemma \ref{L:simple_automorphisms}. Direct calculation of the sum of images of these monomials under automorphism $\Phi_{N,1}$ gives 
$$
H_1(x)^{(ln+dnN)}=\varepsilon^{(ln+dnN)}(x^{(ln+dnN)(\tilde{M}_{I-1}+1)}+t(ln+dnN)(\tilde{M}_{I-1}+1)
x^{(\tilde{M}_{I-1}+1)(ln+dnN-1) +\tilde{M}_{I-1}}+\ldots ),
$$
\begin{multline*}
\tilde{g}(x)^{n((\tilde{M}_{I-1}l/d +1)+(\tilde{M}_{I-1}+1)N)}=x^{dn((\tilde{M}_{I-1}l/d +1)+(\tilde{M}_{I-1}+1)N)} +\\
td n((\tilde{M}_{I-1}l/d +1)+(\tilde{M}_{I-1}+1)N)x^{d(n((\tilde{M}_{I-1}l/d +1)+(\tilde{M}_{I-1}+1)N)-1)+d-1}+\ldots ,
\end{multline*}
whence 
\begin{multline*}
\varepsilon^{(ln+dnN)}(t(ln+dnN)(\tilde{M}_{I-1}+1)
(x+y^N)^{(\tilde{M}_{I-1}+1)(ln+dnN-1) +\tilde{M}_{I-1}}y^{(\tilde{M}_{I-1}l/d +1)(ln+dnN)}- \\
td n((\tilde{M}_{I-1}l/d +1)+(\tilde{M}_{I-1}+1)N)(x+y^N)^{d(n((\tilde{M}_{I-1}l/d +1)+(\tilde{M}_{I-1}+1)N)-1)+d-1} y^{ln((\tilde{M}_{I-1}l/d +1)+(\tilde{M}_{I-1}+1)N)})=\\
-\varepsilon^{(ln+dnN)}(d-l)nty^{(ln+dnN)((\tilde{M}_{I-1}l/d +1)+(\tilde{M}_{I-1}+1)N)-N}+ \mbox{lower terms}.
\end{multline*}

So, we get the claim and moreover we have calculated the highest symbol 
$$
\sigma (\Phi_{N,1 }(\hat{F}_{I+1}(P_n,Q_n)))=-\varepsilon^{(ln+dnN)}(d-l)nt\partial^{(ln+dnN)((\tilde{M}_{I-1}l/d +1)+ (\tilde{M}_{I-1}+1)N)-N}.
$$
The same observations show that the homogeneous components of $Q_{N,I+1}=\Phi_{N,1}(\hat{F}_{I+1}(P_n,Q_n))$ have the property 
$$
(Q_{N,I+1})_q=c_q\partial^{q}, \quad c_q\in K \quad \mbox{for $\Ord (Q_{N,I+1})-N <q<\Ord (Q_{N,I+1})-1$, }
$$
hence 
$$
Sdeg_A ((Q_{N,I+1})_{q})=0 \quad \mbox{ for $\Ord (Q_{N,I+1})-N <q<\Ord (Q_{N,I+1})-1$}
$$ 
(and $Sdeg_A((Q_{N,I+1})_{\Ord (Q_{N,I+1})-1})=1$).

{\it Step} 5. Now note that $SeqGCD_{I+1}(\Phi_{N,1}(P_n), \Phi_{N,1}(Q_n))=1$. Indeed, the number \\
$SeqGCD_{I}(\Phi_{N,1}(P_n), \Phi_{N,1}(Q_n))$ divides $(d_2-l_2)SeqGCD_{I-1}(P,Q)$, because $(kd_2-1)$ is coprime with $(k(M_{I-1}-1)-M_{I-1}/d_2)$. And also this number divides $\ord (Q_{N,I})$. Since 
$$
\ord (Q_{N,I+1})=(ln+dnN)(\ord (Q_{N,I}))-N=(ln+dnN)(\ord (Q_{N,I}))-(d_2-l_2)k+1
$$
and $k$ is divisible by $SeqGCD_{I-1}(P,Q)$, we get
$$
SeqGCD_{I+1}(\Phi_{N,1}(P_n), \Phi_{N,1}(Q_n))=GCD(SeqGCD_{I}(\Phi_{N,1}(P_n), \Phi_{N,1}(Q_n)),\ord (Q_{N,I+1}))=1. 
$$
Moreover, 
$$
\ord (Q_{N,I+1})=1 \quad \mbox{mod $SeqGCD_I(\Phi_{N,1}(P_n), \Phi_{N,1}(Q_n))$}.
$$

{\it Step} 6. At last, the distance is 
\begin{multline*}
Dist(\Phi_{N,1}(P_n), \Phi_{N,1}(Q_n))= \sum_{i=0}^{I-1}(n_i\ord (Q_{N,i})-\ord (Q_{N,i+1})) +
(\hat{n}_I\ord (Q_{N,I})-\ord (Q_{N,I+1}))= \\
(N+\frac{l_2}{d_2})(Dist(P,Q)+1)-1=\ord (\Phi_{N,1}(P_n))+\ord (\Phi_{N,1}(Q_n))-1.
\end{multline*}
 
\end{proof}

\subsection{Some explicit calculations}
\label{S:Schur_adj_DC_pairs}

In this subsection we provide  calculations of  coefficients of the Puiseux series corresponding to the sequence of commuting differential operators mentioned at the beginning of this section.

{\it Step} 1. Consider the  operators $P_n, Q_n, \Phi_{N,1}(P_n), \Phi_{N,1}(Q_n)$ from  lemma \ref{L:Phi_N(P)}.  Note that they are monic as differential operators and as operators in $\hat{D}_1^{sym}$, 
$$\sigma (\Phi_{N,1}(P_n))=\partial^{ln+dnN}, \quad \sigma (\Phi_{N,1}(Q_n))=\partial^{lm+dmN},$$  and  
$$(\Phi_{N,1}(P_n))_q=0\quad \mbox{ for \quad } ln+dnN>q>ln+ndN-N$$ 
(cf. the observations at the end of Step 1 of the proof lemma \ref{L:Phi_N(P)}), i.e. $\Phi_{N,1}(P_n)$ is normalized, and analogously 
$$
(\Phi_{N,1}(Q_n))_q=0\quad \mbox{ for \quad } lm+dmN>q>lm+dmN-N.
$$ 
Besides, $(\Phi_{N,1}(P_n))_{ln+ndN-N}=\tilde{c}_1\partial^{ln+ndN-N}$, $(\Phi_{N,1}(Q_n))_{lm+dmN-N}=\tilde{c}_2\partial^{lm+dmN-N}$, where $\tilde{c}_1, \tilde{c}_2\in K$ depend on $t$.

{\it Step} 2. Let $S$ be a Schur operator from proposition \ref{P:Si} for $\Phi_{N,1}(P_n)$: $S^{-1}\Phi_{N,1}(P_n)S=\partial^{ln+dnN}$. Note that $Sdeg_A(S_q)=0$ for $q>-N$ (because  otherwise  $(\Phi_{N,1}(P_n))_q\neq 0$). Therefore, such $S_q\in C(\partial^{ln+dnN})$ and we can find, multiplying with an appropriate operator from the centralizer, an operator $S$ with $S_q=0$ for $q=-1, \ldots ,-N+1$. 

Consider now  the operators $\tilde{Q}_{N,i}:=S^{-1}Q_{N,i}S$. Then 
$$
(\tilde{Q}_{N,i})_q=(Q_{N,i})_q\in K \quad \mbox{ for $q> \Ord (Q_{N,i})-N$ for all $i\le I+1$}.
$$ 
In view of lemma \ref{L:Phi_N(P)}, the pair $(\partial^{ln+dnN}, \tilde{Q}_{N,0})$ is a normal string pair of order $(ln+dnN, lm+dmN)$. Let
$$
(\partial^{ln+dnN},\tilde{Q}_{N,0}'):= (\partial^{ln+dnN}, \tilde{Q}_{N,0}-(\tilde{Q}_{N,0})_{-ln-dnN})
$$ 
be the corresponding normal pair (see theorem \ref{T:string_equation} and corollary \ref{C:one-to-one2}). 

{\it Step} 3.  For simplicity of further calculations let's take a sequence of polynomials $F_i(X,Y)$ for initial operators $P,Q$ from section \ref{S:prelim_lemmas} for which $n_i=p$ for all $i$ (see theorem \ref{T:DC-pairs} for notation). To fix notation, let $(\ord Q_i, \epsilon_i )$, $i=0, \ldots ,I$ be the corresponding sequence for the pair $P,Q$. Then by lemma \ref{L:Phi_N(P)} there exists similar sequence $(\ord Q_{N,i}, \epsilon_i )$ also for the pair $\Phi_{N,1}(P_n), \Phi_{N,1}(Q_n)$ (with the same $\epsilon_i$ for $i=0, \ldots ,I$). Below we'll continue to use notations from lemma \ref{L:Phi_N(P)} and assume $N>1$. 

Let $P_{N,n}', Q_{N,n}'$ be the normalised pair of commuting operators corresponding to the string pair $\Phi_{N,1}(P_n), \Phi_{N,1}(Q_n)$ according to corollary \ref{C:one-to-one2}. Denote by $Q_{N,i}':=F_i(P_{N,n}', Q_{N,n}')$ the corresponding sequence of differential operators (by corollary \ref{C:one-to-one2} this is the same sequence of polynomials $F_i$ as for $\Phi_{N,1}(P_n), \Phi_{N,1}(Q_n)$ from lemma \ref{L:Phi_N(P)}), and by $\tilde{Q}_{N,i}':=F_i(\tilde{P}_{N,n}', \tilde{Q}_{N,n}')$ the sequence of normal forms of $Q_{N,i}'$ with respect to $P_{N,n}'$. 

Denote by $X_{N,n}\in K((\tilde{D}^{-1}))$ the Puiseux series obtained from  $\tilde{Q}_{N,0}'=\tilde{Q}_{N,n}'$  by conjugation with the auxiliary operator $\tilde{S}$ from lemma \ref{L:invariant_conjugation} (so, $X_{N,n}=\tilde{S}^{-1}\hat{\Phi}(\tilde{Q}_{N,0}')\tilde{S}$ is a solution of the BC equation $f(\tilde{D}^{ln+Ndn}, X)=0$ giving the algebraic relation between $P_{N,n}', Q_{N,n}'$, cf. remark \ref{R:classif_ODO}), and by $X$ the corresponding Puiseux series obtained from  $\tilde{Q}'$ (a normal form of $Q'$ with respect to $P'$, where $P',Q'$ is the normalised pair of commuting operators corresponding to the string pair $P,Q$).  Denote by $X_{i}:=F_i(\tilde{D}^{dn},X)\in K((\tilde{D}^{-1}))$
$X_{N,i}:=F_i(\tilde{D}^{ln+Ndn},X_{N,n})\in K((\tilde{D}^{-1}))$ the corresponding sequence of series. To fix notation let
$$
X=\tilde{D}^{dm}+\sum_{j<dm} c_{j}\tilde{D}^j, \quad X_i=\epsilon_i\tilde{D}^{\ord Q_{i}}+\sum_{j<\ord Q_{i}} c_{i;j}\tilde{D}^j;
$$
$$
X_{N,n}=\tilde{D}^{dmN+lm}+\sum_{j<dmN+lm} c_{n,j}\tilde{D}^j, \quad X_{N,i}=\epsilon_i\tilde{D}^{\ord Q_{N,i}}+\sum_{j<\ord Q_{N,i}} c_{N,i;j}\tilde{D}^j. 
$$
All series $X_{N,n}, X_{N,i}$ have {\it distinguished coefficients} -- these are coefficients of homogeneous terms corresponding to the beginning of the "tail" of operators $\tilde{Q}_{N,i}$. For  $X_{N,n}$ such coefficient is $c_{n,-dnN-ln}$, and for arbitrary $X_{N,k}$ such coefficient is $c_{N,k; r(k)}$ where 
$$
r(k)=\ord Q_{N,k} -((N+\frac{l_2}{d_2})(dn+dm)-\sum_{i=0}^{k-1}(p\ord (Q_{N,i})-\ord (Q_{N,i+1})))
$$
(to explain this formula, lets note that at each step of the head-chopping process  the distance between a new head and the beginning of the tail is shortened by the amount $(p\ord (Q_{N,i})-\ord (Q_{N,i+1}))$, and the initial distance is $(N+\frac{l_2}{d_2})(dn+dm)=\ord (\Phi_{N,1}(P_n))+ \ord (\Phi_{N,1}(Q_n))$).

By remark \ref{R:distinguished_coefficient}  we know that $c_{n,-dnN-ln}=0$. From this we can immediately deduce that all distinguished coefficients $c_{N,k; r(k)}$ for $k\le I$ are polynomial expressions in $\epsilon_i$, $i<I$ with integer coefficients {\it not depending on $N$}: 

\begin{lemma}
\label{L:distinguished coefficients}
In the notations above the distinguished coefficients $c_{N,k; r(k)}$ for $k\le I$ are polynomial expressions in $\epsilon_i$, $i<I$ with integer coefficients {\it not depending on $N$} (i.e. they are the same for any $N$).
\end{lemma}

\begin{proof}
By lemma \ref{L:Phi_N(P)} (cf. formula \eqref{E:formula_for_orders})  we have 
$$
\ord Q_{N,i}=(d_2N+l_2)\ord (Q_i)/d_2 \quad \mbox{ for $i< I$}, 
$$
and $\ord Q_{N,i}=\ord Q_{N,i}'$, $\ord (Q_i)=\ord (Q_i')$ (by corollary \ref{C:one-to-one2} the sequences coincide both for $Q_{N,n}$, $Q_{N,n}'$ and for $Q$, $Q'$). 
Therefore, all non-zero coefficients $c_{n,i}$ for $i>-dnN-ln+N+1$ (cf. equation \eqref{E:promezh_distance})  must have indices $i$ divisible by $(d_2N+l_2)$ (cf. remark \ref{R:relation_between_coeff}). Moreover, 
$$
c_{n,(d_2N+l_2)i}=c_{i} \quad \mbox{for $i>-dn+1$}, 
$$
because the polynomials $F_i$, $i\le I$ are the same for both pairs of operators ($P',Q'$ and $P_{N,n}', Q_{N,n}'$) and therefore the expressions for these coefficients in terms of $\epsilon_i$ coincide.\footnote{By the same reason $c_{-dn+1}=c_{n,-dnN-ln+N+1}$, though we'll not need this equality below.} 

Now our claim can be proved by induction on $k$. For $k=0$ the distinguished coefficients $c_{N,0; r(0)}$ are zero by remark \ref{R:distinguished_coefficient}. 

For $k\le I$ the expressions for $c_{N,k; r(k)}$ in terms of coefficients $c_{n,i}$ are {\it polynomials in 
$c_{n,(d_2N+l_2)i}$ for $i>-dn+1$} and  are the same for all $N$ (and coincide with the expression for the distinguished coefficients of $X_i$ in terms of $c_i$ for $i>-dn+1$). To prove it, we need to find expression for $c_{N,k; r(k)}$ in terms of coefficients of the series $X_{N,k-1}$. By induction hypothesis, the expressions for coefficients
$c_{N,k-1;j}$  for 
\begin{equation}
\label{E:inequality_for_j}
j>  \ord Q_{N,k-1} -((N+\frac{l_2}{d_2})(dn+dm)-\sum_{i=0}^{k-2}(p\ord (Q_{N,i})-\ord (Q_{N,i+1})))+N+1
\end{equation}
in terms of coefficients $c_{n,(d_2N+l_2)i}$ for $i>-dn+1$ are the same for all $N$ (the explanation of this formula is analogous to the formula for $r(k)$ above; in view of equation \eqref{E:promezh_distance} the distance between the distinguished coefficient $c_{N,I, r(I)}=0$ and the head of $Q_{N,I}$ is $N+1$, so we need to take into account only  coefficients of bigger distances from the distinguished coefficients in all series $X_{N,i}$). 

The expression for $c_{N,k; r(k)}$ in terms of coefficients of the series $X_{N,k-1}$ is the homogeneous component of the series $X_{N,k-1}^p$ of order $r(k)$ -- the sum of products of homogeneous components of the series $X_{N,k-1}$, let's call such products as monomials for short. If $k<I$,  in view of formula \eqref{E:formula_for_orders}, the distance between the head of $X_{N,k-1}$ and the next non-zero homogeneous component is 
$$
p\ord Q_{N,k-1}-\ord Q_{N,k}= (d_2 N+l_2)(p\ord Q_{k-1}-\ord Q_{k})> N+1,
$$
therefore any non-zero monomial of the series $X_{N,k-1}^p$ of order $r(k)$ will consist only of products of homogeneous components of orders $j$ satisfying inequality \eqref{E:inequality_for_j}. 

If $k=I$, by Step 2 of the proof of lemma \ref{L:Phi_N(P)} (formula \eqref{E:formula_for_orders}) we have 
$$
p\ord Q_{N,k-1}-\ord Q_{N,k}= (d_2-l_2)(k(M_{I-1}-1)-\frac{M_{I-1}}{d_2}),
$$
and this number is less than $N+1=(d_2-l_2)k$ only if $M_{I-1}=d_2=2$. In this case $l_2=1$, an so the case 2) of lemma \ref{L:Phi_N(P)} holds, and 
$$
p\ord Q_{N,I-1}-\ord Q_{N,I}=N.
$$
In this case a non-zero monomial of the series $X_{N,I-1}^p$ of order $r(I)$ can contain a homogeneous components of order $j$ not satisfying inequality \eqref{E:inequality_for_j} only if the homogeneous component of order $r(I-1)+N$ is not equal to zero. But this is impossible: by formulae \eqref{E:promezh_distance}, \eqref{E:properties_of_Q_N,i} we have $(Q_{N,I})_{r(I)+N}=(Q_{N,I})_{\ord Q_{N,I}-1}=0$. Clearly, $r(I)+N, r(I-1)+N\notin (2N+1)\dz$ (as all numbers $r(k)$ are divisible by $d_2N+l_2$ by lemma \ref{L:Phi_N(P)}), so by the observations from remark \ref{R:relation_between_coeff} we must have $c_{N,I-1; r(I-1)+N}=0$. So, in this case again any non-zero monomial of the series $X_{N,k-1}^p$ of order $r(k)$ will consist only of products of homogeneous components of orders $j$ satisfying inequality \eqref{E:inequality_for_j}, and we are done. 
\end{proof}

{\it Step} 4. The next step in our proof is to calculate the distinguished coefficient $c_{N,I;r(I)}$. 
To provide this calculation, we need to calculate first the corresponding homogeneous component of order $r(I)$ of the operator $\tilde{Q}_{N,I}$, then calculate the body-tail decomposition for this component (this component is the beginning of the tail, so the calculation is not difficult), and finally calculate the result of conjugation by the auxiliary operator $\tilde{S}$. In fact, we will not even need to calculate $\tilde{S}$, because the result of conjugation for this component will be just the averaging of components of its vector form. We'll see that the result of direct calculation of $c_{N,I;r(I)}$ will depend on $N$, thus contradicting lemma \ref{L:distinguished coefficients}. 
 
Let's calculate the homogeneous component of order $r(I)$ of the operator $\tilde{Q}_{N,I}$ -- this is the component $(\tilde{Q}_{N,I})_{\ord Q_{N,I}-N-1}$ by lemma \ref{L:Phi_N(P)}, formula \eqref{E:promezh_distance}. To calculate it, we need to find first two non-zero homogeneous components $S_{-N}, S_{-N-1}$ of the Schur operator $S$ from Step 2. As usual, $S$ can be represented as a product of elementary operators $\bar{S}_i:=1+s_i\int^{-i}$, and here $i\le -N$. Note that conjugation by $\bar{S}_{-N}$ does not affect the homogeneous component of order $\ord Q_{N,I}-N-1$, so we need to find only the operator $\bar{S}_{-N-1}$ and the result of conjugation by it. 

Note that we have 
\begin{multline*}
(\Phi_{N,1}(P_n))_{Ndn+ln-N-1}=(dn \Gamma_1+a)\partial^{Ndn+ln-N-1}, \quad \\
(Q_{N,I})_{\Ord (Q_{N,I})-N-1}=\epsilon_I((\tilde{M}_{I-1}+1)\Gamma_1+b)\partial^{N\tilde{M}_{I-1}+\tilde{M}_{I-1}l_2/d_2}, 
\end{multline*}
where $a$ is a coefficient of the operator $P_n$  and $b$ is a coefficient of the operator $Q_{n,I}=F_I(P_n,Q_n)$ which are therefore do not depend on $N$. So, $\bar{S}_{-N-1}$ should be of the shape (cf. proposition \ref{P:Si})
$$
\bar{S}_{-N-1}=1+(\gamma_2\Gamma_2+\gamma_1\Gamma_1+u)\int^{N+1},
$$
where $\Phi (u)\in K^{\oplus Ndn+ln}$ and $\gamma_1, \gamma_2\in K$ ($u$ can be found with the help of the condition $\bar{S}_{-N-1}$ is totally free of $B_j$). From the formula above and commutation relations of lemma \ref{L:1} we get 
\begin{multline*}
\bar{S}_{-N-1}^{-1}\bar{S}_{-N}^{-1}(\Phi_{N,1}(P_n))\bar{S}_{-N}\bar{S}_{-N-1}= \partial^{Ndn+ln}+ \\ (\gamma_2(\Gamma_1+ln+Ndn)^2+\gamma_1(\Gamma_1+ln+Ndn)-\gamma_2\Gamma_2-\gamma_1\Gamma_1+dn \Gamma_1+a)\partial^{Ndn+ln-N-1}+\ldots,
\end{multline*}
where from we get 
$$
\gamma_2=-\frac{d_2}{2(Nd_2+l_2)}, \quad \gamma_1=-\frac{a}{ln+Ndn}+\frac{dn}{2}.
$$
Now we can calculate $(\tilde{Q}_{N,I})_{\ord Q_{N,I}-N-1}$. Since $\sigma (Q_{N,I})=\epsilon_I\partial^{N(\tilde{M}_{I-1}+1)+\tilde{M}_{I-1}l_2/d_2+1}$, we get after conjugation, similarly to the formula above, 
\begin{multline}
\label{E:tildeQ_N,I}
(\tilde{Q}_{N,I})_{\ord Q_{N,I}-N-1}=\epsilon_I(\gamma_2(\Gamma_1+N(\tilde{M}_{I-1}+1)+\tilde{M}_{I-1}l_2/d_2+1)^2+\gamma_1(\Gamma_1+N(\tilde{M}_{I-1}+1)+\tilde{M}_{I-1}l_2/d_2+1)+\\
\sigma^{-\ord Q_{N,I}}(u)-\gamma_2\Gamma_2-\gamma_1\Gamma_1-u+(\tilde{M}_{I-1}+1)\Gamma_1+b)\partial^{\ord Q_{N,I}-N-1} =\\
\epsilon_I(-\frac{d-l}{d N+l}\Gamma_1 +(\sigma^{-\ord Q_{N,I}}(u)-u)+A\frac{N^2}{Nd+l}+B\frac{N}{Nd+l}+C\frac{1}{Nd+l} +b),
\end{multline}
where 
$$
A=\frac{1}{2} d (\tilde{M}_{I-1}+1) (d n-\tilde{M}_{I-1}-1),  
$$
$$
B=\frac{1}{2} \left(-\frac{2 (\tilde{M}_{I-1}+1) (a+l \tilde{M}_{I-1} n)}{n}+d^2 n+d (2 \tilde{M}_{I-1} (l n-1)+l n-2)\right) ,
$$
$$
C=-\frac{(d+l \tilde{M}_{I-1}) (2 a+n (d (-l) n+d+l \tilde{M}_{I-1}))}{2 d n}.
$$
Note that $A\neq 0$, because $\tilde{M}_{I-1}=\ord (Q_{I-1})nd-M_{I-1}$, $M_{I-1}<p+q=d(n+m)$ and $\ord (Q_{I-1})\ge dm$ is a polynomial in $p$ of order $I-1$  (see lemma \ref{L:induction_step1}).
\begin{rem}
\label{R:S_-N}
Since $(\Phi_{N,1}(P_n))_{ln+ndN-N}=\tilde{c}_1\partial^{ln+ndN-N}$, where $\tilde{c}_1$ is a constant, the coefficient $s_{-N}$ of the operator $\bar{S}_{-N}$ is a linear polynomial in $\Gamma_1$, hence $Sdeg_A((\tilde{Q}_{N,I})_{\ord Q_{N,I}-N})=0$ and $(\tilde{Q}_{N,I})_{\ord Q_{N,I}-N}\in C(\partial^{ln+ndN})$. So, the component $(\tilde{Q}_{N,I})_{\ord Q_{N,I}-N-1}$ is the first component with $Sdeg_A>0$, i.e. it is the beginning of the tail for the operator $\tilde{Q}_{N,I}$. 
\end{rem}

{\it Step} 5. (body-tail decomposition) Put 
$$\tilde{d}:=SeqGCD_{I-1}(\Phi_{N,1}(P_n),\Phi_{N,1}(Q_n))=SeqGCD_{I-1}(P,Q)(d_2N+l_2)/d_2, \quad \tilde{p}:=ln+dnN.$$ 
Since $SeqGCD_{I-1}(P,Q)/d_2$ divides $GCD(l,d)$ (by theorem \ref{T:DC-pairs}), we can decompose $\tilde{p}$ as the product 
$$\tilde{p}=\tilde{d}\tilde{d}_1n,$$ 
where $GCD(l,d)=\tilde{d}_1SeqGCD_{I-1}(P,Q)/d_2$. 

We need to find a relation between operators $\tilde{Q}_{N,I}'={F}_{I}(\partial^{\tilde{p}},\tilde{Q}_{N,0} -(\tilde{Q}_{N,0})_{-\tilde{p}})$ and $\tilde{Q}_{N,I}={F}_{I}(\partial^{\tilde{p}},\tilde{Q}_{N,0})$.

 From previous step we know the shape of the operator $\tilde{Q}_{N,I}$:
$$
\tilde{Q}_{N,I}=\epsilon_I\partial^{\Ord (\tilde{Q}_{N,I})}+\tilde{q}\partial^{\Ord (\tilde{Q}_{N,I})-N}+
(\tilde{Q}_{N,I})_{\ord Q_{N,I}-N-1}+\ldots , \quad Sdeg_A(\tilde{q})=0.
$$

Let's introduce the following terminology. Each operator $\tilde{Q}_{N,i}$ can be decomposed as a sum of a "body" from the centralizer $C(\partial^{\tilde{p}})$ and a "tail", which is a sum of products of the "bodies" and the "tails" $(\tilde{Q}_{N,0})_{-\tilde{p}}$. Note that any non-zero homogeneous component of this "tail" has positive $Sdeg_A$. This can be shown by induction. For the operator $(\tilde{Q}_{N,0})$ the body is 
$(\tilde{Q}_{N,0} -(\tilde{Q}_{N,0})_{-\tilde{p}})$ and the tail is $(\tilde{Q}_{N,0})_{-\tilde{p}}$, and its form is determined by lemma \ref{L:Q_p}. Then 
$$
\tilde{Q}_{N,1}=(\tilde{Q}_{N,0})^{n_0}-\epsilon_0^{n_0}\partial^{\tilde{p}m_0}=
(\tilde{Q}_{N,0} -(\tilde{Q}_{N,0})_{-\tilde{p}})^{n_0}-\epsilon_0^{n_0}\partial^{\tilde{p}m_0} +\mbox{tail},
$$
where the body is $(\tilde{Q}_{N,0} -(\tilde{Q}_{N,0})_{-\tilde{p}})^{n_0}-\epsilon_0^{n_0}\partial^{\tilde{p}m_0}$ and the tail consists of sum of "monomials" of the shape 
$$
(\tilde{Q}_{N,0} -(\tilde{Q}_{N,0})_{-\tilde{p}})^{k_1}(\tilde{Q}_{N,0})_{-\tilde{p}}(\tilde{Q}_{N,0} -(\tilde{Q}_{N,0})_{-\tilde{p}})^{k_2}(\tilde{Q}_{N,0})_{-\tilde{p}}\ldots .
$$
From lemma \ref{L:Q_p} we know the shape of $(\tilde{Q}_{N,0})_{-\tilde{p}}$:  $(\tilde{Q}_{N,0})_{-\tilde{p}}=(-\Gamma_1/\tilde{p}+u)\int^{\tilde{p}}$, where $u\in C(\partial^{\tilde{p}})$. Then by commutation relations from lemma \ref{L:1} we can rewrite these "monomials" in the form 
$$
(\tilde{Q}_{N,0} -(\tilde{Q}_{N,0})_{-\tilde{p}})^{\sum k_i}(-\Gamma_1/\tilde{p}+u_1)\int^{\tilde{p}}\ldots (-\Gamma_1/\tilde{p}+u_r)\int^{\tilde{p}}, \quad u_i\in C(\partial^{\tilde{p}}).
$$
The product of terms containing $\Gamma_1$ has $Sdeg_A$ equal to $r$ (the highest term of this HCP is just the product of the highest terms of each factor), and therefore each non-zero homogeneous component of the expression above has $Sdeg_A$ equal to $r$. Moreover, all these non-zero homogeneous components arise from  
non-zero homogeneous components of $(\tilde{Q}_{N,0} -(\tilde{Q}_{N,0})_{-\tilde{p}})^{\sum k_i}$ -- the power of the body of the operator $(\tilde{Q}_{N,0})$. 

By induction, 
\begin{equation}
\label{E:body-tale}
\tilde{Q}_{N,i+1}=(body)_i^{n_i}-\epsilon_i^{n_i}\partial^{\tilde{p}m_i}+\mbox{tail},
\end{equation}
where the tail consist of a sum of "monomials" of $Sdeg_A>0$ (one "monomial" for each order), and the "monomial" of $Sdeg_A=r$ has the shape 
$$
(body)_i^{k_1}(body)_{i-1}^{l_1}\ldots (body)_i^{k_2} (body)_{i-1}^{l_2}\ldots (body)_{j}^{l_j}(-\Gamma_1/\tilde{p}+u_1)\int^{\tilde{p}}\ldots (-\Gamma_1/\tilde{p}+u_r)\int^{\tilde{p}}
$$
(we'll call \eqref{E:body-tale} as a "body-tail" decomposition).

{\it Step} 6. Returning to the operator $\tilde{Q}_{N,I}$, we see that the homogeneous component $(\tilde{Q}_{N,I})_{\ord Q_{N,I}-N-1}$ 
includes a summand from the tail. This  summand from the tail is the highest order homogeneous component of the "monomial" above of $Sdeg_A=1$, and it can be calculated: 

\begin{lemma}
\label{L:tail}
The tail component of the term $(\tilde{Q}_{N,I})_{\ord Q_{N,I}-N-1}$ is
$$
\epsilon_I(-dn+ln)(\Gamma_1/\tilde{p}-\tilde{u}+c_0)\partial^{\ord ({Q}_{N,I})-1-N},
$$
where 

i) $\tilde{u}\in C(\partial^{\tilde{p}})$ is a "$\tilde{d}$-symmetrization" of the corresponding term in the tail given by lemma \ref{L:Q_p}, which can be effectively written in the vector form:  
$$
\Phi(\tilde{u})=(a_0, \ldots , a_{\tilde{p}-1}), 
$$ 
where $a_i=a_{i+\tilde{d}}=a_{i+2\tilde{d}}=\ldots$, and for $i=0,\ldots ,\tilde{d}-1$ 
$$
a_i=(\frac{i}{\tilde{p}}+\frac{i+\tilde{d}}{\tilde{p}}+\ldots +\frac{\tilde{d}(n\tilde{d}_1-1)+i}{\tilde{p}})/(n\tilde{d}_1)=\frac{n\tilde{d}_1-1}{2 n\tilde{d}_1} + \frac{i}{\tilde{p}};
$$

ii) $c_0\in K$  is given by the formula 
$$
c_0=(\sum_{i=0}^{I-1} \ord (Q_i))\cdot \frac{(p-1)}{2p}
$$
(recall that we put all $n_i=p=nd$ in Step 3). 
\end{lemma}

\begin{proof}
The highest coefficient of the tail component we know from formula \eqref{E:tildeQ_N,I}. 
To find other coefficients let's use induction.  The tail component of the operator $\tilde{Q}_{N,0}$ in vector form is 
$(-\Gamma_1/\tilde{p} + (0, 1/\tilde{p}, \ldots , (\tilde{p}-1)/\tilde{p}))\tilde{D}^{-\tilde{p}}$ according to lemma \ref{L:Q_p}.  The tail component of the operator $\tilde{Q}_{N,1}$ in vector form is 
\begin{multline*}
 \sum_{i=0}^{p-1}\tilde{D}^{i\ord (Q_{N,0})}(-\Gamma_1/\tilde{p} +(0, 1/\tilde{p}, \ldots , (\tilde{p}-1)/\tilde{p})\tilde{D}^{-\tilde{p}+ (p-1-i)\ord (Q_{N,0})}= \\
 \sum_{i=0}^{p-1}(-(\Gamma_1+i\ord (Q_{N,0}))/\tilde{p} +\sigma^{-i\ord (Q_{N,0})}((0, 1/\tilde{p}, \ldots , (\tilde{p}-1)/\tilde{p})))=\\
p(-\Gamma_1/\tilde{p}-\frac{\ord (Q_{N,0})}{\tilde{p}}\cdot \frac{(p-1)}{2} +\tilde{u}_1)\tilde{D}^{-\tilde{p}+(p-1)\ord (Q_{N,0})},
\end{multline*}
where $\tilde{u}_1$ is the $GCD(\tilde{p}, \ord Q_{N,0})$-symmetrization of $(0, 1/\tilde{p}, \ldots , (\tilde{p}-1)/\tilde{p})$. By induction, the tail component of the operator $\tilde{Q}_{N,i}$ in vector form is
$$
\epsilon_1^{p-1}\cdots \epsilon_{i-1}^{p-1} p^i(-\Gamma_1/\tilde{p}-\frac{\sum_{j=0}^{i-1}\ord (Q_{N,j})}{\tilde{p}}\cdot \frac{(p-1)}{2} +\tilde{u}_i)\tilde{D}^{-\tilde{p}+(p-1)\sum_{j=0}^{i-1}\ord (Q_{N,j})},
$$
where $\tilde{u}_i$ is the $SeqGCD_{i-1}(\Phi_{N,1}(P_n),\Phi_{N,1}(Q_n))$-symmetrization of $(0, 1/\tilde{p}, \ldots , (\tilde{p}-1)/\tilde{p})$. Then the tail component of the operator $\tilde{Q}_{N,i+1}$ in vector form is 
\begin{multline*}
\epsilon_1^{p-1}\cdots \epsilon_{i-1}^{p-1}\epsilon_{i}^{p-1} p^i \sum_{k=0}^{p-1}\tilde{D}^{k\ord (Q_{N,i})}(-\Gamma_1/\tilde{p} -\frac{\sum_{j=0}^{i-1}\ord (Q_{N,j})}{\tilde{p}}\cdot \frac{(p-1)}{2}\\+\tilde{u}_i)\tilde{D}^{-\tilde{p}+ (p-1)\sum_{j=0}^{i-1}\ord (Q_{N,j}) + (p-1-k)\ord (Q_{N,i})}= \\
 \epsilon_1^{p-1}\cdots \epsilon_{i-1}^{p-1}\epsilon_{i}^{p-1} p^{i+1} (-\Gamma_1/\tilde{p}-\frac{\sum_{j=0}^{i}\ord (Q_{N,j})}{\tilde{p}}\cdot \frac{(p-1)}{2} +\tilde{u}_{i+1})\tilde{D}^{-\tilde{p}+ (p-1)\sum_{j=0}^{i}\ord (Q_{N,j})},
\end{multline*}
whence we get the claim of lemma with 
$$
 \epsilon_1^{p-1}\cdots \epsilon_{I-1}^{p-1} p^{I}=\epsilon_I(dn-ln) 
$$
by formula \eqref{E:tildeQ_N,I}, and the formula for $c_0$ holds because $\ord (Q_{N,j})=(d_2N+l_2)\ord (Q_j)/d_2$.
\end{proof}

So, the "body-tail" decomposition of the operator $\tilde{Q}_{N,I}$ looks like 
\begin{multline*}
\tilde{Q}_{N,I}=(\epsilon_I\partial^{\Ord (\tilde{Q}_{N,I})}+\tilde{q}\partial^{\Ord (\tilde{Q}_{N,I})-N}+
\epsilon_I((-dn+ln)(\tilde{u}-c_0) +(\sigma^{-\ord Q_{N,I}}(u)-u)+\\
A\frac{N^2}{Nd+l}+B\frac{N}{Nd+l}+C\frac{1}{Nd+l} +b)\partial^{\Ord (\tilde{Q}_{N,I})-1-N}+\ldots )+ \\
\epsilon_I(-dn+ln)(\Gamma_1/\tilde{p}-\tilde{u}+c_0)\partial^{\ord ({Q}_{N,I})-1-N}+\ldots 
\end{multline*}
Hence, first terms of the operator $\tilde{Q}_{N,I}'$ are 
\begin{multline*}
\tilde{Q}_{N,I}'= \epsilon_I\partial^{\Ord (\tilde{Q}_{N,I})}+\tilde{q}\partial^{\Ord (\tilde{Q}_{N,I})-N}+
\epsilon_I((-dn+ln)(\tilde{u}-c_0) +(\sigma^{-\ord Q_{N,I}}(u)-u)+\\
A\frac{N^2}{Nd+l}+B\frac{N}{Nd+l}+C\frac{1}{Nd+l} +b)\partial^{\Ord (\tilde{Q}_{N,I})-1-N}+\ldots 
\end{multline*}

{\it Step} 7. Now we can calculate the distinguished coefficient $c_{N,I;r(I)}$. To find it, we need to find the result of conjugation of $\hat{\Phi}(\tilde{Q}_{N,I}')$ by the auxiliary operator $\tilde{S}$ from lemma \ref{L:invariant_conjugation}. According to the proof of this lemma, $\tilde{S}$ can be represented as a product of elementary operators $\tilde{S}_i=1+s_i\tilde{D}^{-i}$. In our case obviously 
$i\ge N$. Note that conjugation by $\bar{S}_{N}$ does not affect the homogeneous component of the operator $\hat{\Phi}(\tilde{Q}_{N,I}')$ of order $\Ord \tilde{Q}_{N,I}-N-1$, so we need to find only to find the result of conjugation by the auxiliary operator $\bar{S}_{N+1}$. This is just the averaging of components of its vector components. 

The averaging of the vector components of the term $(\sigma^{-\ord Q_{N,I}}(u)-u)$ is obviously zero. So, we need only to calculate the averaging of $\tilde{u}$ from lemma \ref{L:tail}. It is equal to
$$
\frac{1}{\tilde{d}}(\sum_{i=0}^{\tilde{d}-1}a_i)=\frac{n\tilde{d}_1-1}{2n\tilde{d}_1}+\frac{(\tilde{d}-1)}{2\tilde{p}}
$$
Combining all terms together, we get
$$
c_{N,I;r(I)}=\epsilon_I((-dn+ln)(\frac{1}{2}-\frac{1}{2n (Nd+l)} -(\sum_{i=0}^{I-1} \ord (Q_i))\cdot \frac{(p-1)}{2p})+A\frac{N^2}{Nd+l}+B\frac{N}{Nd+l}+C\frac{1}{Nd+l} +b).
$$
It remains to note that this expression non-trivially depends on $N$ (as $A\neq 0$) -- a contradiction with lemma \ref{L:distinguished coefficients}.

So, there are no DC-pairs, and $\End (A_1)\backslash 0=\Aut (A_1)$.

\vspace{0.5cm}

\noindent A. Zheglov,  Lomonosov Moscow State  University, Faculty
of Mechanics and Mathematics, Department of differential geometry
and applications, Leninskie gory, GSP, Moscow, \nopagebreak 119899,
Russia
\\ \noindent e-mail
 $alexander.zheglov@math.msu.ru$, $abzv24@mail.ru$

\end{document}